\documentclass[a4paper,11pt]{article}

\usepackage{amssymb, amstext, amsmath, amsthm, latexsym, textcomp}

\usepackage[utf8]{inputenc} 
\usepackage[english]{babel}

\newtheorem{defi}{Definition}[section]

\newtheorem{lemma}[defi]{Lemma}
\newtheorem{prop}[defi]{Proposition}
\newtheorem{rem}[defi]{Remark}

\newtheorem{co}[defi]{Corollary}

\DeclareMathOperator{\ad}{ad}
\DeclareMathOperator{\st}{st}
\DeclareMathOperator{\Ad}{Ad}

\DeclareMathOperator{\prr}{Rad}

\DeclareMathOperator{\Ch}{char}

\DeclareMathOperator{\Spec}{Spec}
\DeclareMathOperator{\pideg}{p.i.deg}
\DeclareMathOperator{\Id}{Id}

\DeclareMathOperator{\Ker}{Ker}
\DeclareMathOperator{\Hom}{Hom}
\DeclareMathOperator{\End}{End}

\DeclareMathOperator{\PDer}{PDer}
\DeclareMathOperator{\Der}{Der}

\DeclareMathOperator{\Sym}{Sym}

\DeclareMathOperator{\Jr}{J}

\DeclareMathOperator{\CM}{CM}
\DeclareMathOperator{\Ann}{Ann}

\begin{document}

\date{}

\title{Orthogonal completion of algebraic systems II} 

\author{A.~Yu.~Golubkov} 

\maketitle

\begin{abstract}
In this paper, ideas developed by K.~I.~Beidar and A.~V.~Mikhalev are implemented in the description 
of orthogonally complete non-degenerate alternative, linear Jordan, special 
Lie and Mal'tsev algebras based on well-known classification of their 
strongly prime quotient algebras. 
\end{abstract}

\section{Introduction}

Continuing the line of \cite{GolO}, we will consider another approach to 
constructing the orthogonal completion of semiprime algebras (of signature 
$\Omega$ from \cite{Raz}), which appeared in \cite{BD} in the discussion of 
almost classical localizations and appears in \cite{Mont, Mont1}. 
Its application together with the ideas of \cite{BD, BMS, BMO} allows us to 
transform the classification of strongly prime algebras into a description 
of orthogonally complete non-degenerate algebras for the classes of 
alternative, linear Jordan, special Lie and Mal'tsev algebras in the spirit 
of a generalization of the Posner theorem (Fisher --- Martindale --- 
Markov --- Armendariz --- Steinberg theorem \cite{Fish, Mart1, Mar3, ArmS}) 
for semiprime associative $PI$--algebras and its Jordan version in 
\cite{Mont}. 

Everywhere below, in the absence of additional conditions, $F$ is any 
associative commutative ring with 1, all $F$--modules are unitary, 
both left and right with identical action of $F$, the action of their 
endomorphisms will be written on the right, all $F$--algebras are 
linear, classes of $F$--algebras contain the zero algebra and are closed 
under taking isomorphic copies. If $F$ is a field, then 
$\overline{F}$ is the algebraic closure of $F$. If we talk about the 
algebra (triple system) $R$ without mentioning $F$, then by default $R$ is an 
$F$--algebra ($F$--system). As usual, ${\Hom(M, N)_A}$ 
(${\Hom_A(M, N)}$) is the Abelian group ($F$--module) of homomorphisms of the
right (left) modules $M$ and $N$ over the associative ring ($F$--algebra) $A$ 
from $M$ to $N$, $\End(M)_A$ ($\End_A(M)$) is the 
ring ($F$--algebra) of endomorphisms of $M_A$ (${}_A M$), 
${\Ann_A L=\{x\in A\mid L x=\{0\}\}}$ is the annihilator in $A$ of the set 
${L\subseteq M_A}$, $\Id_M$ is the identical isomorphism of the module $M$, an 
essential submodule of $M$ is a submodule of $M$ that has a non-zero intersection 
with any non-zero submodule of $M$, ${R^1=F\oplus R}$ is the 
$F$--algebra obtained from $R$ by standard addition of 1 (over which $F$ 
$R$ and $R^1$ are considered follows from the context). Algebras (triple systems) 
without $k$--torsion are $k$--torsion--free algebras (systems) in additive group, 
${k\geq 1}$ (${k x\ne 0}$ for all ${x\ne 0}$). 

For any $F$--algebra $R$ and set ${B\subseteq R}$, $F B$, $\langle B\rangle$ 
and $(B)_R$ are the $F$--submodule, subalgebra and ideal of $R$ generated by 
$B$, $\End_F(R)$ is the $F$--algebra of endomorphisms of the $F$--module $R$, 
${\mathcal{E}(R)=\{I\lhd R\mid I\cap J\ne \{0\}\ \forall \{0\}\ne J\lhd R\}}$ 
is the set of essential ideals of $R$, $R^{(-)}$ ($R^{(+)}$ for $F$ with $1/2$)
is the algebra obtained from $R$ by replacing multiplication with 
${[x, y]=x y -y x}$ (${x\cdot y=1/2 (x y+y x)}$), ${x, y\in R}$, 
$\Der(R)$ ($\Der_F(R)$ if we need to indicate $F$) is the Lie algebra of 
$F$--derivations of $R$ (a subalgebra of the Lie algebra 
$\End_F(R)^{(-)}$),\linebreak 
${M^R(B)=\langle t_x\mid t=l, r,\ x\in B\rangle}$, 
${M^R(B)'=F \Id_R+M^R(B)}$, ${M(R)=M^R(R)}$ ($M(R)'$) is the 
algebra of multiplications (with 1) of $R$, 
${\Ann_t B=\{x\in R\mid B t_x=\{0\}\}}$, ${t=l, r}$,
$\Ann B=\Ann_l B\cap \Ann_r B$ (${\Ann_t B=\Ann_t(R, B)}$, 
${\Ann B=\Ann(R, B)}$ if it is necessary to indicate $R$), 
$N(R)$, $K(R)$ and $Z(R)$ is the associative center, commutative center and 
center of $R$, 
\begin{gather*}
N(R)\ =\ \{x\in R\mid (x, R, R)=(R, x, R)=(R, R, x)=\{0\}\}\,,
\\
K(R)\ =\ \{x\in R\mid [x, R]=\{0\}\}\,, 
Z(R)\ =\ N(R)\cap K(R)\ =\ \{x\in R\mid l_x=r_x\in Z(M(R))\}\,,
\end{gather*}
where ${l_x: y\longmapsto x y}$ and  
${r_x: y\longmapsto y x, x, y\in R}$, are the operators of left and right 
multiplication by $x$, ${(x, y, z)=(x y) z-x (y z)}$, ${x, y, z\in R}$, 
${\Ann R\subseteq Z(R)}$. In case $R$ is anticommu\-tative, 
${\Ann B=\Ann_l B=\Ann_r B}$, ${B\subseteq R}$,  
${K(R)=\Ann 2 R=\{x\in R\mid 2 x\in \Ann R\}}$ and for ${R=2 R}$ or 
$R$ without $2$--torsion one has ${K(R)=Z(R)=\Ann R}$. If $R$ is a 
Lie algebra, then\linebreak ${\ad_x=r_x=-l_x}$, ${x\in R}$, the ideal of 
inner derivations ${\ad(R)=\{\ad_x\mid x\in R\}\lhd \Der(R)}$ is the 
image of $R$ under the action of the canonical epimorphism 
${\ad: x\longmapsto \ad_x, x\in R}$, $\Ker \ad=\Ann R$, where the choice 
${\ad=r}$ is related with the right action of $\End_F(R)$ on $R$ in the 
paper. 

An algebra $R$ is \emph{prime} (\emph{semiprime}) if ${I J\ne \{0\}}$ 
(${I^2\ne \{0\}}$) for all ${\{0\}\ne I, J\lhd R}$. The primeness of $R$ is 
equivalent to its semiprimeness and ${I\cap J\ne \{0\}}$ for all 
${\{0\}\ne I, J\lhd R}$. An ideal ${P\lhd R}$ is a \emph{prime ideal of $R$} 
if $R/P$ is prime, $\Spec(R)$ is the set of all prime ${P\lhd R}$, 
${\prr(R)=\bigcap\limits_{P\in \Spec(R)} P}$ is the \emph{prime radical of $R$}, 
${\prr(R)=\prr({}_{\mathbb{Z}} R)}$ is the smallest element of 
${\{I\lhd R\mid \prr(R/I)=\{0\}\}=
\{I\lhd {}_{\mathbb{Z}} R\mid \prr({}_{\mathbb{Z}} R/I)=\{0\}\}}$, where 
${}_{\mathbb{Z}} R$ is $R$ as a ring.

For an alternative, Jordan or Lie (Mal'tsev without $2$--torsion) algebra 
$R$, \emph{non-degeneracy} is the absence of ${0\ne x\in R}$ with the 
corresponding condition: ${l_x r_x=0}$, ${U_x=0}$, where 
${U_x=2 r_x^2-r_{x^2}}$ for $F$ with $1/2$, or 
${\ad_x (F \Id_R+\ad(R)) \ad_x=\{0\}}$ (${r_x^2=0}$), \emph{strong primeness} 
is primeness and non-degeneracy. Non-degenerate algebras from the classes 
of alternative and Jordan algebras form their semisimple subclasses, the 
upper radicals determined by them are the Slin'ko --- McCrimmon and 
McCrimmon radicals $Mc$ \cite[Chs.~8,~14]{ZShS}, 
\cite[the~observations~after~Lemma~2.6]{GolO}, 
\cite[Addition~1]{GolA} (the analogue of $Mc$ on the classes 
of Lie and Mal'tsev algebras is the Kostrikin radical $K$, although its 
radicality in the Kurosh --- Amitsur sense has so far been established for 
the field $F$, ${\Ch F=0}$, and with lesser restrictions on $F$ for individual 
subclasses), the strongly prime spectrum 
${\Spec_{Mc}(R)=\{P\in \Spec(R)\mid Mc(R/P)=\{0\}\}}$ is associated with 
$Mc$.

An element with associative powers $x$ of an algebra $R$ is \emph{integral} if 
there is a polynomial 
${{}_x f(t)=t^n+{}_x f_{n-1} t^{n-1}+\ldots+{}_x f_1 t\in F[t]}$, 
${{}_x f(x)=0}$. For an anticommutative $R$, an element ${x\in R}$ 
is \emph{strongly algebraic} if ${r_x=-l_x}$ is integral in $M(R)$, and 
\emph{algebraic} if for any ${y\in R}$ there exists 
${{}_{y, x} f(t)=t^{n_{y, x}}+{}_{y, x} f_{n_{y, x}-1} t^{n_{y, x}-1}+\ldots
+{}_{y, x} f_1 t\in F[t], y {}_{y, x} f(r_x)=0}$. 
Mal'tsev (Lie) algebras consisting of strongly algebraic elements are 
\emph{algebras with an algebraic regular} (\emph{adjoint}) 
\emph{representation}, and Mal'tsev (Jordan, alternative) algebras 
consisting of algebraic (integral) elements are \emph{algebraic algebras} 
\cite[the~conditions~of~algebraicity~of~algebras]{Gol3}.

All constructions of the paper devoted to central closures of semiprime 
algebras relate to non-zero algebras. In \cite{GolO} we collect the 
necessary information about the Martindale centroid $\CM(R)$, the extended 
centroid ${}_R C$, and the central closures $P(R)$ and ${}_R Q$ of the 
semiprime algebra $R$,
\[
{}_R C\ \cong\ \CM(R)\ =\ \End(P(R))_{M(R)'}\,,\quad
{}_R Q\cong P(R)\ =\ \CM(R) R\ =\ \End(I(R))_{M(R)'} R\,,
\]
$I(R)$ is the injective hull of the module $R_{M(R)'}$ chosen for 
constructing ${P(R)=P(P(R))}$, $\CM(R)$ is a commutative regular 
$F$--algebra (a field for the prime $R$), $P(R)$ is the quasi-injective 
hull of $R_{M(R)'}$ in $I(R)$ and a semiprime (prime for the prime $R$) 
$\CM(R)$--algebra whose operations are extended from $R$ by 
$\CM(R)$--linearity. 

As in \cite{GolO}, 
${B(R)=\{\alpha\in \CM(R)\mid \alpha=\alpha^2\}}$ is an orthogonally 
complete Boolean ring with multiplication of $\CM(R)$ and addition 
${\alpha\oplus \beta=(\alpha-\beta)^2=
\alpha+\beta-2 \alpha \beta}$, ${\alpha, \beta\in B(R)}$, 
$E\subseteq B(R)$ is \emph{dense} in $B(R)$ if ${\Ann_{B(R)} E=\{0\}}$ 
(${\Ann_{\CM(R)} E=\{0\}}$), and \emph{orthogonal} if ${\alpha \beta=0}$ 
for all ${\alpha, \beta\in E}$, ${\alpha\ne \beta}$. We fix $I(P(R))$ 
and denote by $\mathcal{U}(R)$ the set of all ultrafilters in 
${B(R)=B(P(R))}$. A set ${S\subseteq I(P(R))}$ is \emph{orthogonally 
complete} if for any dense orthogonal 
${\{\xi_a\}_{a\in I}\subseteq B(R)}$, ${\{x_a\}_{a\in I}\subseteq S}$ 
there exists ${x\in S}$, ${\xi_a x=\xi_a x_a}$ for all ${a\in I}$. Such a 
uniquely defined $x$ will be denoted by 
${{\sum\limits_{a\in I}}^{\perp} \xi_a x_a}$. An \emph{orthogonal completion} 
$O(S)$ of a set ${S\subseteq I(P(R))}$ is the intersection of orthogonally 
complete ${S'\subseteq I(P(R))}$, ${S\subseteq S'}$, 
${I(P(R))=O(I(P(R)))}$, ${O(S)=O(O(S))}$ is the set of all 
${{\sum\limits_{a\in I}}^{\perp} \xi_a x_a\in I(P(R))}$ for a dense 
orthogonal ${\{\xi_a\}_{a\in I}\subseteq B(R)}$, 
${\{x_a\}_{a\in I}\subseteq S}$, $O(P(R))$ is a semiprime $\CM(R)$--algebra 
with the operations  
\[
{\sum\limits_{a\in I}}^{\perp} \xi_a x_a\star 
{\sum\limits_{b\in J}}^{\perp} \chi_b y_b\ =\ 
{\sum\limits_{a\in I,\ b\in J}}^{\perp} \xi_a \chi_b (x_a\star y_b)\quad 
(\star=+, \cdot)
\] 
for dense orthogonal ${\{\xi_a\}_{a\in I}, \{\chi_b\}_{b\in J}\subseteq B(R)}$, 
${\{x_a\}_{a\in I}, \{y_b\}_{b\in J}\subseteq P(R)}$, 
$O(R)$ is its semi\-prime $F$--subalgebra, $O(P(R))$ is a quasi-injective 
(fully invariant) $M(P(R))'$--submodule of $I(P(R))$, 
\[
\End(O(P(R)))_{M(P(R))'}\ =\ \End(O(P(R)))_{M(O(P(R)))'}\ =\ 
\CM(R) \Id_{O(P(R))}
\] 
\cite[Proposition~2.4]{GolO}. Since $R_{M(R)'}$ is essential in $I(R)$, for 
$R$ without $k$--torsion, ${k\geq 1}$, $I(R)$, $P(R)$, $I(P(R))$, $O(P(R))$ 
and $O(R)$ have no $k$--torsion. 

In the text, the short notation for orthogonal completeness of $R$ is 
${R=O(R)}$, 
${S_B=S/P S}$, ${P S=\{\alpha x\mid \alpha\in P,\ x\in S\}\lhd S}$ for 
${S=R, \CM(R)}$, ${R=O(R)}$, ${B\in \mathcal{U}(R)}$ and the maximal 
ideal ${P=B(R)\setminus B\lhd B(R)}$, ${\nu(T)=\inf T}$ is the exact 
lower bound of ${T\subseteq B(R)}$, $Q(A)$ and $Q^s_m(A)$ are the complete 
(maximal) and symmetric Martindale rings of quotients of the left exact 
(${\Ann_l A=\{0\}}$) associative ring $A$ \cite[Lemma~1.16]{BMO}, \cite[Introductions]{GolO, GolA}, many equalities are understood up to 
isomorphism (in particular, this applies to the identification of 
$\CM(R)$ and $Z(R)$ for ${R=P(R)}$ with 1).

From \cite[Vol.~2,~Theorem~19.14~A,~p.~111]{Fe}, \cite[Proposition~2.4]{GolO} it immediately follows 

\begin{rem} 
If $P(R)$ is a finitely generated $\CM(R)$--module, then 
\[
P(R)\ =\ I(R)\ =\ I(P(R))\ =\ O(P(R))\,.
\]
\end{rem}

\begin{rem}
If ${R=R_1\oplus \ldots\oplus R_n}$ is a direct sum of semiprime 
${R_i=P(R_i)}$, ${I=I M(R_i)}$ for any ${I\lhd R_i}$, 
${i=1, \ldots, n}$, ${n\geq 1}$, then ${\CM(R)\cong \CM(R_1)\oplus \ldots\oplus \CM(R_n)}$, 
${R=P(R)}$ and ${O(R)\cong O(R_1)\oplus \ldots \oplus O(R_n)}$. 
\end{rem}

\begin{proof}
We will identify $\{R_i\}$ with the ideals of $R$. If 
${\phi\in \Hom(J, R)_{M(R)'}}$, ${J\lhd R}$ and $\pi_i$ is the canonical 
projection of $R$ onto $R_i$, ${\pi_i: x_1+\ldots+x_n\longmapsto x_i}$, 
${x_i\in R_i}$, ${i=1, \ldots, n}$, then   
${J_i=\pi_i J=J_i M^R(R_i)=J M^R(R_i)=J\cap R_i\lhd R_i}$, 
${\phi J_i=(\phi J) M^R(R_i)\subseteq R_i}$,  
$\phi|_{J_i}=\phi_i \Id_{J_i}\in \Hom(J_i, R_i)_{M^R(R_i)'}=
\Hom(J_i, R_i)_{M(R_i)'}$ for some ${\phi_i\in \CM(R_i)}$, 
${J=J_1\oplus \ldots \oplus J_n}$ and 
$\phi=\Bigl(\sum\limits_{i=1}^n \phi_i \pi_i\Bigr) \Id_J$, 
${\sum\limits_{i=1}^n \phi_i \pi_i\in \CM(R)}$, 
${\sum\limits_{i=1}^n \phi_i \pi_i: x_1+\ldots+x_n\longmapsto 
\phi_1 x_1+\ldots+\phi_n x_n}$, ${x_i\in R_i}$. 

Thus, ${R=P(R)}$, ${\CM(R)=\CM(R_1)\oplus \ldots\oplus \CM(R_n)}$, 
${B(R)=B(R_1)\oplus\ldots \oplus B(R_n)}$,
to any dense orthogonal ${\{\xi^{(i)}_a\}_{a\in A_i}\subseteq B(R_i)}$, 
${i=1, \ldots, n}$, there corresponds a dense orthogonal 
${\Bigl\{\xi_{(a_1, \ldots, a_n)}=
\sum\limits_{i=1}^n \xi^{(i)}_{a_i} \pi_i\Bigr\}_
{(a_1, \ldots, a_n)\in A}\subseteq B(R)}$, ${A=A_1\times \cdots\times A_n}$,
and to each dense orthogonal ${\{\chi_b\}_{b\in B}\subseteq B(R)}$ there 
correspond dense orthogonal ${\{\chi_b|_{R_i}\}_{b\in B}\subseteq B(R_i)}$, 
${i=1, \ldots, n}$, the mapping 
\begin{multline*}
\biggl({\sum_{a_1\in A_1}}^{\perp} \xi^{(1)}_{a_1} r^{(1)}_{a_1}, \ldots, 
{\sum_{a_n\in A_n}}^{\perp} \xi^{(n)}_{a_n} r^{(n)}_{a_n}\biggr)\ 
\longmapsto
\\
{\sum_{(a_1, \ldots, a_n)\in A}}^{\perp} 
\xi_{(a_1, \ldots, a_n)} (r^{(1)}_{a_1}+\ldots+r^{(n)}_{a_n})\ =\ 
\sum_{i=1}^n
{\sum_{(a_1, \ldots, a_n)\in A}}^{\perp} \xi_{(a_1, \ldots, a_n)} 
r^{(i)}_{a_i}\quad (\{r^{(i)}_{a_i}\}_{a_i\in A_i}\subseteq R_i)
\end{multline*}
and its inverse 
\begin{multline*}
{\sum_{b\in B}}^{\perp} \chi_b (r^{(1)}_b+\ldots+r^{(n)}_b)\ =\ 
\sum_{i=1}^n {\sum_{b\in B}}^{\perp} \chi_b r^{(i)}_b\ \longmapsto
\\
\biggl({\sum_{b\in B}}^{\perp} \chi_b r^{(1)}_b, \ldots, 
{\sum_{b\in B}}^{\perp} \chi_b r^{(n)}_b\biggr)\quad 
(\{r^{(i)}_b\}_{b\in B}\subseteq R_i)
\end{multline*}
are $\CM(R)$--isomorphisms of ${O(R_1)\oplus \ldots\oplus O(R_n)}$ 
and ${O(R)=O'(R_1)\oplus \ldots \oplus O'(R_n)}$ ($O(R_i)$ in $I(R_i)$, 
$O'(R_i)$ in ${O(R)\subseteq I(R)}$, ${O(R_i)\cong O'(R_i)}$ over 
${\CM(R_i)=\{\alpha|_{J_i}\mid \alpha\in \CM(R)\}}$).
\end{proof}

\begin{rem}
If $R$ semiprime, ${Z(R)=Q(Z(R))}$, ${I\cap Z(R)\ne \{0\}}$ for all 
${\{0\}\ne I\lhd R}$, then ${R=P(R)}$, ${\CM(R)=Z(R) \Id_R\cong Z(R)}$.
\end{rem}

\begin{proof}
From ${1\in Z(R)=Q(Z(R))=P(Z(R))\cong \CM(Z(R))}$ and ${\Ann_t Z(R)=\{0\}}$, 
${t=l, r}$ (see observations after Lemma 2.6) it follows that ${1\in R}$. 
If ${\phi\in \Hom(I, R)_{M(R)}}$, ${\{0\}\ne I\lhd R}$, then  
${\phi|_{I\cap Z(R)}\in \Hom(I\cap Z(R), Z(R))_{Z(R)}}$, 
${\phi|_{I\cap Z(R)}=\alpha \Id_{I\cap Z(R)}}$ for suitable 
${\alpha\in Z(R)}$, $J=(\phi-\alpha \Id_I)((I^2)_R)\subseteq I$, 
\[
\{0\}\ =\ (I\cap Z(R)) J\ =\ (J\cap Z(R))^2\ =\ J\cap Z(R)\ =\ J\ =\ 
((\phi-\alpha \Id_I) I)^2\ =\ (\phi -\alpha \Id_I) I\,,
\]
${\phi=\alpha \Id_I}$. Therefore ${R=P(R)}$, 
${\CM(R)=Z(R) \Id_R=\{l_z=r_z\mid z\in Z(R)\}\cong Z(R)}$. 
\end{proof}                                                           
     
From Remarks 1.1, 1.2, 1.3 it follows that ${R=P(R)=O(R)}$, 
${\CM(R)=Z(R) \Id_R\cong Z(R)}$ if $R$ is a finite direct sum of algebras 
of the form $M_n(F)$ and ${\mathcal{O}_F(\mu, \beta, \gamma)}$ over 
${F=Q(F)}$ (not necessarily over one $F$), where 
$M_n(F)$ is the algebra of ${n\times n}$ matrices over $F$, ${n\geq 1}$, 
${\mathcal{O}_F(\mu, \beta, \gamma)}$ is the Cayley --- Dickson algebra over 
$F$ for invertible ${\mu, \beta, \gamma\in F}$. It is sufficient\linebreak to 
note that ${F=Q(F)=P(F)=I(F)\cong \CM(F)}$ is regular, $F$, $M_n(F)$, 
${\mathcal{O}_F(\mu, \beta, \gamma)}$ are subdirect products of fields $F/P$, 
central simple $F/P$--algebras 
\[
M_n(F)/P M_n(F)\ \cong\ M_n(F/P)\,,\ 
\mathcal{O}_F(\mu, \beta, \gamma)/P \mathcal{O}_F(\mu, \beta, \gamma)\ \cong\ 
\mathcal{O}_{F/P}(\mu+P, \beta+P, \gamma+P)\,,
\] 
${F\ne P\in \Spec(F)}$, $M_n(F)$, ${\mathcal{O}_F(\mu, \beta, \gamma)}$ are 
non-degenerate and semiprimitive, $Z(A)=F 1\cong F$ for 
${1\in A}$, ${I\cap Z(A)\ne \{0\}}$ for all 
${\{0\}\ne I\lhd A=M_n(F), \mathcal{O}_F(\mu, \beta, \gamma)}$ 
\cite[Theorems~3,~7,~p.\linebreak 212,~225]{ZShS} (${\mathcal{O}_F(\mu, \beta, \gamma)}$ 
is a free $F$--module with basis ${\{e_0=1, e_1, \ldots, e_7\}}$ and 
structure constants in it from \cite[Theorem~3.1.11]{BD}, 
${z=f_1 e_1+\ldots+f_7 e_7\in Z(\mathcal{O}_F(\mu, \beta, \gamma))}$, 
${f_i\in F}$, is equivalent to ${z\in P \mathcal{O}_F(\mu, \beta, \gamma)}$ for all 
${P\in \Spec(F)}$, ${z=0}$).

The paper also deals with $\Gamma$--graded $F$--algebras 
${(R, \{R_{\gamma}\}_{\gamma\in \Gamma})}$, where $\Gamma$ is an 
Abelian group in additive notation, 
$\{R_{\gamma}\}_{\gamma\in \Gamma}$ are $F$--modules 
${R=\bigoplus\limits_{\gamma\in \Gamma} R_{\gamma}}$, 
${R_{\alpha} R_{\beta}\subseteq R_{\alpha+\beta}}$ for all  
${\alpha, \beta\in \Gamma}$. Recall that $I$ is a \emph{homogeneous ideal} 
of $R$, ${I\lhd_{gr} R}$, if  
${I=\bigoplus\limits_{\gamma\in \Gamma} (I\cap R_{\gamma})}$, $R$ is 
\emph{homogeneously simple} if ${R^2\ne \{0\}}$ and in $R$ there are no 
homogeneous ideals different from $\{0\}$ and $R$, and is 
\emph{homogeneously prime} (\emph{semiprime}) if 
${I J\ne \{0\}}$ (${I^2\ne \{0\}}$) for all ${\{0\}\ne I, J\lhd_{gr} R}$, 
${\mathcal{E}_{gr}(R)=\{I\lhd_{gr} R\mid \{0\}\ne I\cap J\ \forall 
\{0\}\ne J\lhd_{gr} R\}}$. The graded central closure $P_{gr}(R)$ and the 
Martindale centroid $\CM_{gr}(R)$ of a homogeneously semiprime $R$ are 
described in \cite{Gol6}, where 
${P_{gr}(R)=\End(I_{gr}(R))_{M(R)'-gr} R}$ and $I_{gr}(R)$ are the graded 
quasi-injective and injective hulls of the $\Gamma$--graded module $R$ 
over the $\Gamma$--graded algebra 
${(M(R)', \{M_{\gamma}\}_{\gamma\in \Gamma})}$ with ${M_0=F \Id_R+M(R)_0}$, 
${M_{\gamma}=M(R)_{\gamma}}$ for ${0\ne \gamma\in \Gamma}$,  
\[
M(R)_{\alpha}\ =\ \Biggl\{\sum_{k=1}^m t_{k 1}\cdots t_{kn_k}\Biggl| 
t_{ki}\in \{l_{x_{ki}}, r_{x_{ki}}\},\ x_{ki}\in R_{\gamma_{ki}},\ 
\sum_{i=1}^{n_k} \gamma_{ki}=\alpha,\ m, n_k\geq 1\Biggr\}\quad 
(\alpha\in \Gamma)\,,
\]
$\End(I_{gr}(R))_{M(R)'-gr}$ and ${\CM_{gr}(R)=\End(P_{gr}(R))_{M(R)'-gr}}$
are the algebras of homogeneous $M(R)'$--endomorphisms of $I_{gr}(R)$ and 
$P_{gr}(R)$, $\CM_{gr}(R)$ is a commutative regular $F$--algebra (a field for 
the homogeneously prime $R$), 
${(P_{gr}(R)=\CM_{gr}(R) R, \{\CM_{gr}(R) R_{\gamma}\}_{\gamma\in \Gamma})}$ 
is a $\Gamma$--graded homogeneously semiprime $\CM_{gr}(R)$--algebra 
(homogeneously prime for the homogeneously prime $R$). Similarly to the 
ungraded case, we define an orthogonally complete Boolean ring 
${B_{gr}(R)=\{\alpha\in \CM_{gr}(R)\mid \alpha=\alpha^2\}}$, the orthogonal 
completeness of ${S\subseteq I_{gr}(P_{gr}(R))}$ (the notation is  
${S=O_{gr}(S)}$) as the presence of 
${{\sum\limits_{a\in I}}^{\perp} \xi_a x_a\in S}$ for any dense 
orthogonal $\{\xi_a\}_{a\in I}\subseteq B_{gr}(R)$, 
${\{x_a\}_{a\in I}\subseteq S}$, the orthogonal completion $O_{gr}(S)$ with 
the correction that its existence is established for $R$ with finite grading, 
${|\{\gamma\in \Gamma\mid R_{\gamma}\ne \{0\}\}|< \infty}$ 
\cite[the~beginning~of~Sec.~3,\linebreak the~graded~version~of~Proposition~2.4]{GolO}. 

In the constructions of the paper $F\langle X\rangle$ and 
${F_*\langle X\rangle, *=Ass, Alt, Jor, Lie, Mal}$, are free non-associative 
and free associative, alternative, linear Jordan, Lie, Mal'tsev 
$F$--algebras with a countable set of free generators 
${X=\{x_i\}_{i=1}^{\infty}}$. An element ${f\in F_{Ass}\langle X\rangle}$ is 
called \emph{proper} if $F$ coincides with its ideal generated by the 
coefficients of the irreducible record of $f$. An associative $F$--algebra $A$ 
is a \emph{$PI$--algebra} if the identity ${f=0}$ holds on $A$ for the proper 
${f\in F_{Ass}\langle X\rangle}$. Any $PI$--algebra is a $PI$--ring with the
identity ${\st_n^k=0}$ for some ${n, k\geq 1}$, where\linebreak 
\[
\st_n\ =\ \st_n(x_1, \ldots, x_n)\ =\ \sum_{\sigma\in \mathfrak{S}_n} 
(-1)^{\sigma} x_{\sigma(1)}\cdots x_{\sigma(n)}
\]
is the standard polynomial of degree $n$, $\mathfrak{S}_n$ is the symmetric 
group of degree $n$, $(-1)^{\sigma}$ is the sign of ${\sigma\in \mathfrak{S}_n}$. 
An alternative, linear Jordan over $F$ with $1/2$, Lie and Mal'tsev 
$PI$--algebra can be defined as an algebra from the corresponding variety with 
the identity ${f=0}$ for ${f\in F_*\langle X\rangle, 
*=Alt, Jor, Lie, Mal}$, with a proper image under the action of the canonical homomorphisms 
${F_{Alt}\langle X\rangle\longrightarrow F_{Ass}\langle X\rangle}$, 
${F_{Jor}\langle X\rangle\longrightarrow F_{Ass}\langle X\rangle^{(+)}}$, 
${F_*\langle X\rangle\longrightarrow F_{Ass}\langle X\rangle^{(-)}}$, 
${*=Lie, Mal}$, extending the identification ${x_i\longmapsto x_i}$, ${i\geq 1}$ 
\cite[Introduction]{Gol2}. The Lie and Mal'tsev $PI$--algebras over a 
field $F$ considered below can be defined as algebras with the identity 
${f=0}$ for ${0\ne f\in F_*\langle X\rangle}$, ${*=Lie, Mal}$.

The \emph{polynomial degree} $\pideg A$ of an associative $PI$--algebra $A$ is 
the integer part $[\frac{n}2]$, where $n$ is the smallest of all ${m\geq 1}$, 
${\st_m=0}$ is the identity of $A/\prr(A)$, ${\pideg A=0}$ for ${A=\prr(A)}$ and  
${\pideg A=\sqrt{\max\limits_{A\ne P\in \Spec(A)} \dim_{\CM(A/P)} P(A/P)}}$
for ${A\ne \prr(A)}$.

The construction of orthogonal completion of semiprime algebras from \cite{GolO} 
described above applies not only to the linear algebras with two binary 
operations considered in the paper, but also to any algebras of signature 
$\Omega$ from \cite{Raz} (with multilinear operations). This is close to the 
original degree of generality of the orthogonal completeness method from 
\cite{BM, BMO} for orthogonally complete $\Omega$--algebras (see definitions 
below from \cite{BD}).

A Boolean ring $B$ is \emph{orthogonally complete} if any of the equivalent 
conditions is satisfied:
\begin{enumerate}

\item for any dense orthogonal ${\{\xi_a\}_{a\in I}\subseteq B}$ and 
${\{x_a\}_{a\in I}\subseteq B}$ there is ${x\in B}$, ${\xi_a x=\xi_a x_a}$ 
for all ${a\in I}$;

\item $\Ann_B E$ is a principal ideal of $B$ for any ${E\subseteq B}$

\end{enumerate}
\cite[Secs.~2.1.3,~2.1.2]{BD}, \cite[Sec.~1.7]{BMO}. For such a $B$, an algebra $R$ 
of signature $\Omega$ is an \emph{$\Omega$--algebra over $B$} if there exists a 
mapping ${B\times R\ni (b, x)\longmapsto b x\in R}$, 
\begin{enumerate}

\item ${(b c) x=b (c x)}$, ${(b\oplus c) x=b x+c x-2 b c x}$ for all 
${b, c\in B}$, ${x\in R}$; 

\item ${\omega(b x_1, \ldots, b x_n)=b \omega(x_1, \ldots, x_n)}$ for 
any ${n\geq 1}$, $n$--ary operation ${\omega\in \Omega}$, ${x_i\in R}$, 
${b\in B}$,

\end{enumerate}
an $\Omega$--algebra $R$ over $B$ is \emph{orthogonally complete} if $R$ is 
non-singular ($\Ann_B x$ is a principal ideal for all ${x\in R}$) and $R$ 
is an orthogonally complete set (${S\subseteq R}$ for a non-singular $R$ is 
orthogonally complete if for any dense orthogonal 
${\{\xi_a\}_{a\in I}\subseteq B}$, ${\{x_a\}_{a\in I}\subseteq S}$ there is 
${x\in S}$, ${\xi_a x=\xi_a x_a}$ for all ${a\in I}$) 
\cite[Secs.~2.2.1,~2.2.5]{BD}, \cite[Secs.~2.11,~5.2]{BMO}.

A formula $\Psi$ of the first stage is a \emph{Horn formula of signature 
$\Omega$} if $\Psi$ equivalent to a formula of the form 
${(Q_1 x_1)\ldots (Q_n x_n)(\Psi_1\wedge \cdots \wedge \Psi_n)}$,
where ${n\geq 1}$, ${Q_i\in \{\exists, \forall\}}$, each $\Psi_i$ is one 
of the formulas of the form: 
\begin{enumerate}

\item ${t_1\in T_1}$;

\item ${(t_1\notin T_1)\vee\cdots \vee (t_k\notin T_k)}$;

\item ${(t_1\notin T_1)\vee\cdots \vee (t_k\notin T_k)\vee 
(t_{k+1}\in T_{k+1})}$

\end{enumerate}
for some ${k\geq 1}$, terms $t_j$ of signature $\Omega$, orthogonally 
complete ${0\in T_j\subseteq R}$ \cite[Sec.~2.3.2]{BD}, 
\cite[Sec.~5.13]{BMO} 
(instead of ${t\in \{0\}}$ and ${t\notin \{0\}}$, equivalent formulas 
${t=0}$ and ${t\ne 0}$ are used, $t$ is a term of signature $\Omega$).

The key result of \cite{BD} for our constructions is Theorem 2.3.9 
(\cite[Theorem~5.21,~Corollary~5.22]{BMO}): if $R$ is an orthogonally 
complete $\Omega$--algebra over $B$, ${\{\Psi_i(x_1, \ldots, x_n)\}_{i=1}^k}$ 
are Horn formulas, ${k\geq 1}$, ${\Phi(y_1, \ldots, y_m)}$ is a hereditary 
formula in the $\Omega$--algebra $R$ such that 
${\neg \Phi(y_1, \ldots, y_m)}$ is a Horn formula ($\{x_i\}$, $\{y_j\}$ are 
free variables; for any ${b, c\in B}$, ${a_1, \ldots, a_m\in R}$ from 
$c=c b\ne 0$ and the truth of ${\Phi(b a_1, \ldots, b a_m)}$ in $b R$ 
follows the truth of ${\Phi(c a_1, \ldots, c a_m)}$ in $c R$), 
${u_1, \ldots, u_n, v_1, \ldots, v_m\in R}$, ${\Phi(v_1, \ldots, v_m)}$ is true 
in $R$ and for any ideal ${B\ne P\in \Spec(B)}$ there exists $i$, 
${\Phi(\pi_P v_1, \ldots, \pi_P v_m)\to \Psi_i(\pi_P u_1, \ldots, \pi_P u_n)}$ 
is true in ${R/P R}$, where $\pi_P$ is the canonical epimorphism of $R$ onto 
$R/P R$, then there are pairwise orthogonal ${e_1, \ldots, e_k\in B}$, 
${1=e_1\oplus \ldots \oplus e_k}$, ${\Psi_i(e_i u_1, \ldots, e_i u_n)}$ is 
true in $e_i R$ for ${e_i\ne 0}$. More precisely, as in 
\cite[Theorem~3.1.11]{BD}, we will use the proof of Theorem 2.3.9, in which the existence of 
${e_1, \ldots, e_k}$ is deduced from the truth of 
${\Psi_i(\pi_P u_1, \ldots, \pi_P u_n)}$ in $R/P R$ for any  
${B\ne P\in \Spec(B)}$ and some ${1\leq i=i(P)\leq k}$.

In all Horn formulas below, ${R=O(R)}$ is any semiprime orthogonally complete 
algebra over $B(R)$; the question of their truth relates to the algebras under 
consideration, which can also be denoted by $R$. 

\section{Almost classical localizations and orthogonal completion}

The construction of an orthogonal completion of a semiprime algebra is 
possible not only in the ijective hull of its central closure 
as a module over its algebra of multiplication with unity 
\cite[Proposition~2.4]{GolO}, but also in the almost classical localization of its 
central closure, the construction of which from \cite{BD} is considered 
in this section. Almost classical localizations of $F$--algebras do not 
require the mandatory presense of 1 in $F$ and, accordingly, for $F$ 
without 1, the unitarity of $F$--algebras as $F$--modules. 

Let the ring $F$ be semiprime, $R$ be an $F$--algebra, 
${\mathcal{F}=\{I\lhd F\mid \Ann_F I=\{0\}\}=\mathcal{E}(F)}$, 
$R$ does not have $\mathcal{F}$--torsion, i.e. 
${I x\ne \{0\}}$ for any ${0\ne x\in R}$, ${I\in \mathcal{F}}$. On the set 
of pairs 
${\mathcal{C}=\{(\phi, I)\mid I\in \mathcal{F},\ \phi\in \Hom(I, R)_F\}}$ 
we introduce the equivalence relation: ${(\phi, I)\sim (\psi, J)}$ if 
${(\phi-\psi) (I\cap J)=\{0\}}$ (equivalent to ${(\phi-\psi) I'=\{0\}}$ 
for some ${I'\in \mathcal{F}}$, ${I'\subseteq I\cap J}$, due to   
${\{0\}=(\phi-\psi) ((I\cap J) I')=(\phi-\psi)(I\cap J) I'=
(\phi-\psi)(I\cap J)}$), and transform the quotient set 
${R_{\mathcal{F}}=\mathcal{C}/\sim}$ into an $F$--algebra with 
the operatons of equivalence classes 
\begin{gather*}
[(\phi, I)]+[(\psi, J)]\ =\ [(\phi+\psi, I J)]\,,\quad 
[(\phi, I)] [(\psi, J)]\ =\ [(\tau, I J)]\,,
\\ 
f [(\phi, I)]\ =\ [(\phi, I)] f\ =\ [(f \phi, I)]\quad 
((\phi, I), (\psi, J)\in \mathcal{C},\ f\in F)\,,
\end{gather*}
${\tau \sum\limits_{i=1}^k x_i y_i=\sum\limits_{i=1}^k 
(\phi  x_i) (\psi y_i)}$, ${x_i\in I}$, ${y_i\in J}$, ${k\geq 1}$. Then 
${x\longmapsto [(r_x, F)]}$, ${x\in R}$, ${r_x: f\longmapsto f x}$, 
${f\in F}$, is an embedding of $F$--algebras 
${R\hookrightarrow R_{\mathcal{F}}}$, $R_{\mathcal{F}}$ has no 
$\mathcal{F}$--torsion, ${z [(\phi, I)]=[(r_{\phi z}, F)]}$ for all 
${[(\phi, I)]\in R_{\mathcal{F}}}$, ${z\in I}$, and hence, modulo the 
identification of $R$ with its image in $R_{\mathcal{F}}$, for any 
${0\ne r\in R_{\mathcal{F}}}$ there is ${I\in \mathcal{F}}$, 
${\{0\}\ne I r\subseteq R}$, $R$ is an essential $F$--submodule of 
$R_{\mathcal{F}}$, $R \theta=\{0\}$ for ${\theta\in M(R_{\mathcal{F}})'}$ is 
equivalent to ${\theta=0}$, ${M^{R_{\mathcal{F}}}(R)'\cong M(R)'}$, 
${R_{\mathcal{F}}\hookrightarrow I(R)}$, $R_{\mathcal{F}}$ inherits the 
homogeneous identities of $R$, as well as the non-degeneracy (strong primeness) 
of $R$ from the class of algebras with such a condition 
(elementwise characterization of strong primeness), 
for any ${I\in \mathcal{F}, \phi\in \Hom(I, R_{\mathcal{F}})_F}$ there exists 
${x\in R_{\mathcal{F}}, \phi z=z x, z\in I}$ 
\cite[Propositions~2.4.2,~2.4.4]{BD}, 
${R_{\mathcal{F}}\cong (R_{\mathcal{F}})_{\mathcal{F}}}$. For $F$ this 
construction gives a complete ring of quotients 
\[
F_{\mathcal{F}}\ =\ Q(F)\ =\ P(F)\ =\ I(F)\ =\ \CM(F)\, 1\ \cong\ \CM(F)\,, 
\]
which is regular and self-injective \cite[Sec.~2.4.3]{BD}, \cite[Proposition~1,~p.~75]{Lamb}, 
\cite[Vol.~2,~Theorem\linebreak 19.14~A,~p.~111]{Fe} 
(self-injectivity of ${Q(F)=Q(Q(F))}$ with 1 is equivalent to quasi-injectivity 
of $Q(F)_{Q(F)}$ \cite[(1.15)]{Ut}, \cite[Lemma~1,~p.~147]{Lamb}, 
\cite[Example~19.1.4,~p.~103]{Fe}, for any $F$--submodule ${M\subseteq Q(F)_F}$, 
${\sigma\in \Hom(M, Q(F))_F}$ $\sigma$ can be extended by $Q(F)$--linearity 
to an element ${\Hom(M Q(F), Q(F))_{Q(F)}}$ (if 
${\sum\limits_i x_i q_i=0, x_i\in M, q_i\in Q(F)}$, then there are 
${H_i\in \mathcal{F}}$, ${q_i H_i\subseteq F}$, 
${H=\bigcap\limits_i H_i\in \mathcal{F}}$, 
${\sum\limits_i (\sigma x_i) q_i h=
\sigma\Bigl(\sum\limits_i x_i q_i h\Bigl)=0}$ for all ${h\in H}$, 
${\sum\limits_i (\sigma x_i) q_i=0}$), and so there is  
${q\in Q(F)}$, ${\sigma=l_q=r_q}$, 
${\End(Q(F))_F=\End(Q(F))_{Q(F)}=M(Q(F))\cong Q(F)}$). The action of $F$ on 
$R_{\mathcal{F}}$ can be extended to the action of $F_{\mathcal{F}}$ on 
$R_{\mathcal{F}}$ by setting 
$\langle (\alpha, H)\rangle [(\phi, I)]=[(\xi, H I)]$, 
${\langle (\alpha, H)\rangle\in F_{\mathcal{F}}}$, 
${[(\phi, I)]\in R_{\mathcal{F}}}$, 
${\xi \Bigl(\sum\limits_i x_i y_i\Bigr)=\sum\limits_i (\alpha x_i) 
(\phi y_i)}$, ${x_i\in H}$, ${y_i\in I}$. The 
$F_{\mathcal{F}}$--algebra $R_{\mathcal{F}}$ defined in this way is 
orthogonally complete over the orthogonally complete Boolean ring 
${B(F)=\{f\in F_{\mathcal{F}}\mid f=f^2\}}$ and is called an 
\emph{almost classical localization of $R$} 
\cite[Proposition~2.4.4]{BD}, \cite[the~conclusions~after~Lemma 2.3]{GolO}. 

If ${R\ne \{0\}}$ is semiprime, ${\mathcal{F}=\mathcal{E}(\CM(R))}$, then  
${I P(R)\in \mathcal{E}(P(R))}$ for all ${I\in \mathcal{F}}$, since otherwise 
for ${I\in \mathcal{F}}$ there are ${\{0\}\ne K\lhd P(R)}$, 
${I P(R)\cap K=\{0\}}$, and ${0\ne \theta=\theta^2\in \CM(R)}$, 
${\theta|_K=\Id_K}$, ${\theta I P(R)=\{0\}}$, ${\theta I=\{0\}}$, 
${\theta=0}$, a contradiction. Therefore ${(x)_{P(R)} I P(R)\ne \{0\}}$ and ${I x\ne \{0\}}$ 
for all ${0\ne x\in P(R)}$, ${I\in \mathcal{F}}$, $P(R)$ has no 
$\mathcal{F}$--torsion. Due to the self-injectivity of $\CM(R)$ \cite[Theorem~19.27,~p.~123]{Fe}, \cite[Lemma~2.3]{GolO}, 
${\CM(R)=\CM(R)_{\mathcal{F}}=Q(\CM(R))}$, the $\CM(R)$--algebra 
$P(R)_{\mathcal{F}}$ is orthogonally complete over ${B(R)=B(\CM(R))}$. This 
allows us to define orthogonal completion of $O(R)$ and $O(P(R))$ in 
${P(R)_{\mathcal{F}}=O(P(R)_{\mathcal{F}})\subseteq I(P(R))}$. 

\begin{prop}
The algebra $P(R)_{\mathcal{F}}$ is the orthogonal 
completion of the algebra $P(R)$, 
\[
P(R)_{\mathcal{F}}\ =\ O(P(R))\ =\ P(P(R)_{\mathcal{F}})\,,\quad 
\CM(P(R)_{\mathcal{F}})\ =\ \CM(R) \Id_{P(R)_{\mathcal{F}}}\,.
\]
\end{prop}

\begin{proof} 
If ${x\in P(R)_{\mathcal{F}}}$, ${I\in \mathcal{F}}$, ${I x\subseteq P(R)}$ and 
${\{\xi_a\}_{a\in A}}$ is any maximal orthogonal system in 
${I\cap B(R)}$, then 
${\{0\}=\Ann_{I\cap B(R)} \{\xi_a\}_{a\in A}=
\Ann_I \{\xi_a\}_{a\in A}\supseteq I \Ann_{\CM(R)} \{\xi_a\}_{a\in A}}$ 
(regularity of $\CM(R)$), 
${\Ann_{\CM(R)} \{\xi_a\}_{a\in A}=\{0\}}$, 
$\{\xi_a\}_{a\in A}$ is dense in $B(R)$, ${\xi_a (x-x')=0}$ for all ${a\in A}$\linebreak 
and ${x'={\sum\limits_{a\in A}}^{\perp} \xi_a x\in 
O(P(R))\subseteq P(R)_{\mathcal{F}}}$, ${x=x'}$ \cite[Remark~2.2]{GolO}. 
So, ${O(P(R))=P(R)_{\mathcal{F}}}$.

If ${\psi\in \Hom(H, P(R)_{\mathcal{F}})_{M(P(R)_{\mathcal{F}})'}}$, 
${H\lhd P(R)_{\mathcal{F}}}$, ${H'=P(R)\cap \psi^{-1} P(R)}$, then 
\[
\psi|_{H'}\ =\ \psi'\ =\ \overline{\psi'} \Id_{H'}\in 
\Hom(H', P(R))_{M^{P(R)_{\mathcal{F}}}(P(R))'}\ =\ 
\Hom(H', P(R))_{M(P(R))'}\,,
\]
${\overline{\psi'}\in \CM(R)}$ (${P(R)=P(P(R))}$), for any ${[(\phi, I)]\in H}$ 
and ${\psi [(\phi, I)]=[(\phi', I')]\in P(R)_{\mathcal{F}}}$
\[
(I\cap I') (\psi [(\phi, I)])\ =\ \psi' ((I\cap I') [(\phi, I)])\ =\ 
\overline{\psi'} (I\cap I') [(\phi, I)]
\]
and ${\psi=\overline{\psi'} \Id_H=
\overline{\psi'} \Id_{P(R)_{\mathcal{F}}}|_H}$ 
($P(R)_{\mathcal{F}}$ has no $\mathcal{F}$--torsion). Hence 
$P(R)_{\mathcal{F}}$ is a quadi-injective $M(P(R)_{\mathcal{F}})'$--module, 
${P(R)_{\mathcal{F}}=P(P(R)_{\mathcal{F}})}$, 
${\CM(P(R)_{\mathcal{F}})=\CM(R) \Id_{P(R)_{\mathcal{F}}}}$ 
\cite[Proposition 2.4~(4)]{GolO}. If $R$ is prime, then $\CM(R)$ is a field, 
${P(R)=P(R)_{\mathcal{F}}}$. 
\end{proof}

A semiprime algebra $R$ with a semiprime (prime) algebra $M(R)$ is called 
a \emph{multiplicatively semiprime} (\emph{prime}), hereinafter a 
\emph{$m.s.p.$--algebra} (\emph{$m.p.$--algebra}). Each $m.p.$--algebra 
is prime. For a semiprime $R$, the semiprimeness conditions of $M(R)$ 
and $M(R)'$ are equivalent, since ${M(R)\lhd M(R)'}$, 
${\prr(M(R))=M(R)\cap \prr(M(R)')}$, for the $m.s.p.$--algebra $R$ the 
presence of ${I\lhd M(R)'}$, ${I^2=\{0\}}$, implies 
\begin{gather*}
I M(R)\subseteq I\cap M(R)\subseteq \prr(M(R))\ =\ \{0\}\,,
\\
\{0\}\ =\ R (I M(R))\ =\ (R I) M(R)\ =\ (R I)^2\ =\ R I
\end{gather*}
(${R J\lhd R}$ for all ${J\lhd M(R)' (J\lhd M(R))}$), ${I=\{0\}}$, 
${\prr(M(R)')=\{0\}}$. The proof of \cite[Theorem~4.3]{Cab3}, allows us to 
conclude that
\begin{gather*}
P(M(R))\cong M(P(R))\,,\quad P(M(R)')\cong M(P(R))'\,,
\\
\CM(M(R))\ =\ \CM(R) \Id_{P(M(R))}\cong \CM(R)\cong 
\CM(M(R)')\ =\ \CM(R) \Id_{P(M(R)')}\,.
\end{gather*}

If $R$ is a subdirect product of $m.p.$--algebras $R/P_a$, 
${R\ne P_a\in \Spec(R)}$, ${a\in A}$, then $M(R)$ is a subdirect 
product of ${M(R/P_a)\cong M(R)/C(R, P_a)}$, ${a\in A}$,
\[
\bigcap\limits_{a\in A} C(R, P_a)\ =\ C(R, \prr(R))\ =\ C(R, \{0\})\ =\ \{0\}\,,
\]
where ${C(R, I)=\{\phi\in M(R)\mid R \phi\subseteq I\}}$ is the kernel of the 
epimorphism ${M(R)\longrightarrow M(R/I)}$ induced by the natural epimorphism 
${R\longrightarrow R/I}$, ${I\lhd R}$, and $R$ is a $m.s.p.$--algebra. If $R$ 
is semiprime, $P(R)$ is simple, ${\dim_{\CM(R)} P(R)=n< \infty}$ 
(${\CM(R)=\CM(P(R))}$ is a field), then by the density theorem 
\[ 
M(P(R))\ =\ M(P(R))'\ =\ \CM(R) M^{P(R)}(R)\ \cong\ M_n(\CM(R))
\]
is a simple finite-dimensional $\CM(R)$--algebra \cite[Proposition~3,~p.~92]{Lamb}, 
${M^{P(R)}(R)\cong M(R)}$ is prime, satisfies the same homogeneous identities 
as $M(P(R))$, including all ${\st_m=0}$, ${m\geq 2 n}$, $R$ is a 
$m.p.$--algebra. In particular, the latter is true if $R$ is prime, 
${S=Z(R)\setminus \{0\}\ne \emptyset}$ and the ring of quotients 
${R S^{-1}}$ is a central simple algebra over the field 
${Z(R S^{-1})=Z(R) S^{-1}}$, ${\dim_{Z(R S^{-1})} R S^{-1}=n< \infty}$, since 
in this case ${P(R)=R S^{-1}}$, ${\CM(R)=Z(R S^{-1})}$ 
\cite[Theorem~1.7]{CFGM}, \cite[\S 3]{Gol6}.  

As a consequence, if $R$ is a subdirect product of $R/P_a$, 
${R\ne P_a\in \Spec(R)}$, $P(R/P_a)$ is simple, 
${\dim_{\CM(R/P_a)} P(R/P_a)=n_a< \infty}$, ${a\in A}$, then $R$ is a 
$m.s.p.$--algebra (and $M(R)$ is a $PI$--algebra iff  
${n=\sup\limits_{a\in A} n_a< \infty}$ ($M(R)$ satisfies ${\st_m=0}$, 
${m\geq 2 n}$)).  

\begin{rem}
If $R$ is a $m.s.p.$--algebra and there exist ${k, n\geq 1}$ 
such that the elements of $M(P(R))$ are sums of $k$ products 
of at most $n$ operators $t_x$, ${t=l, r}$, ${x\in P(R)}$, 
then  ${O(M(P(R)))=M(O(P(R)))}$, ${O(M(P(R))')=M(O(P(R)))'}$. 
\end{rem}

\begin{proof}
In this case, $M(P(R))'$ is identified with $M^{O(P(R))}(P(R))'$, due to 
\[
M(P(R))'\ =\ \{\psi|_{P(R)}\mid \psi\in M^{O(P(R))}(P(R))'\}\ \cong\ 
M^{O(P(R))}(P(R))'\,.
\] 
Since each ${\psi\in M(O(P(R)))'}$ has the form  
${\psi=\tau \Id_{O(P(R))}+\sum\limits_{i=1}^l 
t^{(i1)}_{x(i, 1)}\cdots t^{(im_i)}_{x(i, m_i)}}$, 
${l, m_i\geq 1}$, ${\tau\in \CM(R)}$, ${t^{(ij)}\in \{l, r\}}$, 
${x(i, j)={\sum\limits_{a\in A_{ij}}}^{\perp} \xi(i, j)_a x(i, j)_a
\in O(P(R))}$ for some dense orthogonal 
${\{\xi(i, j)_a\}_{a\in A_{ij}}\subseteq B(R)}$, 
${\{x(i, j)_a\}_{a\in A_{ij}}\subseteq P(R)}$, ${\Bigl\{\xi(\{a_{ij}\})=
\prod\limits_{i, j} \xi(i, j)_{a_{ij}}\Bigr\}_{\{a_{ij}\}\in A}\subseteq B(R)}$ 
is dense and orthogonal, ${A=\{\{a_{ij}\}\mid a_{ij}\in A_{ij}\}}$, and 
\[
\xi(\{a_{ij}\}) \psi\ =\ \xi(\{a_{ij}\}) 
\biggl(\tau \Id_{O(P(R))}+\sum_{i=1}^l 
t^{(i1)}_{x(i, 1)_{a_{i1}}}\cdots t^{(im_i)}_{x(i, m_i)_{a_{i m_i}}}\biggr)
\in M(P(R))'\quad (\{a_{ij}\}\in A)\,,
\]
${M(O(P(R)))\subseteq O(M(P(R)))}$, ${M(O(P(R)))'\subseteq O(M(P(R))')}$ 
without additional condition with ${k, n\geq 1}$. For any dense orthogonal 
${\{\xi_a\}_{a\in A}\subseteq B(R)}$, ${\{\phi_a\}_{a\in A}\subseteq M(P(R))'}$,  
\[
\phi_a\ =\ \tau_a \Id_{O(P(R))}+
\sum_{j=1}^n \sum_{i=1}^k \sum_{T=(t^{(1)}, \ldots, t^{(j)})\in 
\{l, r\}^j} t^{(1)}_{x(a, i, T, 1)}\cdots t^{(j)}_{x(a, i, T, j)}
\]
for ${\tau_a\in \CM(R)}$, ${x(a, i, T, j)\in P(R)}$, ${a\in A}$, where the 
sum of ${1+k (2^{n+1}-2)}$ terms contains no more than ${k+1}$ non-zero terms, 
\[
{\sum_{a\in A}}^{\perp} \xi_a \phi_a\ =\ 
\tau \Id_{O(P(R))}+
\sum_{j=1}^n \sum_{i=1}^k \sum_{T=(t^{(1)}, \ldots, t^{(j)})\in 
\{l, r\}^j} t^{(1)}_{x(i, T, 1)}\cdots t^{(j)}_{x(i, T, j)}\in 
M(O(P(R)))'\,,
\]
${\tau\in \CM(R)}$, ${\xi_a \tau=\xi_a \tau_a}$, ${a\in A}$, 
${x(i, T, l)={\sum\limits_{a\in I}}^{\perp} \xi_a x(a, i, T, l)}$ 
\cite[the~proofs~of~Lemmas~2.3, 2.6]{GolO}. Therefore,  
${O(M(P(R)))=M(O(P(R)))}$ inherits the condition for $M(P(R))$ 
with the same ${k, n\geq 1}$, ${O(M(P(R))')=M(O(P(R)))'}$.
\end{proof}

In particular, $M(P(R))$ satisfies the condition of Remark 2.2 with suitable 
${k, n\geq 1}$ if $M(P(R))$ is a finitely generated $\CM(R)$--module.

Remark 2.3 follows directly from \cite[Theorem~2.3.14]{BD}. Similarly, 
\cite[Lemmas~2.8,~2.9]{GolO} and similar conclusions can be 
considered as variants of \cite[Theorem~2.3.12]{BD}.

\begin{rem}
If ${R=O(R)}$, then ${Z(R)=O(Z(R))}$, ${Z(R_B)=(Z(R)+P R)/P R\cong Z(R)_B}$ 
for any ${B\in \mathcal{U}(R)}$, ${P=B(R)\setminus B}$.
\end{rem}

\begin{proof}
By virtue of ${Z(R)=\{x\in R\mid Z_x=\{0\}\}}$, where   
${Z_x=\{f(x, y_2, \ldots, y_8)\mid y_i\in R\}}$ for all ${x\in R}$ and 
\[
f(x_1, x_2, \ldots, x_8)\ =\ (x_1, x_2, x_3)+(x_4, x_1, x_5)+
(x_6, x_7, x_1)+[x_1, x_8]\in F\langle X\rangle\,,
\]
for any dense orthogonal ${\{\xi_a\}_{a\in A}\subseteq B(R)}$, 
${\{x_a\}_{a\in A}\subseteq Z(R)}$ and 
${x={\sum\limits_{a\in A}}^{\perp} \xi_a x_a}$ 
\[
\xi_a Z_x\ =\ Z_{\xi_a x}\ =\ Z_{\xi_a x_a}\ =\ \xi_a Z_{x_a}\ =\ \{0\}
\quad (a\in A)\,,
\]
${Z_x=\{0\}}$, ${x\in Z(R)=O(Z(R))}$ and, as a consequence, 
${B(R) Z(R)\subseteq Z(R)}$. 

If ${x\in R}$, ${x+P R\in Z(R_B)}$, ${P=B(R)\setminus B}$, then 
${Z_x=O(Z_x)\subseteq P R}$ and there is ${\alpha\in B}$, 
${\alpha Z_x=Z_{\alpha x}=\{0\}}$ \cite[Proposition~2.4~(6)]{GolO}, 
${\alpha x\in Z(R)}$, 
\[
x+P R\ =\ \alpha x+P R\in (Z(R)+P R)/P R\ \cong\ 
Z(R)/P Z(R)\ =\ Z(R)_B
\]
(${P Z(R)=Z(R)\cap P R}$).
\end{proof} 

\begin{lemma}
If ${R=P(R)=O(R)}$, ${S_Q=Z(R/Q)\setminus \{0\}\ne \emptyset}$,  
${R_Q=R/Q\, S_Q^{-1}}$ is a central simple finite-dimensional algebra over 
the field ${K_Q=Z(R/Q) S_Q^{-1}=Z(R_Q)}$ for all $R\ne Q\in \Spec(R)$, 
${m=\sup\limits_{R\ne Q\in \Spec(R)} \dim_{K_Q} R_Q< \infty}$, then 
${m(R)=\max\limits_{B\in \mathcal{U}(R)} \dim_{K_B} R_B}$ is the minimum 
number of generators of the $Z(R)$--module $R$, ${m(R)\leq m}$, where 
\[
R_B\ =\ R/P R\ =\ P(R_B)\,,\quad K_B\ =\ Z(R_B)\ =\ \CM(R_B)\ =\ \CM(R)_B
\]
for any ${B\in \mathcal{U}(R)}$, ${P=B(R)\setminus B}$, and ${M(R)=O(M(R))}$.
\end{lemma}

\begin{proof}
By condition, $R$ is a $m.s.p.$--algebra with ${Z(R)\ne \{0\}}$ and a 
$PI$--algebra $M(R)$, $m=\pideg M(R)$, ${R_Q=P(R/Q)}$, ${K_Q=\CM(R/Q)}$, 
$R_B$ is prime and  
\begin{gather*}
\Ann_{\CM(R)} R_B\ =\ \{\alpha\in \CM(R)\mid \alpha R\subseteq P R\}\ =\ 
P \CM(R)\,,
\\
\CM(R) \Id_{R_B}\ \cong\ \CM(R)/\Ann_{\CM(R)} R_B\ =\ \CM(R)_B
\end{gather*}
for any ${R\ne Q\in \Spec(R), B\in \mathcal{U}(R), P=B(R)\setminus B}$ 
(Remark 2.3, the observations before it, \cite[Corollary~2.7]{GolO}; 
if ${\alpha\in \Ann_{\CM(R)} R_B}$, then ${\alpha R=O(\alpha R)\subseteq P R}$, 
${\beta \alpha R=\{0\}}$, ${\beta \alpha=0}$ for some ${\beta\in B}$, 
${\alpha\in P \CM(R)}$ \cite[Proposition~2.4~(6)]{GolO}). In view of the 
isomorphisms $z\longmapsto l_z=r_z\in \CM(R)$, ${z\in Z(R)}$, 
${\CM(R)_B\cong \CM(R) \Id_{R_B}}$, we can assume that ${Z(R)\lhd \CM(R)}$, 
${Z(R)_B=Z(R_B)\subseteq \CM(R)_B\subseteq \CM(R_B)}$. Since $\CM(R)$ is 
regular and $P \CM(R)$ is its maximal ideal 
\cite[the~observations~before~Lemma~2.5]{GolO}, $\CM(R)_B=Z(R)_B$ is a field, ${R_B=P(R_B)}$, 
${K_B=Z(R_B)=\CM(R_B)=\CM(R)_B}$. 

Similar to \cite[the~proof~of~Theorem~3.1.11]{BD}, from the truth of the 
Horn formula 
$\Psi_k=(\exists x_1)\ldots (\exists x_k) \Phi_k(x_1, \ldots, x_k)$ for 
${k=m(R)}$ in all $R_B$, ${B\in \mathcal{U}(R)}$, 
\[
\Phi_k(x_1, \ldots, x_k)\ =\ (\forall x)(\exists y_1)\ldots 
(\exists y_k)\biggl[\biggl\{\bigwedge_{i=1}^k (y_i\in Z(R))
\biggr\}\wedge \biggl\{\sum_{i=1}^k y_i x_i=x\biggr\}\biggr]\,,
\]
one can deduce its truth in $R$ \cite[Theorem~2.3.9]{BD}, 
\cite[Corollary~5.22]{BMO}. Therefore, the $Z(R)$--module $R$ has $m(R)$ 
generators, the minimality of the number of generators $m(R)$ follows 
from the definition of $m(R)$. 
Since ${\dim_{K_B} R_B=k}$, ${1\leq k\leq m(R)}$, ${B\in \mathcal{U}(R)}$, 
if and only if the Horn formula ${\Psi_k'=(\exists x_1)\ldots (\exists x_k) 
[\Phi_k(x_1, \ldots, x_k)\wedge \Phi_k'(x_1, \ldots, x_k)]}$ is true in $R_B$,
\[
\Phi_k'(x_1, \ldots, x_k)\ =\ (\forall y_1)\ldots (\forall y_k)
\biggl[ \biggl\{\bigvee_{i=1}^k (y_i\notin Z(R))\biggr\} \vee 
\biggl\{\bigwedge_{i=1}^k (y_i=0)\biggr\} \vee 
\biggl\{\sum_{i=1}^k y_i x_i\ne 0\biggr\} \biggr]\,,
\]
one can select ${0\ne \eta_i\in B(R)}$, ${\eta_i \eta_j=\delta_{ij} \eta_i}$, 
${\Id_R=\eta_1+\ldots+\eta_l}$, $\Psi_{k_i}'$ is true in $\eta_i R$, 
$\{k_i\}_{i=1}^l\subseteq \{\dim\limits_{K_B} R_B\mid B\in \mathcal{U}(R)\}$, 
${k_0=0< k_1<\ldots < k_l=m(R)}$ \cite[Theorem~2.3.9]{BD}, 
\cite[Corollary 5.22]{BMO}, the $Z(\eta_i R)$--module $\eta_i R$ has a system of 
generators $\{x_{i j}\}_{j=1}^{k_i}$ through which its elements are uniquely 
expressed over ${Z(\eta_i R)=\eta_i Z(R)}$, ${i=1, \ldots, l}$, 
${R=\sum\limits_{s=1}^{m(R)} Z(R) x_s}$, 
${x_s=\sum\limits_{j=i}^l x_{j s}}$, ${k_{i-1}+1\leq s\leq k_i}$.

It remains to note that ${\dim_{K_B} M(R_B)\leq m(R)^2}$, 
${B\in \mathcal{U}(R)}$, and hence, in $R_B$, ${B\in \mathcal{U}(R)}$, and 
$R$ the Horn formula  
${\Psi''_k=(\exists x_1)\ldots (\exists x_k)[\Phi_k(x_1, \ldots, x_k)\wedge 
\Phi_k''(x_1, \ldots, x_k)]}$, 
\begin{multline*}
\Phi_k''(x_1, \ldots, x_k)\ =\ 
(\exists y_{11})\ldots (\exists y_{m n}) 
\biggl[\biggl\{\bigwedge_{i=1}^m \bigwedge_{j=1}^n (y_{ij}\in Z(R))\biggr\}
\wedge
\\
\biggl\{\bigwedge_{i=1}^m \bigwedge_{l=1}^k 
\biggl(\sum_{j=1}^n y_{ij} x_l \omega_j(x_1,\ldots x_k)\ =\ 
x_l \upsilon_i(x_1, \ldots, x_k)\biggr)\biggr\}\biggr]
\end{multline*}
is true for ${k=m(R)}$, ${m=(2 k)^{k^2+1}}$, ${n=((2 k)^{k^2+1}-1)/(2 k-1)}$, where 
\begin{gather*}
W_i(z_1, \ldots, z_k)\ =\ \{t_1\cdots t_i\mid t_i\in \{t_{z_i}\mid t=l, r,\ 
i=1, \ldots, k\}\}\ \subseteq\ M(F\langle X\rangle)\quad (i\geq 1)\,,
\\
W_{m(R)^2+1}(z_1, \ldots, z_k)\ =\ \{\upsilon_i(z_1, \ldots, z_k)\}_{i=1}^m\,,
\quad  
\bigcup_{l=1}^{m(R)^2} W_l(z_1, \ldots, z_k)\ =\ 
\{\omega_j(z_1, \ldots, z_k)\}_{j=1}^n
\end{gather*}
for any ${\{z_l\}\subseteq X}$ and selected numbering methods. As a 
consequence, ${R=\sum\limits_{i=1}^{m(R)} Z(R) x_i}$,
\[
t_1\ldots t_s\in 
\sum_{j=1}^n Z(R) \omega_j(x_1, \ldots, x_{m(R)})\quad 
(s\geq m,\ t_i\in \{t_x\mid t=l, r,\ x\in R\})\,,
\]
${M(R)=\sum\limits_{j=1}^n Z(R) \omega_j(x_1, \ldots, x_{m(R)})=O(M(R))}$ 
(Remark 2.2; if $\CM(R)$ is Noetherian, then this also follows from 
the Noetherian property of the $\CM(R)$--module $\End_{\CM(R)}(R)$ as a 
homomorphic image of a subalgebra of $M_{m(R)}(\CM(R))$).
\end{proof}  

In particular, this is true for any associative $PI$--algebra  
${R=P(R)=O(R)}$ with $m(R)\leq (\pideg R)^2$. 
If in Lemma 2.4 each $K_B$--algebra $R_B$, ${B\in \mathcal{U}(R)}$, 
possesses a basis $\{x_i\}_{i=1}^{n_B}$,  
${x_i x_j=\sum\limits_{k=1}^{n_B} \alpha_{ij}^k x_k}$, 
${n_B=\dim_{K_B} R_B\leq m(R)}$, with relations for the structure 
constants $\{\alpha_{ij}^k\}\subseteq K_B$ given by some Horn formula,
then we can assume that in the decomposition from the proof of Lemma 2.4
${R=\eta_1 R\oplus \ldots\oplus \eta_l R}$ the 
structure constants in $Z(\eta_i R)$ for the generators of the 
$Z(\eta_i R)$--module $\eta_i R$ are related by the corresponding Horn 
formula, ${i=1, \ldots, l}$ \cite[the~proof~of~Theorem 3.1.11]{BD}. 

\begin{prop}
If ${R=P(R)=O(R)}$, ${B\in \mathcal{U}(R)}$, ${\nu(B)\in B}$ or (and) $R$ is a 
$m.s.p.$--algebra with a $PI$--algebra ${M(R)=O(M(R))}$, then ${R_B=P(R_B)}$, 
${\CM(R_B)=\CM(R)_B}$.
\end{prop}

\begin{proof}
First, note that ${\CM(R) \Id_{R_B}\cong \CM(R)_B}$, 
${M(R_B)'\cong M(R)'_B}$, $R_B$ is prime (Lemma 2.4, 
\cite[Lemmas~2.6,~2.5]{GolO}; ${M(P(R/Q))=\CM(R/Q) M^{P(R/Q)}(R/Q)}$ and 
$M^{P(R/Q)}(R/Q)\cong M(R/Q)\cong M(R)/C(R, Q)$, ${Q\in \Spec(R)}$; 
the primeness of $R_B$ can also be deduced from the 
${M(R)_B\cong M(R_B)}$ \cite[Lemma~2.6,~Corollary~2.7]{GolO} and  
$\Ann_t R_B=\{0\}$, ${t=l, r}$ 
(if ${t=l, r}$, ${x\in R}$, ${R t_x=O(R t_x)\subseteq P R}$, then 
${\beta R t_x=R t_{\beta x}=\{0\}}$, ${\beta x=0}$ for some 
${\beta\in B}$ \cite[Proposition~2.4~(6)]{GolO}, ${x\in P R}$, 
${P=B(R)\setminus B}$)). 

If ${\nu(B)\in B}$, ${\psi\in \Hom(H, R_B)_{M(R_B)'}}$, ${H\lhd R_B}$, 
then ${\nu(B) P=\{0\}}$, ${P=B(R)\setminus B}$, the isomorphisms 
${\alpha_S: S_B\longrightarrow \nu(B) S}$, 
${\alpha_S: x+P S\longmapsto \nu(B) x}$, ${x\in S=R, \CM(R)}$, allow 
us to write ${\alpha_R \psi \alpha_R^{-1}=\beta \Id_{\alpha_R H}\in 
\Hom(\alpha_R H, \nu(B) R)_{M(R)'}}$ for some ${\beta\in \CM(R)}$, 
\begin{multline*}
\psi (x+P R)\ =\ \psi (\nu(B) x+P R)
\ =\ \psi \alpha_R^{-1} \alpha_R (x+P R)\ =
\\ 
\alpha_R^{-1} \beta \alpha_R (x+P R)\ =\ 
\beta (x+P R)\ =\ \nu(B) \beta (x+P R)\quad (x+P R\in H)\,,
\end{multline*}
${\psi=\beta \Id_H=\beta \Id_{R_B}|_H}$. 
If ${H=(H'+P R)/P R\in \mathcal{E}(R_B)}$, ${H'\in \mathcal{E}(R)}$ 
(the transition to such $H$ and its preimage $H'$ in $R$ is always 
possible), then 
${\nu(B) \beta}$ is uniquely defined, since ${\gamma H'\subseteq P R}$ 
for ${\gamma\in \CM(R)}$ is equivalent to ${\nu(B) \gamma H'=\{0\}}$, 
${\nu(B) \gamma=\{0\}}$, ${\gamma\in P \CM(R)}$. So, 
${R_B=P(R_B)}$, ${\CM(R_B)=\CM(R)_B}$.

For $R$ with a semiprime $PI$--algebra $M(R)$ the following conditions 
are equivalent:\linebreak ${M(R)=O(M(R))}$; $M(R)$ is a finitely generated 
$\CM(R)$--module; there are ${k, n\geq 1}$ such that the elements of 
$M(R)$ are sums of $k$ products of at most $n$ operators $t_x$, ${t=l, r}$, 
${x\in R}$ (Remark 2.2, Lemma 2.4, ${M(R)=P(M(R))}$, 
${\CM(M(R))=\CM(R) \Id_{M(R)}}$). When they (any of them) are satisfied, 
${M(R_B)\cong M(R)_B=P(M(R)_B)}$ is a central simple finite-dimensional 
algebra over the field 
\[
\CM(M(R)_B)\ =\ Z(M(R)_B)\ =\ \CM(M(R))_B\ \cong\ 
\CM(R)_B\ =\ \CM(R_B)\,,
\]
$R_B$ is a $m.p.$--algebra, ${P(R_B)=\CM(R)_B R_B=\CM(R) R_B=R_B}$
(Lemma 2.4, \cite[Lemma 2.6]{GolO}).
\end{proof}

If ${R=P(R), R^1=\CM(R)\oplus R}$, then ${Z(R^1)=Z(R)^1=\CM(R)\oplus Z(R)}$, 
$R^1$ is semiprime, since the presence of ${I\lhd R^1}$, ${I^2=\{0\}}$, implies  
${(\alpha+x)^2=\alpha^2+(2 \alpha x+x^2)=\alpha=0}$ for all ${\alpha+x\in I}$, 
${\alpha\in \CM(R)}$, ${x\in R}$, ${I\subseteq R}$, ${I=\{0\}}$. For $R$ with 1 
${(\Id_R-1) R=R (\Id_R-1)=\{0\}}$ and therefore the primeness of $R$ with 1 
does not imply the primeness of $R^1$. At the same time, if $R$ 
is prime, anticommutative, and 2--torsion-free, then $R^1$ is prime, due to 
${I\cap R\ne \{0\}}$ for all ${\{0\}\ne I\lhd R^1}$ (otherwise if 
${\{0\}=I\cap R\ne I\lhd R^1}$, ${0\ne \alpha+x\in I}$, 
${0\ne \alpha\in \CM(R)}$, ${x\in R}$, then 
${0=(\alpha+x) y=y (\alpha+x)=x y=y x=\alpha y=y}$ for all ${y\in R}$, a contradiction).

Let ${R=P(R)=O(R)}$ and ${\mathcal{F}=\mathcal{E}(\CM(R))}$. Then $R$, 
$\CM(R)$, $R^1$ have no $\mathcal{F}$--torsion (see above), for any 
${I\in \mathcal{F}}$, ${\phi\in \Hom(I, R^1)_{\CM(R)}}$ we can write 
${\phi=\pi_{\CM(R)} \phi+\pi_R \phi}$, 
$\pi_{\CM(R)} \phi\in \Hom(I, \CM(R))_{\CM(R)}$, 
${\pi_R \phi\in \Hom(I, R)_{\CM(R)}}$, where 
$\pi_{\CM(R)}$ and $\pi_R$ are the canonical projections of $R^1$ onto 
$\CM(R)$ and $R$, ${\pi_{\CM(R)} \phi: \beta\longmapsto \beta \alpha}$, 
${\pi_R \phi: \beta\longmapsto \beta x}$, 
${\phi: \beta\longmapsto \beta (\alpha+x)}$, ${\beta\in I}$, for some 
${\alpha\in \CM(R)}$, ${x\in R}$ and so 
${[(\phi, I)]=[(r_{\alpha+x}, \CM(R))]}$, ${R^1=R^1_{\mathcal{F}}=O(R^1)}$ 
\cite[Proposition~2.4.4]{BD}, \cite[Proposition~2.4~(2)]{GolO}
(${\CM(R)=O(\CM(R)), R=O(R)}$) and up to isomorphism 
${R^1_B=\CM(R)_B\oplus R_B=({}_{\CM(R)_B} R_B)^1}$ for any 
${B\in \mathcal{U}(R)}$ (${P R^1=P \CM(R)\oplus P R}$, ${P=B(R)\setminus B}$). 

\begin{lemma}
If ${R=2 R=P(R)=O(R)}$ is anticommutative, ${R_B=R_B^2}$ for all 
$B\in \mathcal{U}(R)$, 
${\sup\limits_{B\in \mathcal{U}(R)} \dim_{\CM(R)_B} R_B=m(R)< \infty}$, 
then $m(R)$ is the minimum number of generators of the $\CM(R)$--module 
${R=R^2}$ and ${M(R)=O(M(R))}$. 
\end{lemma}

\begin{proof}
Since ${Z(R^1)=\CM(R), Z(R^1_B)=\CM(R)_B}$ for all ${B\in \mathcal{U}(R)}$ 
(Introduction), it is sufficient, by analogy with the proof of 
Lemma 2.4, to deduce the truth in $R^1$ of the Horn formula 
${\Lambda_k=(\exists x_1)\ldots (\exists x_{k+1}) 
[\Delta_k(x_1, \ldots, x_{k+1})\wedge \Phi_{k+1}(x_1, \ldots, x_{k+1})]}$ 
for ${k=m(R)}$, 
\begin{gather*}
\Delta_k(x_1, \ldots, x_{k+1})\ =\ (x_1^2=x_1)\wedge 
\biggl\{\bigwedge_{i=2}^{k+1} (x_1 x_i=x_i)\biggr\}\wedge 
\biggl\{\bigwedge_{i=2}^{k+1} (x_i\in [R, R]_k)\biggr\}\,,
\\
[R, R]_k\ =\ [R^1, R^1]_k\ =\ 
\{[x_1, y_1]+\ldots+[x_k, y_k]\mid x_i, y_i\in R\}\ =\ O([R, R]_k)
\end{gather*}
(${[R_B, R_B]_k=R_B^2=R_B, B\in \mathcal{U}(R)}$; 
${[x, y]=2 x y, x, y\in R}$); this is equivalent to the existence of 
generators $\{x_i\}_{i=1}^{m(R)+1}$ of the $\CM(R)$--module $R^1$, 
${x_1^2=x_1}$ and ${x_1 x_i=x_i\in [R, R]_{m(R)}}$, ${i=2, \ldots, m(R)+1}$. 
Therefore ${x_1=1}$, ${R=R^2=\sum\limits_{i=2}^{m(R)+1} \CM(R) x_i}$, the 
coincidence of $m(R)$ with the minimum number of generators of the 
$\CM(R)$--module $R$ follows from the definition of $m(R)$.

As in Lemma 2.4, due to the truth of the Horn formula 
\[
\Lambda_k'\ =\ (\exists x_1)\ldots (\exists x_{k+1}) 
[\Delta_k(x_1, \ldots, x_{k+1})\wedge \Phi_{k+1}(x_1, \ldots, x_{k+1})
\wedge \Phi_{k+1}'(x_1, \ldots, x_{k+1})]\,,
\] 
${1\leq k=\dim_{\CM(R)_B} R_B\leq m(R)}$ in each $R^1_B$, 
${B\in \mathcal{U}(R)}$, there are  
${0\ne \eta_i\in B(R)}$, ${\eta_i \eta_j=\delta_{ij} \eta_i}$, 
${\Id_{R^1}=\eta_1+\ldots+\eta_l}$, $\Lambda_{k_i}'$ is true in $\eta_i R^1$, 
${\{k_i\}_{i=1}^l\subseteq \{\dim\limits_{K_B} R_B\mid B\in \mathcal{U}(R)\}}$, 
$k_0=0< k_1<\ldots < k_l=m(R)$, the $Z(\eta_i R^1)$--module 
${\eta_i R^1=\eta_i \CM(R)\oplus \eta_i R}$ has a system of generators 
$\{x_{i j}\}_{j=1}^{k_i+1}$, ${x_{i 1}=\eta_i}$ and  
${x_{i j}\in [\eta_i R, \eta_i R]_{k_i}=\eta_i [R, R]_{k_i}}$, 
${j=2, \ldots, k_i+1}$, through which its elements are uniquely expressed over 
${Z(\eta_i R^1)=\eta_i \CM(R)}$, 
$\eta_i R=\eta_i R^2=\sum\limits_{j=2}^{k_i+1} 
\eta_i \CM(R) x_{i j}=\sum\limits_{j=2}^{k_i+1} \CM(R) x_{i j}$, 
${R=R^2=\sum\limits_{s=2}^{m(R)+1} \CM(R) x_s}$,          
${x_s=\sum\limits_{j=i}^l x_{j s}}$, ${k_{i-1}+1\leq s-1\leq k_i}$, 
${i=1, \ldots, l}$.

If in each $R^1_B$, ${B\in \mathcal{U}(R)}$, the Horn formula 
\begin{multline*}
\Lambda_k''\ =\ (\exists x_1)\cdots (\exists x_{k+1})
(\exists \alpha_{22}^2)\ldots (\exists \alpha_{k+1 k+1}^{k+1})
[\Delta_k'(x_1, \ldots, x_{k+1}, 
\alpha_{22}^2, \ldots, \alpha_{k+1 k+1}^{k+1})\wedge
\\ 
\shoveright{
\Phi_{k+1}(x_1, \ldots, x_{k+1})\wedge \Phi_{k+1}'(x_1, \ldots, x_{k+1})\wedge 
\Sigma_k(\alpha_{22}^2, \ldots, \alpha_{k+1 k+1}^{k+1})]\,,}
\\
\shoveleft{
\Delta_k'(x_1, \ldots, x_{k+1}, 
\alpha_{22}^2, \ldots, \alpha_{k+1 k+1}^{k+1})\ =\ 
\Delta_k(x_1, \ldots, x_{k+1})\wedge 
\biggl\{\bigwedge_{i, j=2}^{k+1} \biggl(x_i x_j=\sum_{m=2}^{k+1} \alpha_{ij}^m 
x_m\biggr)\biggr\}\,,}
\\
\Sigma_k(\alpha_{22}^2, \ldots, \alpha_{k+1 k+1}^{k+1})\ =\ 
(Q_1 y_1)\ldots (Q_{n_k} y_{n_k}) 
\biggl[\biggl\{\bigwedge_{i, j, m=2}^{k+1} 
(\alpha_{ij}^m\in Z(R))\biggr\}\wedge T_1\wedge \ldots \wedge T_{l_k}\biggr]\,,
\end{multline*}
${k=\dim_{\CM(R)_B} R_B}$ is true, where ${Q_i\in \{\exists, \forall\}}$ and 
the formulas $T_j$ define the relations between $\{\alpha_{ij}^m\}$ \cite[the~proof~of~Theorem~3.1.11]{BD}, then $\Lambda_{k_i}''$ is true in $\eta_i R^1$, 
$\Sigma_{k_i}$ defines the relations between the structure constants for 
$\{x_{i j}\}_{j=2}^{k_i+1}$ in $\eta_i \CM(R)$, ${i=1, \ldots, l}$. At the 
same time, in $R^1$ the Horn formula 
\[
\Lambda_k'''\ =\ (\exists x_1)\ldots (\exists x_{k+1})
[\Delta_k(x_1, \ldots, x_{k+1})\wedge \Phi_{k+1}(x_1, \ldots, x_{k+1})\wedge 
\Phi_k''(x_2, \ldots, x_{k+1})]\,,
\]
${k=m(R)}$ is true, ${M(R)=\sum\limits_{
\genfrac{}{}{0pt}{1}{1\leq l\leq (k^{k^2+1}-1)/(k-1),}
{1\leq i_1, \ldots, i_l\leq k}} 
\CM(R) l_{x_{i_1}}\cdots l_{x_{i_l}}=O(M(R))}$ (Lemma 2.4, the end of the proof). 
\end{proof}

Let us consider separately the case of a semiprime $R$ with ${Z(R)\ne \{0\}}$ 
without $\mathcal{F}'$--torsion, ${\mathcal{F}'=\mathcal{E}(Z(R))}$. 
In particular, $R$ is without $\mathcal{F}'$--torsion if 
${I\cap Z(R)\ne \{0\}}$ for all ${\{0\}\ne I\lhd R}$, since for any 
${J\in \mathcal{F}'}$ 
\begin{gather*}
\Ann J\ =\ \Ann_l J\ =\ \Ann_r J\ =\ \{x\in R\mid x J=J x=\{0\}\}\lhd R\,,
\\
\{0\}\ =\ J (\Ann J\cap Z(R))\ =\ \Ann J\cap Z(R)\ =\ \Ann J\,.
\end{gather*}
Since ${Z(R)\subseteq Z(R)_{\mathcal{F}'}\subseteq Z(R_{\mathcal{F}'})}$ and 
for any ${x\in Z(R_{\mathcal{F}'})}$ 
${I x\subseteq R\cap Z(R_{\mathcal{F}'})=Z(R)}$ for some  
${I\in \mathcal{F}', x\in Z(R)_{\mathcal{F}'}, 
Z(R_{\mathcal{F}'})=Z(R)_{\mathcal{F}'}, 1\in R_{\mathcal{F}'}}$.
For ${R=P(R)}$ with ${Z(R)\ne \{0\}}$, the absence of $\mathcal{F}'$--torsion 
is equivalent to the absence of ${}_F \mathcal{F}'$--torsion, where 
\[
\mathcal{F}'\ =\ \mathcal{E}({}_{\CM(R)} Z(R))\ =\ 
\{\CM(R) I\mid I\in {}_F \mathcal{F}'\}\ \subseteq\ 
{}_F \mathcal{F}'\ =\ \mathcal{E}({}_F Z(R))\,.
\]
When they are satisfied,  
${Z(R) (\Ann_{\CM(R)} Z(R)) R=\{0\}}$, ${\Ann_{\CM(R)} Z(R)=\{0\}}$,
${Z(R)\in \mathcal{F}}$, 
$\mathcal{F}'=\{I\cap Z(R)\mid I\in \mathcal{F}\}\subseteq \mathcal{F}$, 
${\Hom(J, R)_{Z(R)}=\Hom(J, R)_{\CM(R)}}$ for all ${J\in \mathcal{F}'}$, 
\[
(\phi (\alpha x) -\alpha \phi x) J\ =\ \{0\}\,,\quad 
\phi (\alpha x)-\alpha \phi x\ =\ 0\quad 
(\phi\in \Hom(J, R)_{Z(R)},\ \alpha\in \CM(R),\ x\in J)\,,
\]
and every ${\psi\in \Hom(I, R)_{Z(R)}}$, ${I\in {}_F \mathcal{F}'}$, can be 
extended to ${\overline{\psi}\in \Hom(\CM(R) I, R)_{\CM(R)}}$, 
${\overline{\psi} \sum\limits_{i=1}^k \alpha_i x_i=
\sum\limits_{i=1}^k \alpha_i \psi x_i}$ for all 
${\alpha_i\in \CM(R)}$, ${x_i\in I}$, 
${k\geq 1}$, since ${\sum\limits_{i=1}^k \alpha_i x_i=0}$ implies 
\[
\biggl(\sum_{i=1}^k \alpha_i \psi x_i\biggr) y\ =\ 
\sum_{i=1}^k (\alpha_i y) \psi x_i\ =\ 
\psi\biggl(\sum_{i=1}^k (\alpha_i y) x_i\biggr)\ =\ 
\psi\biggl(\biggl(\sum_{i=1}^k \alpha_i x_i\biggr) y\biggr)\ =\ 0
\quad (y\in I)\,,
\]
${\biggl(\sum\limits_{i=1}^k \alpha_i (\psi x_i)\biggr) I=\{0\}}$, 
${\sum\limits_{i=1}^k \alpha_i (\psi x_i)=0}$. Thus, in this case 
${R_{\mathcal{F}'}=R_{\mathcal{F}}=O(R)}$ is isomorphic to 
$R_{{}_F \mathcal{F}'}$ as an $F$--algebra via  
${[(\psi, I)]\longmapsto [(\overline{\psi}, \CM(R) I)]}$, 
${\psi\in \Hom(I, R)_{Z(R)}}$, ${I\in {}_F \mathcal{F}'}$, the 
isomorphism of $R_{{}_F \mathcal{F}'}$ and $R_{\mathcal{F}'}$ over 
$F$ induces their isomorphism over 
\[
Z(R_{{}_F \mathcal{F}'})\ \cong\ Z(R_{\mathcal{F}'})\ =\ 
\CM(R_{\mathcal{F}'})\ =\ \CM(R) \Id_{R_{\mathcal{F}'}}\ =\ Q(Z(R))
\]
(Proposition 2.1; here ${Z(R)\lhd Q(Z(R))\cong \End(Z(R))_{Z(R)}}$).

\begin{prop}
If ${Z(I)=I\cap Z(R)\ne \{0\}}$ for all ${\{0\}\ne I\lhd R}$, then 
${R_{\mathcal{F}'}=P(R_{\mathcal{F}'})}$ and 
${\CM(R_{\mathcal{F}'})\cong Z(R_{\mathcal{F}'})=Z(R)_{\mathcal{F}'}}$.
\end{prop}

\begin{proof}
Let ${\psi\in \Hom(H, R_{\mathcal{F}'})_{M(R_{\mathcal{F}'})}}$, 
${H\lhd R_{\mathcal{F}'}}$, and without loss of generality 
${H\in \mathcal{E}(R_{\mathcal{F}'})}$. Then 
${H'=R\cap \psi^{-1} R\in \mathcal{E}(R)}$, ${H''=H'\cap Z(R)\in \mathcal{F}'}$, 
since for any ${\{0\}\ne I\lhd R}$ there exist 
${0\ne x\in (I)_{R_{\mathcal{F}'}}\cap H}$, ${K, K', K\cap K'\in \mathcal{F}'}$, 
${K x\subseteq I\cap H}$, ${K' \psi x\subseteq R}$, 
${\{0\}\ne (K\cap K') x\subseteq I\cap H'}$ (${K=T^n}$, $T$ is the intersection 
of ${T'\in \mathcal{F}'}$, ${T' y\subseteq R}$, for all 
${y\in R_{\mathcal{F}'}}$ in the expression of $x$ as a non-associative 
polynomial of the elements of $I$ and $R_{\mathcal{F}'}$, $n$ is the maximum 
number of occurrences of $\{y\}$ in the monomials of this expression) and 
\[
\{0\}\ =\ (H''\cap \Ann H'')^2\ =\ H''\cap \Ann H''\ =\ 
H'\cap \Ann H''\cap Z(R)\ =\ \Ann H''
\]
(semiprimeness of $Z(R)$, ${\Ann H''\lhd R}$, ${H'\in \mathcal{E}(R)}$). 
Due to ${\psi'=\psi|_{H'}\in \Hom(H', R)_{M(R)'}}$, 
${\psi''=\psi'|_{H''}\in \Hom(H'', R)_{Z(R)}}$, there are 
${x\in R_{\mathcal{F}'}}$, ${\psi'': h\longmapsto h x}$, ${h\in H''}$
(see the beginning of this Sec.), 
\[
h \psi' g\ =\ \psi'(h g)\ =\ (\psi' h) g\ =\ (\psi'' h) g\ =\ 
(h x) g\ =\ h (x g)\quad (g\in H',\ h\in H'')\,,
\]
${\psi': g\longmapsto x g}$, ${g\in H'}$, and for any ${z\in H}$
there are ${J, J'\in \mathcal{F}'}$, ${J z, J' \psi z\subseteq R}$, 
\[
a \psi z\ =\ \psi(a z)\ =\ \psi'(a z)\ =\ 
a x z\quad (a\in J\cap J')\,,
\]
${\psi: z\longmapsto x z}$, ${z\in H}$, where for any ${z, z'\in H}$, 
${y\in R_{\mathcal{F}'}}$
\begin{gather*}
x (z y)\ =\ (x z) y\,,\quad x (y z)\ =\ y (x z)\,,\quad 
0\ =\ [x, z] z'\ =\ (x z) z'-(z x) z'\ =\ [x, z]\,,
\\
x (y z)\ =\ (y z) x\ =\ y (z x)\,,\quad (z x) z'\ =\ x (z z')\ =\ z (x z')
\end{gather*}
(${H''\in \mathcal{F}'}$, $R_{\mathcal{F}'}$ is without $\mathcal{F}'$--torsion), 
${f(x, z_2, \ldots, z_8)=0}$ for all ${z_i\in H}$, $f$ from Remark 2.3, 
${Z(H) x\subseteq Z(H)}$. For any ${y\in Z(H)}$ there is ${J\in \mathcal{F}'}$, 
\begin{gather*}
J y\ \subseteq\ R\cap Z(H)\ \subseteq\ 
Z(H\cap R)\ =\ H\cap R\cap Z(R)\ =\ H\cap R\cap Z(R_{\mathcal{F}'})\,,
\\
\{0\}\ =\ J f(y, x_2, \ldots, x_8)\ =\ f(y, x_2, \ldots, x_8)\quad 
(x_i\in R_{\mathcal{F}'})\,,
\end{gather*}
${y\in Z(R_{\mathcal{F}'})}$. Therefore ${Z(H)=H\cap Z(R_{\mathcal{F}'})}$,
\[ 
\{0\}\ =\ Z(H) f(x, x_2, \ldots, x_8)\ =\ f(x, x_2, \ldots, x_8)\quad 
(x_i\in R_{\mathcal{F}'})\,,
\]
${x\in Z(R_{\mathcal{F}'})}$, ${\psi=x \Id_H}$. Thus, 
${R_{\mathcal{F}'}=P(R_{\mathcal{F}'})}$, 
${\CM(R_{\mathcal{F}'})\cong Z(R_{\mathcal{F}'})}$.
If $R$ is prime, then $R_{\mathcal{F}'}=P(R)=R S^{-1}$ is a central simple 
over the field ${Z(R)_{\mathcal{F}'}=Z(R)_{\mathcal{F}'}=Q(Z(R))=Z(R) S^{-1}}$, 
${S=Z(R)\setminus \{0\}}$. 
\end{proof}

The condition of Proposition 2.7 is satisfied, for example, in all non-zero 
non-degenerate alternative $PI$--algebras \cite{Mar, Row}, \cite[Lemma~7,~Theorems~3,~7,~p.~224,~212,~225]{ZShS}. 
                            
If $R$ is a semiprime subalgebra of the algebra $Q$ and for any ${0\ne q\in Q}$ 
there is ${I\in \mathcal{E}(R)}$, ${\{0\}\ne (q)_R I+I (q)_R\subseteq R}$, 
where ${(q)_R=q M^Q(R)'}$, then, as in \cite[Theorem~3.10]{GolO}, we will call 
$Q$ an \emph{algebra of $\mathfrak{M}$--quotients of $R$}. Such $Q$ inherits 
the semiprimeness of $R$, ${Q\hookrightarrow I(R)}$ identically on $R$ and 
${(q)_R I+I (q)_R\ne \{0\}}$ for all ${0\ne q\in Q}$, ${I\in \mathcal{E}(R)}$ 
($R$ is an essential $M^Q(R)'$--submodule of $Q$; ${\{0\}\ne (q)_R\cap R\lhd R}$). 
In case ${N(R)\subseteq N(Q)}$, ${Z(R)\subseteq Z(Q)}$, since for any 
${z\in Z(R)}$, ${q\in Q}$, ${I\in \mathcal{E}(R)}$, 
${(q)_R I+I (q)_R\subseteq R}$, 
\begin{gather*}
[z, q] t_x\ =\ (z q-q z) t_x\ =\ [z, q t_x]\quad (x\in R,\ t=l, r)\,,
\\
([z, q])_R I+I ([z, q])_R\ =\ [z, (q)_R] I+I [z, (q)_R]\ =\ 
[z, (q)_R I+I (q)_R]\ =\ \{0\}\,,\quad [z, q]\ =\ 0\,.
\end{gather*}

\begin{prop}
If $R$ is a semiprime subalgebra of the algebra $Q$, 
${I\cap Z(R)\ne \{0\}}$ for all ${\{0\}\ne I\lhd R}$, for any 
${0\ne q\in Q}$ there is ${J\in \mathcal{E}(R)}$, 
${\{0\}\ne q J\subseteq R}$ or (and) ${\{0\}\ne J q\subseteq R}$, 
and ${Z(R)\subseteq Z(Q)}$, then $Q$ is an algebra of 
$\mathfrak{M}$--quotients of $R$, 
${Q_{\mathcal{F}'_Q}\cong Q_{\mathcal{F}'}=R_{\mathcal{F}'}}$,  
${\mathcal{F}'_Q=\mathcal{E}(Z(Q))}$, 
${P(Q_{M(R)'})=P(Q)}$, ${\CM(Q)\cong \CM(R)}$
and, in particular, ${O(R)=O(Q)}$ for ${R=P(R)}$.
\end{prop}

\begin{proof}
Due to the observations before Proposition 2.7, $R$ has no   
$\mathcal{F}'$--torsion, for any ${H\in \mathcal{F}'}$
\[
\{0\}\ =\ \Ann(R, H)\ =\ R\cap \Ann(Q, H)\ =\ \Ann(Q, H)\ =\ 
\Ann_t(R, (H)_R)\quad (t=l, r)\,,
\]
$(H)_R=H R^1\in \mathcal{E}(R)$,
$Q$ is without $\mathcal{F}'$--torsion. If 
${I\in \mathcal{E}(R)}$, ${K=\Ann(R, I\cap Z(R))}$, then
${\{0\}=(I\cap K\cap Z(R))^2=I\cap K\cap Z(R)=I\cap K=K}$ 
(semiprimeness of $Z(R)$, ${K\lhd R}$ and ${I\cap K\cap Z(R)\lhd Z(R)}$), 
${I\cap Z(R)\in \mathcal{F}'}$, ${\Ann_t(Q, I)\subseteq \Ann(Q, I\cap Z(R))=\{0\}}$, 
$t=l, r$. As a consequence, for any ${0\ne q\in Q}$, ${I\in \mathcal{E}(R)}$, 
${q I\subseteq R}$ (${I q\subseteq R}$) 
\[
\{0\}\ \ne\ (I\cap Z(R)) q\ \subseteq\
(q)_R (I\cap Z(R))_R+(I\cap Z(R))_R (q)_R\ \subseteq\ R\,,
\]
${r_q|_{I\cap Z(R)}\in \Hom(I\cap Z(R), R)_{Z(R)}}$, $Q$ is an algebra 
of $\mathfrak{M}$--quotients of $R$, $Q\subseteq R_{\mathcal{F}'}\subseteq 
Q_{\mathcal{F}'}\subseteq (R_{\mathcal{F}'})_{\mathcal{F}'}$, 
${Q_{\mathcal{F}'}=R_{\mathcal{F}'}}$ (up to isomorphism). Since 
\[
\{0\}\ \ne\ (x M(Q)'\cap Q) H\ \subseteq\ (x H) M(Q)'\cap Q\quad 
(0\ne x\in P(Q),\ H\in \mathcal{F}')\,, 
\]
$P(Q)$ has no $\mathcal{F}'$--torsion, for any 
${\phi\in \Hom(M, P(Q))_{M^Q(R)'}=\Hom(M, P(Q))_{M(R)'}}$, 
$M_{M(R)'}\subseteq P(Q)_{M(R)'}$ 
(${M(R)'\cong M^{Q_{\mathcal{F}'}}(R)'\cong M^Q(R)'}$) 
from ${\sum\limits_{i=1}^k a_i \psi_i=0}$, ${a_i\in M}$, 
${\psi_i\in M(Q)'}$, ${k\geq 1}$, it follows that
\[
0\ =\ \phi\biggl(g \sum_{i=1}^k a_i \psi_i\biggr)\ =\ 
g \sum_{i=1}^k (\phi a_i) \psi_i\ =\ 
\sum_{i=1}^k (\phi a_i) \psi_i\quad (g\in G)
\]
for ${G\in \mathcal{F}'}$, ${G \psi_i\in M^Q(R)'}$ (${G=T^n}$, $T$ is the 
intersection of ${T'\in \mathcal{F}'}$, ${T' q\subseteq R}$, for all 
${q\in Q}$ in the expressions of $\{\phi_i\}$ as associative polynomials 
of $t_q$, ${t=l, r}$, $n$ is the maximum number of occurrences of 
$\{t_q\}$ in the monomials of these expressions) and hence $\phi$ 
can be correctly extended to 
${\hat{\phi} \Id_{M M(Q)'}\in \Hom(M M(Q)', P(Q))_{M(Q)'}}$, 
${\hat{\phi}\in \CM(Q)}$, 
${\hat{\phi}\Bigl(\sum\limits_{i=1}^k a_i \psi_i\Bigr)=
\sum\limits_{i=1}^k (\phi a_i) \psi_i}$, ${a_i\in M}$, ${\psi_i\in M(Q)'}$. 
Thus, 
\[
P(Q)\ =\ \End(P(Q))_{M(R)'} Q\ =\ P(Q_{M(Q)'})\ =\ \End(I(R))_{M(R)'} Q\,,
\]
${\CM(Q)=\End(P(Q))_{M(R)'}\cong \CM(R)}$, 
the last isomorphism is ${\beta\longmapsto \beta|_{P(R)}, 
\beta\in \CM(Q)}$ (if ${\beta R=\{0\}}$, then ${\beta=0}$, 
due to ${\beta(H q)=H (\beta q)=\{0\}}$, 
${\beta q=0}$ for all ${q\in Q}$, ${H\in \mathcal{F}'}$, 
${H q\subseteq R}$).  

If ${E\in \mathcal{F}'_Q}$, ${S=\Ann(R, E\cap R)}$, then 
${\{0\}=J q (S\cap Z(R))=q (S\cap Z(R))}$ for all ${q\in E}$, 
${J\in \mathcal{F}'}$, ${J q\subseteq E\cap R}$, 
${\{0\}=S\cap Z(R)=S}$, ${E\cap R=E\cap Z(R)\in \mathcal{F}'}$ 
(${\Ann_{Z(Q)} E=\{0\}}$, $Q$ has no $\mathcal{F}'$--torsion), 
${\Ann(Q, E)\subseteq \Ann(Q, E\cap R)=\{0\}}$, $Q$ has no 
$\mathcal{F}'_Q$--torsion. If ${H\in \mathcal{F}'}$, then 
${\Ann(Q, (H)_{Z(Q)})=\Ann(Q, H)=\{0\}}$, 
${(H)_{Z(Q)}=H Z(Q)^1\in \mathcal{F}'_Q}$. 

If ${I\in \mathcal{F}'_Q}$, ${\phi\in \Hom(I, Q)_{Z(Q)}}$, then 
${I''=R\cap \phi^{-1} R=\{x\in I\cap R\mid \phi x\in R\}\in \mathcal{F}'}$, 
$\phi|_{I''}\in \Hom(I'', R)_{Z(R)}$, since for ${U=\Ann(R, I'')}$ and any 
${q\in I}$ there are ${V, V'\in \mathcal{F}'}$, 
${V q, V' \phi q\subseteq R}$, ${(V\cap V') q\subseteq I''}$, 
${\{0\}=(V\cap V') q (U\cap Z(R))=q (U\cap Z(R))=U\cap Z(R)=U}$. 

Any ${\psi\in \Hom(W, R)_{Z(R)}}$, ${W\in \mathcal{F}'}$, extends  
to ${\overline{\psi}\in \Hom((W)_{Z(Q)}, Q)_{Z(Q)}}$,
$\overline{\psi} \Bigl(\sum\limits_{i=1}^k x_i z_i\Bigr)=
\sum\limits_{i=1}^k (\psi x_i) z_i$,
${x_i\in J}$, ${z_i\in Z(Q)^1}$, ${k\geq 1}$. Since for 
${\sum\limits_{i=1}^k x_i z_i=0}$ and ${N\in \mathcal{F}'}$, 
$N z_i\subseteq R\cap Z(Q)=Z(R)$, ${i=1, \ldots, k}$, 
\[
0\ =\ \psi \biggl(a \sum_{i=1}^k x_i z_i\biggr)\ =\ 
a \sum_{i=1}^k (\psi x_i) z_i\ =\ \sum_{i=1}^k (\psi x_i) z_i
\quad (a\in N)\,,
\]
such an extension is correct. Therefore, 
${Q_{\mathcal{F}'_Q}\cong Q_{\mathcal{F}'}}$ via mutually inverse 
$F$--isomorphisms ${[(\phi, I)]\longmapsto [(\phi|_{I''}, I'')]}$ 
and ${[(\psi, W)]\longmapsto [(\overline{\psi}, (W)_{Z(Q)})]}$, 
${I\in \mathcal{F}'_Q}$, ${\phi\in \Hom(I, Q)_{Z(Q)}}$, 
${W\in \mathcal{F}'}$, ${\psi\in \Hom(W, R)_{Z(R)}}$,  
${h \phi q=\phi|_{I''} (h q)=\overline{\phi|_{I''}} (h q)=
h \overline{\phi|_{I''}} q}$ for all ${q\in (I'')_{Z(Q)}}$, 
${h\in H\in \mathcal{F}'}$, ${H q\subseteq I''}$, 
${\phi|_{(I'')_{Z(Q)}}=\overline{\phi|_{I''}}}$,
${[(\phi, I)]=[(\overline{\phi|_{I''}}, (I'')_{Z(Q)})]}$. 
For ${R=P(R)}$ one has $Q_{\mathcal{F}'}=O(R)=P(O(R))$, 
${Z(O(R))=\CM(R) \Id_{O(R)}=Q(Z(R))=\CM(O(R))}$ 
(${Q_{\mathcal{F}'}=R_{\mathcal{F}'}}$, Proposition 2.1 and the 
observation before Proposition 2.7), 
\[
h (\beta q)\ =\ \beta|_R (h q)\ =\ h (\beta|_R q)\,,\quad 
\beta q\ =\ \beta|_R q\quad (\beta\in \CM(Q),\ q\in Q,\ 
h\in H\in \mathcal{F}',\ H q\subseteq R)\,,
\]
$\beta|_R q$ is the product of ${\beta|_R\in \CM(R)}$ and $q$ in the 
$\CM(R)$--algebra $O(R)$, $P(Q)=\CM(R) Q\subseteq O(R)$, 
${\CM(Q)=\End(P(Q))_{M(Q)'}=\CM(R) \Id_{P(Q)}}$ (Proposition 2.1, 
\cite[Proposition~2.4\linebreak (4)]{GolO}), $O(Q)$ in $I(P(Q))$ can be identified 
with ${O(Q)=O(R)}$ in $O(R)$. If $Q$ is a $\CM(R)$--algebra, then 
${Q=P(Q)}$, ${Z(R)\subseteq Z(Q)}$ automatically (${Z(R)\lhd \CM(R)}$).
\end{proof}

Under the conditions of Proposition 2.8, its conclusion applies, in 
particular, to ${Q=R_{\mathcal{F}'}}$. In the general case, the 
orthogonal completion and the algebra of quotients of a semiprime 
algebra are different. So, for example, if $R$ is a simple Lie 
algebra with exterior derivations and $Q_m(R)$ is its maximal 
algebra of quotients \cite{Mol, GolO}, then ${R=P(R)=O(R)}$ 
and $R\cong \ad(R)\subsetneqq Q_m(R)=\Der(R)$. 

\begin{lemma}
If an algebra ${R=P(R)}$ is a free $\CM(R)$--module with basis 
$\{x_a\}_{a\in A}$, then $R$ is a subdirect product of algebras 
$R_B$, ${B\in \mathcal{U}(R)}$, 
\[
R\ \hookrightarrow\ C R\ \subseteq\ R'\ =\ 
\prod_{B\in \mathcal{U}(R)} R_B\,,\quad 
C R\ \hookrightarrow \overline{C} R\ \subseteq\ 
R''\ =\ \prod_{B\in \mathcal{U}(R)} 
(\overline{\CM(R)_B}\mathbin{\otimes_{\CM(R)_B}} R_B)\,, 
\]
where ${C=\prod\limits_{B\in \mathcal{U}(R)} \CM(R)_B}$ and  
${\overline{C}=\prod\limits_{B\in \mathcal{U}(R)} 
\overline{\CM(R)_B}}$ are direct products of fields $\CM(R)_B$, 
${B\in \mathcal{U}(R)}$, and their algebraic closures, 
$C$--algebra ${C R}$ and $\overline{C}$--algebra 
${\overline{C} R}$ are free modules over $C$ and $\overline{C}$ 
with basis $\{x_a\}_{a\in A}$, 
${C R=R'}$ and ${\overline{C} R=R''}$ for ${|A|< \infty}$.  
\end{lemma}

\begin{proof}
Since the proper ideals of 
\begin{gather*}
\Spec(\CM(R))\ =\ \{P \CM(R)\mid P\in \Spec(B(R))\}\,,
\\
\Spec(B(R))\ =\ \{B(R)\setminus B\mid B\in \mathcal{U}(R)\}\cup \{B(R)\}\ =\ 
\{B(R)\cap Q\mid Q\in \Spec(\CM(R))\}
\end{gather*}
are the maximal ideals of $\CM(R)$, $B(R)$ \cite[the~observations~before~Remark~2.2~and~Lemma 2.5]{GolO}, $\CM(R)$ is a subdirect product of the fields 
$\CM(R)_B$, ${B\in \mathcal{U}(R)}$, $R_B$ is a $\CM(R)_B$--space with basis 
${\{x_{a, B}=x_a+P R\}_{a\in A}}$, where ${x_{a, B}\ne 0}$ for all ${a\in A}$, 
${B\in \mathcal{U}(R)}$, ${P=B(R)\setminus B}$, 
\[
\bigcap_{P\in \Spec(B(R))} P R\ =\ \bigcap_{P\in \Spec(B(R))} 
\Bigl(\sum_{a\in A} P \CM(R) x_a\Bigr)\ =\ 
\sum_{a\in A} \biggl(\bigcap_{P\in \Spec(B(R))} P \CM(R)\biggr) x_a\ =\ \{0\}
\]
$R$ is a subdirect product of $R_B$, ${B\in \mathcal{U}(R)}$. We will 
identify the elements of ${S=\CM(R), R}$ with the functions 
${s: B\longmapsto s_B=s+P S}$, ${s\in S}$, ${B\in \mathcal{U}(R)}$, $\CM(R)$ 
and $R$ with the $F$--subalgebra $C$ and the $\CM(R)$--subalgebra of the 
direct product $R'$ of the algebras $R_B$, ${B\in \mathcal{U}(R)}$ 
(${\phi x: B\longmapsto \phi_B x_B}$, ${\phi\in \CM(R)}$, ${x\in R}$, 
${B\in \mathcal{U}(R)}$), $R_B$ with the $\CM(R)_B$--subalg\-ebra 
${1\otimes R_B}$ of its scalar extension 
${\overline{\CM(R)_B}\mathbin{\otimes_{\CM(R)_B}} R_B=
\overline{\CM(R)_B} R_B}$ over $\overline{\CM(R)_B}$, $B\in \mathcal{U}(R)$, 
$R'$ with the $C$--subalgebra of the $\overline{C}$--algebra $R''$. Then  
${C R=\sum\limits_{a\in A} C x_a\subseteq 
\overline{C} R=\sum\limits_{a\in A} \overline{C} x_a}$, 
${C R=R'}$ and ${\overline{C} R=R''}$ for ${|A|< \infty}$, the $C$--subalgebra 
${C R\subseteq R'}$ and the $\overline{C}$--subalgebra 
${\overline{C} R\subseteq R''}$ are free modules over $C$ and $\overline{C}$ 
with basis $\{x_a\}_{a\in A}$ ($\{x_{a, B}\}_{a\in A}$ is a basis of $R_B$ 
over $\CM(R)_B$ and $\overline{\CM(R)_B} R$ over 
$\overline{\CM(R)_B}$, ${B\in \mathcal{U}(R)}$). 
\end{proof}

Before discussing decompositions of orthogonally complete non-degenerate 
algebras, we make two useful observations. Let $R$ be a subdirect product 
of semiprime $\{R_a\}_{a\in A}$, $R'=\prod\limits_{a\in A} R_a$ be their 
direct product, $O'(R)$ be the orthogonal completion of ${R\subseteq R'}$ 
in $O(R')$ and ${\pi_a\in \End(R')_{M(R')'}\cap B(R')}$ 
(${\End(R')_{M(R')'}\hookrightarrow \CM(R')}$), 
${\pi_a|_{R_b}=\delta_{a, b} \Id_{R_b}}$, ${a, b\in A}$. Since 
$R'$ is an essential $M(R')'$--submodule of $P(R')$, $I(R')$, 
${\pi_a \pi_b R'=\{0\}}$, ${\pi_a \pi_b=0}$ for all ${a\ne b\in A}$, 
and for ${\alpha\in \Ann_{\CM(R')} \{\pi_a\}_{a\in A}}$ from 
${\pi_a (R'\cap \alpha P(R'))=\{0\}}$ for all ${a\in A}$ it follows 
that ${\{0\}=R'\cap \alpha P(R')=\alpha P(R')}$, ${\alpha=0}$. Hence 
$\{\pi_a\}_{a\in A}$ is dense and orthogonal in $B(R')$ and, due to 
${\pi_a R=R_a\subseteq B(R') R\subseteq O'(R)}$, ${a\in A}$, 
\[
R'\ =\ \biggl\{{\sum\limits_{a\in A}}^{\perp} \pi_a r_a\biggl| 
r_a\in R_a,\ a\in A\biggr\}\subseteq O'(R)\,,\quad O(R')\ =\ O'(R)\,.
\]
At the same time, ${O(R')=\prod\limits_{a\in A} \pi_a O(R')}$, 
${\pi_a O(R')=O'(R_a)\cong O(R_a)}$, ${a\in A}$, since 
\begin{multline*}
\Hom(M, \pi_a P(R'))_{M(R')'}\ =\ \Hom(M, \pi_a P(R'))_{M(R_a)'}\ =
\\ 
\pi_a \Hom(M, P(R'))_{M(R')'}\ =\ \pi_a \CM(R') \Id_M\quad 
(M_{M(R')'}\subseteq \pi_a P(R'))\,,
\\
\shoveleft{
\pi_a I(P(R')_{M(P(R'))'})\ =\ I(\pi_a P(R')_{M(P(R'))'})\ =\ 
I(\pi_a P(R')_{\pi_a M(P(R'))'})\ =}
\\ 
\shoveright{
I(\pi_a P(R')_{M(\pi_a P(R'))'})\ \cong\ I(P(R_a)_{M(P(R_a))'})\,,}
\\
\shoveleft{
\End(\pi_a P(R'))_{M(P(R'))'}\ =\ \End(\pi_a P(R'))_{M(R')'}\ =}
\\
\shoveright{\End(\pi_a P(R'))_{M(R_a)'}\ =\  
\pi_a \CM(R')\ \cong\ \CM(R_a)\,,}
\\
\shoveleft{\pi_a P(R')\ =\ P(\pi_a P(R')_{\pi_a M(P(R'))'})\ =\ 
P(\pi_a P(R')_{M(R_a)'})\ =}
\\ 
\CM(R') \pi_a R'\ =\ \CM(R') R_a\ \cong\ P(R_a)
\end{multline*}
for all ${a\in A}$ \cite[the~observations~before~Lemma~2.6]{GolO} taking into 
account the essentiality of $M(R_a)'$--submodule ($M(R')'$--submodule)  
${R'\cap \pi_a P(R')=R_a}$ in $\pi_a P(R')$, 
\begin{gather*}
P(R')\ =\ \prod_{a\in A} \pi_a P(R')\ \cong\ 
\prod_{a\in A} P(R_a)\,,\quad 
\CM(R')\ =\ \prod_{a\in A} \pi_a \CM(R')\ \cong\ 
\prod_{a\in A} \CM(R_a)\,,
\\
I(P(R'))\ =\ \prod_{a\in A} \pi_a I(P(R'))\ \cong\ 
\prod_{a\in A} I(P(R_a)_{M(P(R_a))'})
\end{gather*}
\cite[Remark~2.2]{GolO} for a dense $\{\pi_a\}_{a\in A}$ 
(${x={\sum\limits_{a\in A}}^{\perp} \pi_a x}$ for all ${x\in I(P(R'))}$).

Of particular interest is the case of strongly prime $\{R_a\}_{a\in A}$ 
and $\{O(R_a)\}_{a\in I}$ \cite[Proposition~2.4,~the~observations~after~Lemma~2.6]{GolO}.

Now let ${R=P(R)}$, ${\bigcap\limits_{a\in A} P_a=\{0\}}$ for some 
${R\ne P_a\in \Spec(R)}$, ${P=\bigcap\limits_{a\in B} P_a}$, 
${P'=\bigcap\limits_{a\in C} P_a}$, ${\{0\}=P\cap P'\ne P, P'}$, 
${A=B\sqcup C}$, ${\alpha\in B(R)}$, ${\alpha|_P=\Id_P}$, ${\alpha P'=\{0\}}$. 
Since for any ${a\in A}$ 
\[
P_a\ =\ \alpha P_a\oplus (\Id_R-\alpha) P_a\,,\quad 
R/P_a\ \cong\ 
(\alpha R/\alpha P_a)\oplus ((\Id_R-\alpha) R/(\Id_R-\alpha) P_a)\,, 
\]
either 
\[
\alpha R\ =\ \alpha P_a\,,\quad 
P_a\ =\ \alpha R\oplus (\Id_R-\alpha) P_a\,,\quad 
R/P_a\ =\ (\Id_R-\alpha) R/(\Id_R-\alpha) P_a\,,
\]
or 
\[
(\Id_R-\alpha) R\ =\ (\Id_R-\alpha) P_a\,,\quad 
P_a\ =\ \alpha P_a\oplus (\Id_R-\alpha) R\,,\quad 
R/P_a\ =\ \alpha R/\alpha P_a\,.
\]
Therefore, 
\begin{multline*}
P\ =\ \alpha P\ =\ 
\bigcap_{a\in B,\ \alpha P_a\ne \alpha R} \alpha P_a\ =\ 
\bigcap_{a\in B,\ P_a=\alpha P_a\oplus (\Id_R-\alpha) R} \alpha P_a\ =
\\ 
\biggl(\bigcap_{a\in B,\ P_a=\alpha P_a\oplus (\Id_R-\alpha) R} 
P_a\biggr)\cap 
\biggl(\bigcap_{a\in B,\ P_a=\alpha R\oplus (\Id_R-\alpha) P_a} P_a\biggr)
\ =
\\ 
(P\oplus (\Id_R-\alpha) R)\cap 
\biggl(\alpha R\oplus 
\bigcap_{a\in B,\ (\Id_R-\alpha) P_a\ne (\Id_R-\alpha) R} 
(\Id_R-\alpha) P_a\biggr)
\end{multline*}
and hence ${\{a\in B\mid (\Id_R-\alpha) P_a\ne (\Id_R-\alpha) R\}\ne \emptyset}$, 
\[
\bigcap_{a\in B,\ (\Id_R-\alpha) P_a\ne (\Id_R-\alpha) R} 
(\Id_R-\alpha) P_a\ =\ \{0\}\,,\quad 
\alpha R\ =\ \bigcap_{a\in B,\ (\Id_R-\alpha) P_a\ne (\Id_R-\alpha) R} P_a\,,
\]
${(\Id_R-\alpha) R}$ is a subdirect product of $R/P_a$, ${a\in B}$, 
${(\Id_R-\alpha) P_a\ne (\Id_R-\alpha) R}$. Similarly, $\alpha R$ is 
a subdirect product of $R/P_a$, ${a\in C}$, ${\alpha P_a\ne \alpha R}$. 

\section{Orthogonally complete non-degenerate algebras}

Following \cite{ZShS}, we denote by $U(R)$ and $D(R)$ the 
\emph{associative kernel} and \emph{associator ideal} of the 
algebra $R$, ${U(R)=\{x\in N(R)\mid x R\subseteq N(R)\}}$ 
is the largest of ${I\lhd R}$, ${I\subseteq N(R)}$, and 
\begin{gather*}
D(R)\ =\ ((R, R, R))_R\ =\ R^1 (R, R, R)\ =\ (R, R, R) R^1\,,
\\
D(R) U(R)\ =\ U(R) D(R)\ =\ (U(R)\cap D(R))^2\ =\ \{0\}\,,
\end{gather*}
${R^1=F\oplus R}$ \cite[Propositions~8,~9,~10,~p.~203,~204]{ZShS}, and 
call $R$ with ${U(R)=\{0\}}$ \emph{purely non-\linebreak associative} 
(\emph{purely alternative} for the alternative $R$). 

If ${R\ne \{0\}}$ is semiprime, then 
${D(R)\cap U(R)=\{0\}}$, there exists 
${\tau\in B(R)}$, ${\tau|_{U(R)}=\Id_{U(R)}}$, ${\tau D(R)=\{0\}}$, and in 
the case ${\tau R\subseteq R}$
\begin{gather*}
U(\tau R)\ =\ \tau U(R)\ =\ U(R)\,,\quad 
U((\Id_{P(R)}-\tau) R)\ =\ (\Id_{P(R)}-\tau) U(R)\ =\ \{0\}\,,
\\
D(\tau R)\ =\ \tau D(R)\ =\ \{0\}\,,\quad 
D((\Id_{P(R)}-\tau) R)\ =\ (\Id_{P(R)}-\tau) D(R)\ =\ D(R)\,,
\\
N(R)\ =\ \tau R\oplus N((\Id_{P(R)}-\tau) R)\ =\ 
\tau R\oplus (\Id_{P(R)}-\tau) N(R)\,,
\end{gather*}
${R=\tau R\oplus (\Id_{P(R)}-\tau) R}$ for semiprime associative 
${\tau R=U(R)}$ and purely non-associative ${(\Id_{P(R)}-\tau) R}$. Since  
\begin{gather*}
D(P(R))\ =\ \CM(R) D(R)\,,\quad \CM(R) N(R)\ \subseteq\ N(P(R))\,,\quad 
\CM(R) U(R)\ \subseteq\ U(P(R))\,,
\\
D(R)\ \subseteq\ D(O(R))\ \subseteq\ O(D(R))\,,\quad 
O(N(R))\ \subseteq\ N(O(R))\,,\quad O(U(R))\ \subseteq\ U(O(R))\,,
\\
U(O(R)) O(P(R))\subseteq N(O(P(R)))\,,\quad U(O(R))\subseteq U(O(P(R)))\,,
\end{gather*}
where $O(D(R))$, $O(N(R))$ are defined in $O(R)$, 
${O(N(R))=N(O(R))}$, ${O(U(R))=U(O(R))}$ if ${B(R) R\subseteq R}$, 
the decompositions into direct sums of semiprime associative and 
purely non-associative ideals of ${S=\tau S\oplus (\Id_{P(R)}-\tau) S}$, 
${\tau\in B(R)}$, ${\tau|_{U(S)}=\Id_{U(S)}}$, ${\tau D(S)=\{0\}}$, 
${S=P(R), O(R), O(P(R))}$, induce decompositions into direct sums of such 
ideals of $R_{\tau}=\tau R\oplus (\Id_{P(R)}-\tau) R$ 
(Proposition 2.1; the decomposition of $O(R)$ is induced by the 
decom\-position of $O(P(R))$; the components of the decompositions are 
uniquely defined (if ${I=N(I)}$, ${J\lhd R=I\oplus J}$, ${U(J)=\{0\}}$, then 
${U(R)=I\oplus (J\cap U(R))=I\oplus U(J)=I}$ and 
$I\cap \Ann_t I\subseteq \prr(I)=\{0\}$, ${J=\Ann_t I=\Ann I}$, ${t=l, r}$)). 
From $R$ from a variety with the non-degeneracy condition, the non-degeneracy 
of $S$ ($R$ if its non-degeneracy implies the non-degeneracy of $P(R)$) is 
inherited by the ideals of these decompositions, taking into account 
${\pi P(R)\cong P(\pi R)}$, ${P(R)=P(R)_{\pi}\cong P(R_{\pi})}$, 
${O(R)=O(R)_{\pi}=O(R_{\pi})}$, up to isomorphism ${\pi O(R)=O(\pi R)}$ in 
$I(P(R))$ is the orthogonal completion of $\pi R$ in $I(P(\pi R))$, 
${\pi\in B(R)}$ (the observations at the end of Sec. 2 and in \cite[before~and~after~Lemma~2.6]{GolO}). 

The non-degeneracy of the alternative (linear Jordan over $F$ with $1/2$) $R$ 
is inherited by\linebreak its ideals, $P(R)$ and $O(R)$ 
\cite[the~observations~after~Lemma~2.6]{GolO}, 
\cite[Lemma~7,~Corollary 2,~Lemma~3,~p.~224,~377,~352]{ZShS}. 
If $R$ is alternative, non-degenerate and non-associative, then 
${Z(R)\ne \{0\}}$ \cite[Theorem~7,~p.~225]{ZShS}. If $R$ is semiprime and 
purely alternative, then $Z(R)=N(R)$, ${S=P(R), O(R), O(P(R))}$ is 
semiprime and purely alternative \cite[Theorem~11,~p.~205]{ZShS} 
(see above, ${R\cap U(S)\subseteq U(R)=\{0\}}$, 
${U(O(R))\subseteq U(O(P(R)))=O(U(P(R)))=\{0\}}$).

Omitting the details from \cite{MLQA, GolO}, we will briefly discuss 
on an alternative generalization of the construction 
of the maximal right algebra of quotients of an associative algebra. 
If $R$ and $Q$ are alternative, $R$ is a subalgebra of $Q$, 
${N(R)\subseteq N(Q)}$ and for any ${x, y\in Q}$, ${x\ne 0}$, there is 
${z\in N(R)}$, ${x z\ne 0}$, ${y z\in R}$, then $Q$ is a 
\emph{right algebra of quotients of $R$}. A right ideal $I$ in $R$ is 
called \emph{dense} if for any ${x, y\in R}$, ${x\ne 0}$, there is 
${z\in N(R)}$ ${(z\in N(I)=I\cap N(R))}$, ${x z\ne 0}$, ${y z\in I}$.
We choose the sets $\mathcal{D}_r(R)$ and $\mathcal{D}_{nr}(R)$ of dense 
right ideals of $R$ and right ideals $J$ of $N(R)$ such that for any 
${0\ne x\in R}$, ${y\in N(R)}$, there is ${z\in N(R)}$ (${z\in J}$), 
${x z\ne 0}$, ${y z\in J}$, 
\[
\mathcal{D}_r^*(R)\ =\ \{N(J) R^1\mid J\in \mathcal{D}_r(R)\}\ =\ 
\{J R^1\mid J\in \mathcal{D}_{nr}(R)\}\subseteq \mathcal{D}_r(R)
\]
and for all ${I\in \mathcal{D}_r^*(R)}$ 
\begin{multline*} 
\Hom^*(I, R)_{N(R)}\ =
\\ 
\Bigl\{\phi\in \Hom(I, R)_{N(R)}\Bigl| \phi(y x)=(\phi y) x,\ \phi[y, z]\in N(R)\ 
\forall x\in R,\ y, z\in N(I)\Bigr\}\,.
\end{multline*}
Any $R$ with ${\Ann_l N(R)=\{0\}}$, ${\prr(D(R))=\{0\}}$ or $D(R)$ without 
2--torsion has right algebras of quotients (in particular, the non-degenerate 
$R$ has left and right algebras of quotients), they are embedded identically 
on $R$ in its complete (maximal) right algebra of quotients $Q(R)$. The latter 
is uniquely defined up to an isomorphism identical on $R$, and can be 
constructed as a quotient set $\mathcal{C}_{dr}^*(R)/\sim$ of the set 
${\mathcal{C}_{dr}^*(R)=\{(\phi, I)\mid 
\phi\in \Hom^*(I, R)_{N(R)},\ I\in \mathcal{D}_r^*(R)\}}$ under the equivalence 
relation $\sim$: ${(\phi, I)\sim (\phi', I')}$ if 
${\phi=\phi'}$ on ${I\cap I'}$ (on some  
${I''\in \mathcal{D}_r^*(R)}$, ${I''\subseteq I\cap I'}$), with the 
operations  
\begin{gather*}  
[(\phi, I R^1)]+[(\psi, J R^1)]\ =\ [(\phi+\psi, (I\cap J) R^1)]\,,
\\ 
[(\phi, I R^1)] [(\psi, J R^1)]\ =\ 
[(\phi \psi, (\psi^{-1}(I R^1)\cap J) R^1)]\,,
\\
f [(\phi, I R^1)]\ =\ [(f \Id_R, N(R) R^1)] [(\phi, I R^1)]\ =\ 
[(\phi, I R^1)] [(f \Id_R, N(R) R^1)]\ =\ [(f \phi, I R^1)]\,,
\end{gather*}
${(\phi, I R^1), (\psi, J R^1)\in \mathcal{C}_{dr}^*(R)}$, 
${I, J\in \mathcal{D}_{nr}(R)}$, ${f\in F}$, where 
\begin{gather*}
\phi \psi\Bigl(\sum_i x_i y_i\Bigr)\ =\ 
\sum_i (\phi \psi x_i) y_i\quad 
(x_i\in \psi^{-1}(I R^1)\cap J=\{x\in J\mid \psi x\in I R^1\},\ y_i\in R^1)\,,
\\
[(\phi, I R^1)] [(l_x, N(R) R^1)]\ =\ [(l_{\phi x}, I R^1)]\ =\ 
[(l_{\phi x}, N(R) R^1)]\quad (x\in I R^1)\,, 
\end{gather*}
unity ${[(\Id_R, N(R) R^1)]}$, ${R\hookrightarrow Q(R)}$, 
${x\longmapsto [(l_x, N(R) R^1)]}$, ${x\in R}$. The algebras of quotients of 
$R$ inherit its properties such as semiprimeness, primeness, non-degeneracy, 
and pure alternativity (${R\cap U(Q)\subseteq U(R)=\{0\}}$).
Using \cite[Theorem~3.1.11]{BD}, \cite{BMS} and \cite{GolO}, we can deduce 

\begin{prop}
If an algebra $R$ is non-degenerate and purely alternative, then 
\[
Q(R)\ =\ Q(P(R))\ =\ O(P(R))\ =\ P(Q(R))
\] 
is a Cayley --- Dickson algebra over ${\CM(R)\cong Z(Q(R))}$.
\end{prop}

\begin{proof}
By our hypothesis, ${Z(R)=N(R)}$, 
${Z(D(R))=N(D(R))=D(R)\cap Z(R)\ne \{0\}}$ and $\Ann_t Z(R)=\{0\}$, 
${t=l, r}$ \cite[Theorems~11,~1,~2,~3,~7,~Corollary~of~
Theorem~6,~Lemma~7,~p. 205,~210,~211,~212,~225,~224]{ZShS}, \cite[Remark~1.1]{GolO}. If ${I\in \mathcal{F}''=\mathcal{E}(Z(D(R)))}$, then 
$\Ann I=\Ann_t I\lhd R$, ${t=l, r}$, and for ${I'=\Ann I\cap D(R)}$ 
\begin{gather*}
Z(I')\ =\ N(I')\ =\ I'\cap Z(R)\ =\ I'\cap Z(D(R))\,,
\\
\{0\}\ =\ I Z(I')\ =\ (I\cap Z(I'))^2\ =\ I\cap Z(I')\ =\ Z(I')\ =\ I'
\end{gather*}
(see above for \cite{ZShS}, semiprimeness of $Z(D(R))$),
${\Ann I\subseteq \Ann_t D(R)}$, ${t=l, r}$. Due to 
\[
(x y) J\ \subseteq\ x (y J+J y)+(x J) y\ =\ \{0\}\,,\quad 
(y x) J\ \subseteq\ (x y) J+y (x J)+x (y J)\ =\ \{0\}
\]
(${J (x y)=J (y x)=\{0\}}$) for all ${J\lhd R}$, ${x, y\in R}$, ${x J=\{0\}}$ 
(${J x=\{0\}}$), ${\Ann_t J\lhd R}$, ${t=l, r}$, 
\begin{gather*}
D(\Ann_t D(R))\ \subseteq\ D(R)\,,\quad 
\{0\}\ =\ D(\Ann_t D(R))^2\ =\ D(\Ann_t D(R))\,,
\\
\Ann_t D(R)\ =\ Z(\Ann_t D(R))\ =\ N(\Ann_t D(R))\ =\ 
\Ann_t D(R)\cap Z(R)\ \subseteq\ U(R)\ =\ \{0\}\,,
\end{gather*}
${D(R)\in \mathcal{D}_t(R)\cap \mathcal{E}(R)}$, ${\Ann I=\Ann Z(D(R))=\{0\}}$, 
${(I)_{Z(R)}=I Z(R)^1\in \mathcal{D}_{nt}(R)\cap \mathcal{F}''\cap \mathcal{F}'}$, 
${\mathcal{F}'=\mathcal{E}(Z(R))}$, for all ${I\in \mathcal{F}''}$, ${t=l, r}$ 
(for any ${0\ne x\in R}$, ${y\in Z(R)}$ there is ${z\in I}$, 
${x z\ne 0}$, ${y z\in I Z(R)^1}$), $R$ has no $\mathcal{F}''$--torsion. 

If ${J\in \mathcal{F}'}$, ${J'=J\cap D(R)=J\cap Z(D(R))}$ and 
${\Ann J'\ne \{0\}}$, then 
\[
J''\ =\ \Ann J'\cap D(R)\cap Z(R)\ =\ Z(\Ann J'\cap D(R))\ =\ 
\Ann J'\cap Z(D(R))\ \ne\ \{0\}
\] 
\cite[Corollary~of~Theorem~6,~p.~224]{ZShS}, 
${J''\lhd Z(R)}$, ${\{0\}=J' J''=J'\cap J''=J\cap J''=J''}$, a contradiction. So, 
${\mathcal{F}''\cap \mathcal{F}'=\{J\cap D(R)\mid J\in \mathcal{F}'\}=
\{I Z(R)^1\mid I\in \mathcal{F}''\}}$,
$R$ has no $\mathcal{F}'$--torsion.

Any ${\phi\in \Hom(I, R)_{Z(D(R))}}$, ${I\in \mathcal{F}''}$, can be 
extended to ${\overline{\phi}\in \Hom^*(I R^1, R)_{Z(R)}}$ 
by the rule: 
${\overline{\phi} \Bigl(\sum\limits_{i=1}^k x_i y_i\Bigr)=
\sum\limits_{i=1}^k (\phi x_i) y_i}$, 
${x_i\in I}$, ${y_i\in R^1}$, ${k\geq 1}$, since for 
${\sum\limits_{i=1}^k x_i y_i=0}$ 
\[
0\ =\ (\phi a) \sum_{i=1}^k x_i y_i\ =\ 
\sum_{i=1}^k \phi(x_i a) y_i\ =\ 
\biggl(\sum_{i=1}^k (\phi_i x_i) y_i\biggr) a\ =\ 
\sum_{i=1}^k (\phi x_i) y_i
\quad (a\in I)\,,
\]
for any ${x\in R}$ and ${y=\sum\limits_{i=1}^k x_i y_i, z\in 
Z(I R^1)=N(I R^1)=I R^1\cap Z(R)=I R^1\cap Z(D(R))}$ 
\begin{multline*}
0\ =\ \sum_{i=1}^k ((\phi a) (x_i y_i)) x-(\phi a) ((x_i y_i) x)\ =\ 
\sum_{i=1}^k ((((\phi a) x_i) y_i) x-((\phi a) x_i) (y_i x))\ =
\\
\sum_{i=1}^k (((\phi x_i) y_i) x-(\phi x_i) (y_i x)) a\ =\ 
\sum_{i=1}^k ((\phi x_i) y_i x-(\phi x_i) (y_i x))\ =\ 0\quad (a\in I)\,,
\end{multline*}
${\overline{\phi} (y x)=(\overline{\phi} y) x}$, 
${\overline{\phi} [y, z]=\overline{\phi} 0=0}$. If ${J\in \mathcal{D}_r(R)}$, 
then ${J Z(D(R))\subseteq J'=J\cap D(R)\in \mathcal{D}_r(R)}$, 
${Z(J')=N(J')=J\cap Z(D(R))\in \mathcal{F}''\cap \mathcal{F}'}$, 
${Z(J') R^1\in \mathcal{D}_r^*(R)}$ (see above for \cite{ZShS}; 
for each ${0\ne x\in Z(R)}$ there exist ${y\in Z(D(R))}$,
${z\in Z(R)}$, ${0\ne x y z\in J\cap Z(D(R))=Z(J')}$).
If $\psi\in \Hom^*(Z(J) R^1, R)_{Z(R)}$, then 
${\psi (y x)=(\psi y) x}$ for all ${x\in R}$, ${y\in Z(J)}$ and 
\[
\psi'\ =\ \psi|_{Z(J') R^1}\ =\ \overline{\psi'|_{Z(J')}}\,,\, 
\psi'|_{Z(J')}\in \Hom(Z(J'), R)_{Z(R)}\,,\, 
[(\psi, Z(J) R^1)]\ =\ [(\psi', Z(J') R^1)]\,.
\]
Therefore, ${\iota: [(\phi, I)]\longmapsto [(\overline{\phi}, I R^1)]}$, 
${I\in \mathcal{F}''}$, ${\phi\in \Hom(I, R)_{Z(D(R))}}$, is an 
$F$--isomorphism of $R_{\mathcal{F}''}$ and ${\iota R_{\mathcal{F}''}=Q(R)}$, 
identical on ${R\cong \{[(l_x, Z(D(R)))]\mid x\in R\}}$, since 
\begin{gather*}
\iota [(\phi, I)]+\iota [(\psi, J)]\ =\ [(\overline{\phi}, I R^1)]+
[(\overline{\psi}, J R^1)]\ =\ 
[(\overline{\phi+\psi}, (I\cap J) R^1)]\ =\ \iota [(\phi+\psi, I J)]\,,
\\
\iota [(\phi, I)] \iota [(\psi, J)]\ =\ 
[(\overline{\phi}, I R^1)] [(\overline{\psi}, J R^1)]\ =\  
[(\overline{\phi} \overline{\psi}, (\overline{\psi}^{-1}(I R^1)\cap J) R^1)]\ =\ 
\iota([(\phi, I)] [(\psi, J)])\,,
\\
f \iota [(\phi, I)]\ =\ f [(\overline{\phi}, I R^1)]\ =\ 
[(f \overline{\phi}, I R^1)]\ =\ \iota (f [(\phi, I)])
\end{gather*}
for all ${I, J\in \mathcal{F}''}$, ${\phi\in \Hom(I, R)_{Z(D(R))}}$, 
${\psi\in \Hom(J, R)_{Z(D(R))}}$, ${f\in F}$, where 
\begin{gather*}
I J\ \subseteq\ \overline{\psi}^{-1}(I R^1)\cap J=
\{z\in J\mid \psi z\in I R^1\}\,,
\\
\overline{\phi} \overline{\psi} 
\Bigl(\sum_i z_i x_i\Bigr)\ =\ \sum_i (\overline{\phi} \psi z_i) x_i\quad 
(x_i\in R^1,\ z_i\in \psi^{-1}(I R^1)\cap J)\,,
\\
\overline{\phi} \overline{\psi} \Bigl(\sum_i (y_i z_i) x_i\Bigr)\ =\ 
\sum_i \overline{\phi}(y_i \psi z_i) x_i\ =\ 
\sum_i ((\phi y_i) (\psi z_i)) x_i\quad 
(x_i\in R^1,\ y_i\in I,\ z_i\in J)\,.
\end{gather*}
It induces isomorphisms ${\iota: R_{\mathcal{F}''}\longrightarrow Q(R)}$ over 
\begin{multline*}
Z(R_{\mathcal{F}''})\ =\ \{[(\phi, I)]\mid I\in \mathcal{F}'',\ 
\phi\in \Hom(I, Z(D(R)))_{Z(D(R))}\}\ =
\\ 
Z(D(R))_{\mathcal{F}''}\ =\ Q(Z(D(R)))\cong 
\iota Z(R_{\mathcal{F}''})\ =\ Z(Q(R))
\end{multline*}
and ${\iota \rho^{-1}: R_{\mathcal{F}'}\longrightarrow Q(R)}$ over $F$ and 
${Z(R_{\mathcal{F}'})\cong \iota \rho^{-1} Z(R_{\mathcal{F}'})=Z(Q(R))}$, 
identical on $R\cong \{[(l_x, Z(R))]\mid x\in R\}$,  
${\rho: [(\phi, I)]\longmapsto [(\overline{\phi}|_{I Z(R)^1}, I Z(R)^1)]}$, 
${I\in \mathcal{F}''}$, ${\phi\in \Hom(I, R)_{Z(D(R))}}$, is an 
$F$--isomorphism of $R_{\mathcal{F}''}$ and $R_{\mathcal{F}'}$ \cite[Lemma~3.1.8~and~the~observations~above]{BD}. 

If ${\gamma\in \Hom(H, Q(R))_{M(Q(R))}}$ (${1\in Q(R)}$), ${H\lhd Q(R)}$, 
then without loss of generality 
$H\in \mathcal{E}(Q(R))$, ${Z(H)=N(H)=H\cap Z(Q(R))\ne \{0\}}$, 
\[
0\ =\ \gamma f(x, x_2, \ldots, x_8)\ =\ f(\gamma x, x_2, \ldots, x_8)\quad 
(x\in Z(H),\ x_i\in Q(R))\,,
\]
${\gamma|_{Z(H)}\in \Hom(Z(H), Z(Q(R)))_{Z(Q(R))}}$ and  
${\gamma|_{Z(H)}=r_q|_{Z(H)}}$ for some ${q\in Z(Q(R))}$ (see links 
above on \cite{ZShS}; $Q(R)$ is non-degenerate and purely alternative; 
$f$ from Remark 2.3; $Z(Q(R))\cong Q(Z(D(R)))$ is self-injective). 
Since for ${W=\Ann(Q(R), Z(H))}$
\begin{gather*}
Z(H\cap W)\ =\ H\cap W\cap Z(Q(R))\ =\ 
Z(H)\cap Z(W)\,,
\\ 
\{0\}\ =\ Z(H\cap W)^2\ =\ Z(H\cap W)\ =\ H\cap W\ =\ W\,,
\\
0\ =\ (\gamma z-q z) x\ =\ (\gamma x-q x) z\ =\ \gamma x-q x\quad (x\in H,\ z\in Z(H))\,,
\end{gather*}
${\gamma=r_q|_H=q \Id_H}$ (see above for \cite{ZShS}), 
${Q(R)=P(Q(R))}$ (${Q(R)=I(Q(R))}$, due to the endofiniteness of 
the $M(Q(R))$--module ${Q(R)=\iota R_{\mathcal{F}''}}$ 
\cite[Theorem~3.1.11]{BD}, 
\cite[Vol.~2,~Theorem~19.14~A,~p.~111]{Fe}), 
${\CM(Q(R))=Z(Q(R)) \Id_{Q(R)}\cong Z(Q(R))}$ taking into account  
\[
\{0\}\ =\ (\Ann_{Z(Q(R))} Q(R))^2\ =\ 
\Ann_{Z(Q(R))} Q(R)\ =\ \Ann_{Z(Q(R))} R\,.
\]
If now ${\gamma\in \Hom(H, Q(R))_{M(R)'}}$, $H$ is a $M(R)'$--submodule of 
$Q(R)$, then we can assume that $H$ is essential in $Q(R)_{M(R)'}$, and 
hence ${H'=H\cap R\in \mathcal{E}(R), Z(H')=H\cap Z(R)\in \mathcal{F}'}$, 
$H''=Z(H')\cap D(R)=H\cap Z(D(R))\in \mathcal{F}''$, 
${\gamma|_{H''}=r_q|_{H''}\in \Hom(H'', Q(R))_{Z(R)}}$ for some uniquely 
defined ${q\in Q(R)}$ (see the reasoning above for ${H\in \mathcal{E}(Q(R))}$ 
with $H$ replaced by $H'$ and $Q(R)$ replaced by $R$, the beginning of the 
proof and the description of almost classical localizations in Sec. 2), 
\[
0\ =\ (\gamma a-q a) b\ =\ (\gamma b-q b) a\ =\ \gamma b-q b\quad 
(a\in H'',\ b\in H')\,, 
\]
${\gamma|_{H'}=r_q|_{H'}}$ (${Z(Q(R))=\iota Z(D(R))_{\mathcal{F}''}}$, 
${Q(R)=\iota R_{\mathcal{F}''}=\iota \rho^{-1} R_{\mathcal{F}'}}$ has 
no $\mathcal{F}''$--torsion and $\mathcal{F}'$--torsion). Since for any 
${x\in H}$ there is ${I\in \mathcal{F}''}$, ${I x\subseteq H'}$, 
\[
0\ =\ \gamma (x y)-q (x y)\ =\ (\gamma x-q x) y\ =\ \gamma x-q x\quad 
(y\in I)\,,
\]
${\gamma=r_q|_H}$, ${q (x \beta)=(q x) \beta}$ for all ${x\in H}$, 
${\beta\in M(R)'}$ (${\beta\in M^{Q(R)}(R)'}$) and, as a consequence, 
\begin{multline*}
0\ =\ q (x z)-q (x z)\ =\ (q x) z-x (q z)\ =\ [q, x] z\ =\ [q, x]\quad 
(x\in H,\ z\in Z(H'))\,,
\\
\shoveleft{
0\ =\ (q x) v-(q x) v\ =\ (q x) v-q (x v)\ =\ (q, x, v)\quad 
(v\in R,\ x\in H)\,,}
\\
\shoveleft{
0\ =\ [q, v z]\ =\ [q, v] z\ =\ [q, v]\ =}
\\ 
(q, v z, u)\ =\ (q, v, u) z\ =\ (q, v, u)\quad (v, u\in R,\ z\in Z(H'))\,,
\end{multline*}
${q\in Z(Q(R))}$ \cite[Remark~3.12]{GolO}. Thus, 
\begin{gather*}
Q(R)\ =\ P(Q(R)_{M(R)'})\ =\ I(Q(R)_{M(R)'})\,,\quad 
\End(Q(R))_{M(R)'}\ =\ Z(Q(R)) \Id_{Q(R)}\,,
\\
R\ \subseteq\ P(R)\ =\ Z(Q(R)) R\ \subseteq\ Q(R)\,,\quad 
\CM(R)\ =\ Z(Q(R)) \Id_{P(R)}\cong Z(Q(R))\,.
\end{gather*}
The equality of $P(R)$ and $Q(R)$ is equivalent to the finite generation of 
$P(R)$ over $Z(Q(R))$. This is true, in particular, if the finitely 
generated $Z(Q(R))$--module $Q(R)$ is Noetherian ($Z(Q(R))$ is Noetherian; 
for example, for the prime $R$ and the field ${\CM(R)\cong Z(Q(R))}$).

It was noted earlier that $P(R)$ inherits the non-degeneracy and pure 
alternativity of $R$, ${D(P(R))=Z(Q(R)) D(R)}$, ${Z(P(R))\subseteq Z(Q(R))}$ 
and ${Z(D(P(R)))=D(P(R))\cap Z(Q(R))}$ (${\Ann_t(Q(R), R)=\{0\}}$, ${t=l, r}$ 
and 
${Z(P(R)) \Id_{P(R)}\subseteq \CM(P(R))=Z(Q(R)) \Id_{P(R)}}$), for any\linebreak 
${J\in \mathcal{F}'''=\mathcal{E}(Z(D(P(R))))}$, ${z\in J}$ there exists 
${K\in \mathcal{F}''}$, ${K z\subseteq R\cap J=Z(R)\cap J}$, 
($Z(D(R))\subseteq Z(D(P(R)))$), 
${K Z(D(R)) z\subseteq J'=J\cap Z(D(R))=J\cap D(R)}$, 
${K Z(D(R))\in \mathcal{F}''}$, and hence,  
${\{0\}=K Z(D(R)) z \Ann(P(R), J')=z \Ann(P(R), J')}$, 
${\Ann(P(R), J')\subseteq \Ann(P(R), J)=\{0\}}$,
${J'\in \mathcal{F}''}$ ($P(R)$ and $Q(R)$ have no $\mathcal{F}'''$--torsion 
and $\mathcal{F}''$--torsion (the beginning of the proof)). 
If\linebreak ${\psi\in \Hom(J, P(R))_{Z(D(P(R)))}}$, then 
${\psi|_{J'}\in \Hom(J', P(R))_{Z(D(R))}}$, there is a unique 
$q\in Q(R)=\iota R_{\mathcal{F}''}$, ${\psi|_{J'}=r_q|_{J'}}$ 
(the description of almost classical localization in Sec. 2) and 
for any ${z\in J}$, ${K'\in \mathcal{F}''}$, ${K' z\in J'}$ (see above), 
${K' (\psi z-z q)=\{0\}}$, ${\psi z=z q}$, ${\psi=r_q|_J}$. 
If ${I\in \mathcal{F}''}$, then ${I Z(Q(R))\in \mathcal{F}'''}$ and any 
${\phi\in \Hom(I, R)_{Z(D(R))}}$ extends to 
${\hat{\phi}\in \Hom(I Z(Q(R)), P(R))_{Z(Q(R))}}$, 
${\hat{\phi} \Bigl(\sum\limits_{i=1}^k y_i z_i\Bigr)=
\sum\limits_{i=1}^k (\phi y_i) z_i, y_i\in J, z_i\in Z(Q(R)), k\geq 1}$, 
since for ${\sum\limits_{i=1}^k y_i z_i=0}$ there is ${M\in \mathcal{F}''}$, 
${M z_i\subseteq Z(D(R))}$, ${i=1, \ldots, k}$, 
\[
0\ =\ \phi \biggl(\biggl(\sum_{i=1}^k y_i z_i\biggr) z\biggr)\ =\ 
\sum_{i=1}^k (\phi y_i) (z_i z)\ =\ 
\biggl(\sum_{i=1}^k (\phi y_i) z_i\biggr) z\ =\ 
\sum_{i=1}^k (\phi y_i) z_i\quad (z\in M)\,.
\]
Therefore, ${\xi: [(\psi, J)]=[(r_q, J)]\longmapsto q, 
J\in \mathcal{F}''', \psi\in \Hom(J, P(R))_{Z(D(P(R)))}}$, is an 
$F$--isomor\-phism of $P(R)_{\mathcal{F}'''}$ and 
${Q(R)=\iota R_{\mathcal{F}''}}$, identical on 
${P(R)\cong \{[(l_x, Z(D(P(R))))]\mid x\in P(R)\}}$, 
${\xi^{-1} q=[(\hat{\phi}, I Z(Q(R)))]}$, ${[(\phi, I)]=\iota^{-1} q=
[(r_q, I)]}$, ${I\in \mathcal{F}''}$, ${I q\subseteq R}$. 

As a result, applying the previously obtained conclusions to $P(R)$, we obtain 
that, up to isomorphism, ${Q(R)=Q(P(R))=O(P(R))}$ (the observations 
before Proposition 2.7), and, by virtue of 
\cite[Theorem~3.1.18]{BD}, 
${Q(R)\cong R_{\mathcal{F}''}=
\mathcal{O}_{Z(R_{\mathcal{F}''})}(\mu, \beta, \gamma)}$ 
is a Cayley --- Dickson algebra over $Z(R_{\mathcal{F}''})$ for some invertible 
${\mu, \beta, \gamma\in Z(R_{\mathcal{F}''})}$.
\end{proof}

\begin{prop}
If an associative algebra $R$ is semiprime and ${Z(J)\ne \{0\}}$ for all 
$\{0\}\ne J\lhd I\in \mathcal{D}_r(R)$, then 
${Q^s_m(R)=Q(R)=Q(P(R))=O(P(R))=P(Q(R))}$. 
\end{prop}

\begin{proof}
Due to ${\prr(R)=\{0\}}$, ${Z(I)=I\cap Z(R)}$ for any right (left) ideal 
$I$ in $R$ \cite[Theorem~3,~p.~212]{ZShS}, \cite{Mart1}. If $I$ is a right 
ideal of $R$, ${\Ann_l I=\{0\}}$, then ${\prr(I)=\{0\}}$, 
${\Ann_r I=\{0\}}$, since for a right ideal $A I$ in $R$, ${A\subseteq R}$, 
from ${(A I)^2=\{0\}}$ it follows that
${A I\subseteq \prr(R)=\{0\}}$, $A\subseteq \Ann_l I=\{0\}$, 
in particular, ${A=\{0\}}$ for any ${A\lhd I}$, 
${A^2=\{0\}}$, and ${A=\Ann_r I}$ \cite{Mart1} (similarly for the left ideal 
$I$ with ${\Ann_r I=\{0\}}$). If ${I\in \mathcal{D}_r(R)}$, 
${J\in \mathcal{E}(Z(R))}$, then 
$Z(I\cap \Ann Z(I))=I\cap \Ann Z(I)\cap Z(R)\lhd Z(R)$ and by the condition 
\begin{multline*}
\{0\}\ =\ Z(I\cap \Ann Z(I))^2\ =\ Z(I\cap \Ann Z(I))\ =\ I\cap \Ann Z(I)\ =
\\
(J\cap Z(\Ann J))^2\ =\ J\cap Z(\Ann J)\ =\ Z(\Ann J)\ =\ \Ann J\,,
\end{multline*}
${\Ann Z(I)\subseteq \Ann I=\Ann_t I=\{0\}, t=l, r}$, 
${Z(I) R^1, J R^1\in \mathcal{E}(R)\subseteq \mathcal{D}_r(R)}$ 
(${I\cap \Ann Z(I)\lhd I}$, ${R\in \mathcal{D}_r(R)}$; 
\cite[Theorem~3]{Mart1}), ${Z(I) R^1\subseteq I}$, ${Z(I)\in \mathcal{E}(Z(R))}$. 
Therefore, as in the previous\linebreak proof, 
${\iota: [(\phi, J)]\longmapsto [(\overline{\phi}, J R^1)]}$, 
${J\in \mathcal{E}(Z(R))}$, ${\phi\in \Hom(J, R)_{Z(R)}}$, is an 
$F$--isomorphism of $R_{\mathcal{F}'}$ and $Q(R)$, identical on 
${R\cong \{[(l_x, Z(R))]\mid x\in R\}}$, 
${\overline{\phi}\Bigl(\sum\limits_{i=1}^k x_i y_i\Bigr)=
\sum\limits_{i=1}^k (\phi x_i) y_i}$, ${x_i\in J}$, $y_i\in R^1$, 
${k\geq 1}$. Since for each ${q\in Q(R)}$ there exists  
${I\in \mathcal{D}_r(R)}$, ${q I\subseteq R}$, and  
${Z(R)\subseteq Z(Q(R))}$, ${Z(I) R^1 q+q Z(I) R^1\subseteq R}$ for 
${Z(I) R^1\in \mathcal{E}(R)}$, 
\[
P(R)\ =\ Z(Q(R)) R\ \subseteq\ Q^s_m(R)\ =\ Q(R)\ =\ Q(P(R))\ =\ 
Q(Q(R))\ =\ O(Q(R))
\] 
\cite[Introductions]{GolO, GolA}, \cite[Sec.~3]{GolO}, 
\cite[Lemmas~3,~4]{Har2}, ${Q(R)=\iota R_{\mathcal{F}'}}$ has no $\mathcal{F}'$--torsion, 
$R$ is an essential $Z(R)$--submodule of $Q(R)$. So, if  
$B$ is a $F$--subalgebra of $Q(R)$, ${R\subseteq B}$, 
${\{0\}\ne J\lhd I\in \mathcal{D}_r(B)}$, then ${\prr(B)=\prr(I)=\{0\}}$ 
(${\prr(B)\cap R\subseteq \prr(R)=\{0\}}$, the beginning of the proof), for 
any ${x, y\in B}$, ${x\ne 0}$, we can choose ${z\in B}$, 
${H\in \mathcal{F}'}$, ${H x z\ne \{0\}}$, ${H y z\subseteq I\cap R}$, 
${H z\subseteq R}$, ${\{0\}\ne J^2 Z(R)\cap R\subseteq J\cap R\lhd I\cap R\in 
\mathcal{D}_r(R)}$ and by the condition  
\[
\{0\}\ \ne\ Z(J\cap R)\ =\ J\cap R\cap Z(I\cap R)\ =\ J\cap Z(R)\ =\ 
J\cap R\cap Z(Q(R))\ \subseteq\ Z(J)\,, 
\]
$B$ inherits the conditions on $R$. For ${B=P(R)}$, the previously obtained 
conclusions, Proposition 2.1 and the observations before Proposition 2.7 
give $Q(R)=Q(P(R))=O(P(R))=P(Q(R))$. 
\end{proof}

Under the conditions of Proposition 3.2 (they are satisfied, in particular, if 
$R$ is a $PI$--algebra \cite{Mart1}), ${R=P(R)=O(R)}$ is equivalent to ${R=Q(R)}$.
A regular associative algebra $R$ is \emph{strictly regular} if $Z(R)$ contains 
all idempotents of $R$. In other words, for any ${x\in R}$ there is 
${y\in R}$, ${x y\in Z(R)}$, ${x^2 y=x}$ (for ${z^2=z\in R}$ this gives 
${z\in Z(R)}$ and hence ${y x\in Z(R)}$). Proposition 3.3 corresponds to the 
Fisher --- Martindale --- Markov --- Armendariz --- Steinberg theorem 
\cite{Fish, Mart1, Mar3, ArmS}, its proof corresponds to 
\cite[Example~8.14]{BMO} (it is given for full exposition).

\begin{prop}
If $R$ is a semiprime associative $PI$--algebra, then $Q(R)$ is a finite 
direct sum of matrix algebras over strictly regular algebras, finitely 
generated as modules over their centers. 
\end{prop}

\begin{proof} Without loss of generality ${R=Q(R)}$, ${\CM(R)=Z(R)}$. 
By the Posner theorem and Lemma 2.4 ${R_B=Q(R_B)=M_{n_B}(D_B)}$ is an 
algebra of ${n_B\times n_B}$ matrices over the division ring $D_B$, 
${\CM(R_B)=Z(R_B)=Z(R)_B=Z(D_B)}$ is a field, 
${\dim_{\CM(R_B)} D_B=m_B^2}$, where  
$1\leq n_B m_B=\pideg R_B\leq \sqrt{m(R)}\leq \pideg R$,
for all ${B\in \mathcal{U}(R)}$ \cite{Pos, Mar, Row} and so,  
in $R_B$ the Horn formula 
\[
\Xi_k\ =\ (\exists e_{11})\ldots (\exists e_{kk}) 
[\Gamma_k(e_{11}, \ldots, e_{kk})\wedge
\Theta_k(e_{11}, \ldots, e_{kk})\wedge \Pi_k(e_{11}, \ldots, e_{kk})]
\]
is true for ${k=n_B}$, 
${\Gamma_k(e_{11}, \ldots, e_{kk})=\biggl\{\bigwedge\limits_{i, j, p, q=1}^k 
(e_{ij} e_{pq}=\delta_{jp} e_{iq})\biggr\}}$, 
\begin{multline*}
\Theta_k(e_{11}, \ldots, e_{kk})\ =\ 
(\forall x) (\exists z_{11})\ldots (\exists z_{kk})
(\exists y_{11})\ldots (\exists y_{kk})
\\
\shoveleft{
\biggl[
\biggl\{x=\sum_{i, j=1}^k e_{ij} z_{ij}\biggr\}\wedge
\biggl\{\bigwedge_{i, j, p, q=1}^k ([z_{ij}, e_{pq}]=0)\biggr\}\wedge
\biggl\{\bigwedge_{i, j, p, q=1}^k ([y_{ij}, e_{pq}]=0)\biggr\}\wedge}
\\
\shoveright{ 
\biggl\{\bigwedge_{i, j, p, q=1}^k (z_{ij} y_{ij}\in Z(R))\biggr\}\wedge
\biggl\{\bigwedge_{i, j, p, q=1}^k (z_{ij}^2 y_{ij}=z_{ij})\biggr\}\biggr]\,,}
\\
\shoveleft{
\Pi_k(e_{11}, \ldots, e_{kk})\ =\ 
(\forall z_{11})\ldots (\forall z_{kk})\biggl[
\biggl\{\bigvee_{i, j, p, q=1}^k ([z_{ij}, e_{pq}]\ne 0)
\biggr\}\vee}
\\ 
\biggl\{\bigwedge_{i, j=1}^k (z_{ij}=0)\biggr\}\vee 
\biggl(\sum_{i, j=1}^k e_{ij} z_{ij}\ne 0\biggr)\biggr]\,.
\end{multline*}
Therefore, as in Lemmas 2.4, 2.6, there exist  
${0\ne \eta_i\in B(R), \eta_i \eta_j=\delta_{ij} \eta_i, 
1=\eta_1+\ldots+\eta_l}$ (${1\in R=Q(R)}$), $\Xi_{k_i}$ is true in $\eta_i R$, 
there is a system ${E_i=\{e^{(i)}_{pq}\}_{p, q=1}^{k_i}\subset \eta_i R}$, 
${e^{(i)}_{pq} e^{(i)}_{st}=\delta_{qs} e^{(i)}_{pt}}$, 
${\eta_i=\sum\limits_{p=1}^{k_i} e^{(i)}_{pp}}$, any ${x\in \eta_i R}$ is 
uniquely expressed in the form 
$x=\sum\limits_{p, q=1}^{k_i} e^{(i)}_{pq} z_{pq}, 
z_{pq}\in Z(\eta_i R, E_i)=\{z\in \eta_i R\mid [z, E_i]=\{0\}\}$, 
${Z(\eta_i R, E_i)}$ is a strictly regular subalgebra of $\eta_i R$ and, due to 
the finite generation of the $Z(R)$--module $R$ (Lemma 2.4), 
$\eta_i R$, ${Z(\eta_i R, E_i)}$ are finitely generated modules over 
${Z(\eta_i R)=\eta_i Z(R)}$, ${i=1, \ldots, l}$,  
${R=\bigoplus\limits_{i=1}^l \eta_i R\cong \bigoplus\limits_{i=1}^l 
M_{k_i}(Z(\eta_i R, E_i))}$. 
\end{proof}

The decomposition of a non-degenerate alternative ${R=P(R)}$ into a direct 
sum of the associative kernel and the purely alternative ideal 
${R=\tau R\oplus (\Id_{P(R)}-\tau) R=U(R)\oplus \Ann_t U(R)}$, ${t=l, r}$, 
from the beginning of Sec. 3 can be obtained from the last observation of 
Sec. 2 as a decomposition of $R$ into a direct sum of two ideals, the first 
of which is a subdirect product of prime associative algebras, and the 
second is a subdirect product of Cayley --- Dickson rings 
\cite[Addition~1]{GolA}, 
\cite[Proposition~3,~Theorem~9,~p.~228,~229]{ZShS}. 

\begin{co}
If an alternative algebra $R$ is non-degenerate, then 
\begin{multline*}
O(P(R))\ =\ O(P(R))_{\tau}\ \cong\ 
O(P(\tau R))\oplus O(P((\Id_{P(R)}-\tau) R))\ =
\\ 
O(P(\tau R))\oplus Q(P((\Id_{P(R)}-\tau) R))\ =\ 
O(P(\tau R))\oplus Q((\Id_{P(R)}-\tau) R)
\end{multline*}
and for the $PI$--algebra $R$ 
\begin{multline*}
O(P(R))\ \cong\ Q(P(\tau R))\oplus Q(P((\Id_{P(R)}-\tau) R))\ \cong 
\\
Q(P(R))\ \cong\  
Q(\tau R)\oplus Q((\Id_{P(R)}-\tau) R)\ \cong\ Q(R_{\tau})\,,
\end{multline*}
${Q((\Id_{P(R)}-\tau) R)}$ ($Q(\tau R)$ for the $PI$--algebra $R$) is 
described by Proposition 3.1 (3.2, 3.3).
\end{co}

\begin{proof}
In view of the observations made earlier, it is only necessary to establish 
that $Q(R_{\tau})\cong Q(\tau R)\oplus Q((\Id_{P(R)}-\tau) R)$ and, 
as a consequence, 
\[
Q(P(R))\ =\ Q(P(R)_{\tau})\ \cong 
Q(\tau P(R))\oplus Q((\Id_{P(R)}-\tau) P(R))\,.
\]

If ${I\in \mathcal{D}_{nr}(R_{\tau})}$, 
${\phi\in \Hom^*(I R_{\tau}^1, R_{\tau})_{N(R_{\tau})}}$, 
${\pi=\tau, \Id_{P(R)}-\tau}$, then ${\pi I\in \mathcal{D}_{nr}(\pi R)}$,
$\pi I R_{\tau}^1=(\pi I)(\pi R)^1\in \mathcal{D}^*_r(\pi R)$, 
${{}_{\pi} \phi\in \Hom^*(\pi I R_{\tau}^1, \pi R)_{N(R_{\tau})}=
\Hom^*(\pi I R_{\tau}^1, \pi R)_{N(\pi R)}}$, where
${}_{\pi} \phi \pi x=\pi \phi x$, 
${x\in I R_{\tau}^1}$, since for any ${x\in R_{\tau}}$, 
${y\in N(R_{\tau})=N(R_{\tau})_{\tau}}$, ${\pi x\ne 0}$, there exists 
${z\in N(R_{\tau})}$, ${(\pi x) z=(\pi x) (\pi z)\ne 0}$, ${y z\in I, 
\pi (y z)=(\pi y) (\pi z)\in \pi I, \pi z\in \pi N(R_{\tau})=N(\pi R)}$, 
and for any ${u=(\Id_{P(R)}-\pi) u\in I R_{\tau}^1}$ 
\[
(\phi u) v\ =\ \phi (u v)\ =\ (\phi u) (\Id_{P(R)}-\pi) v\ =\ 
((\Id_{P(R)}-\pi) \phi u) v\quad (v\in N(R_{\tau}))\,,
\]
${\pi \phi u\in \Ann_l N(R_{\tau})=\{0\}}$ ($P(R)$, $R_{\tau}$ are 
non-degenerate (the beginning of Sec. 3, 
\cite[the~observations~after~Lemma~2.6]{GolO}), 
\cite[Remark~3.11]{GolO}), 
${\phi={}_{\tau} \phi \tau+{}_{\Id_{P(R)}-\tau} \phi (\Id_{P(R)}-\tau)}$.
If ${I=I_{\tau}}$ is a right ideal of $N(R_{\tau})$, 
${\pi I\in \mathcal{D}_{nr}(\pi R)}$, 
${\psi_{\pi}\in \Hom^*(\pi I R_{\tau}^1, \pi R)_{N(\pi R)}}$, 
${\pi=\tau, \Id_{P(R)}-\tau}$, then 
${I\in \mathcal{D}_{nr}(R_{\tau})}$, 
${\psi=\psi_{\tau} \tau+\psi_{\Id_{P(R)}-\tau} (\Id_{P(R)}-\tau)\in 
\Hom^*(I R_{\tau}^1, R_{\tau})_{N(R_{\tau})}}$ with 
${{}_{\pi} \psi=\psi_{\pi}}$, ${\pi=\tau}$, ${\Id_{P(R)}-\tau}$.  
It remains to note that  
\begin{multline*}
[(\phi, I R_{\tau}^1)]\ =\ 
[({}_{\tau} \phi \tau+{}_{\Id_{P(R)}-\tau} \phi (\Id_{P(R)}-\tau), 
I_{\tau} R_{\tau}^1)]\ \longmapsto
\\ 
([({}_{\tau} \phi, \tau I R_{\tau}^1)], 
[({}_{\Id_{P(R)}-\tau} \phi, (\Id_{P(R)}-\tau) I R_{\tau}^1)])\quad 
((\phi, I R_{\tau}^1)\in \mathcal{C}^*_{dr}(R_{\tau}),\ 
I\in \mathcal{D}_{nr}(R_{\tau}))
\end{multline*}
is an $F$--isomorphism of $Q(R_{\tau})$ and 
${Q(\tau R)\oplus Q((\Id_{P(R)}-\tau) R)}$. 
\end{proof}

Since ${\Ann_t(P(R), I)=\Ann_t(P(R), \CM(R) I)\lhd P(R), t=l, r}$ for 
semiprime alternative $R$, all ${I\lhd R}$ and ${I=N(R)}$ (the conclusion of ${\Ann_t I=\Ann_t(R, I)\lhd R}$ for all ${I\lhd R}$ from the 
beginning of the proof of Proposition 3.1, \cite[the~proof~of~Remark~3.11]{GolO}, ${N(R)\subseteq N(P(R))}$), 
${\{0\}=\Ann_t I=R\cap \Ann_t(P(R), I)=\Ann_t(P(R), I)}$ 
for all ${I\in \mathcal{E}(R)}$ and for ${I=N(R)}$ and non-degenerate $R$. 

Let's move on to non-degenerate linear Jordan algebras over $F$ with $1/2$. 
According to \cite[Theorem~3]{ZelP1}, for any such (strongly) prime algebra 
${R\ne \{0\}}$ one of the following possibilities is realized: 
\begin{description}

\item[(1)] $R$ is an Albert ring ($R S^{-1}$ is a 27--dimensional Albert 
algebra over $Z(R S^{-1})$);

\item[(2)] ${R S^{-1}=\Jr(V, g)}$ is the Jordan algebra of the non-degenerate 
bilinear symmetric form $g$ on the $Z(R S^{-1})$---space $V$, 
${\dim_{Z(R S^{-1})} V> 1}$; 

\item[(3)] ${A^{(+)}\lhd R\subseteq Q^s_m(A)^{(+)}}$ for some prime 
associative algebra $A$; 

\item[(4)] ${\Sym(A, *)=\{x+x^*\mid x\in A\}\lhd R\subseteq \Sym(Q^s_m(A), *)}$ 
for some prime associative algebra $R$ with an involution $*$ uniquely extended 
from $A$ to $Q^s_m(A)$ ($Q(A)$),

\end{description}
in (1, 2) ${R S^{-1}=P(R)}$ is a simple algebra over the field 
${Z(R S^{-1})=Z(R) S^{-1}=\CM(R)}$, $S=Z(R)\setminus \{0\}\ne \emptyset$ 
(${\Jr(V, g)}$ over the field $\mathbb{F}$ with ${\dim_{\mathbb{F}} V=1}$ 
is associative, ${\Jr(V, g)=Z(\Jr(V, g))}$ is isomorphic either to 
${\mathbb{F}\oplus \mathbb{F}}$ or to the field $\mathbb{F}(\alpha)$, 
${\alpha\notin \mathbb{F}}$, ${\alpha^2\in \mathbb{F}}$). From this and the 
observations at the end of Sec. 2 follows the decomposition of any 
non-degenerate Jordan ${R=P(R)}$ into a direct sum 
${R=\alpha_1 R\oplus \alpha_2 R\oplus \alpha_3 R\oplus \alpha_4 R}$, 
${\Id_{P(R)}=\alpha_1+\alpha_2+\alpha_3+\alpha_4}$, ${\alpha_i\in B(R)}$, 
$\alpha_i \alpha_j=\delta_{ij} \alpha_i$, $\alpha_i R$ is subdirect product 
of strongly prime algebras of type $(i)$ for ${\alpha_i\ne 0}$, $\alpha_1 R$ 
for ${\alpha_1\ne 0}$ and  
${(\Id_R-\alpha_1) R=\alpha_2 R\oplus \alpha_3 R\oplus \alpha_4 R}$ are 
exceptional (not a homomorphic image of a special algebra) and 
special algebras. We detail this decomposition of $R$ using the constructions 
of \cite{ZelP, ZelP1}. 

If $R$ is a non-degenerate Jordan algebra over an integrity domain $F$ with 
1/2, then ${S(R)\cap T_H(R)=\{0\}}$, where $S$ and $T_H$ are ideals of 
$F_{Jor}\langle X\rangle$ of $s$--elements and such identities of the Jordan 
algebra of ${3\times 3}$ Hermitian matrices $H_3(\mathcal{O}(F))$ over the 
Cayley --- Dickson matrix algebra ${\mathcal{O}(F)=\mathcal{O}_F(0, 1, 1)}$ 
over $F$ that all their partial linearizations are identities of 
$H_3(\mathcal{O}(F))$, $S(R)$ and $T_H(R)$ are ideals of their values 
(images under the action of homomorphisms) on $R$, and, in particular, for 
${R=P(R)}$ 
\begin{gather*}
S(\alpha_1 R)\ =\ \alpha_1 S(R)\ =\ S(R)\,,\quad 
S((\Id_R-\alpha_1) R)\ =\ (\Id_R-\alpha_1) S(R)\ =\ \{0\}\,,
\\
T_H((\Id_R-\alpha_1) R)\ =\ (\Id_R-\alpha_1) T_H(R)\ =\ T_H(R)\,,
\quad 
T_H(\alpha_1 R)\ =\ \alpha_1 T_H(R)\ =\ \{0\}
\end{gather*}
\cite[Introduction]{ZelP1}, \cite[Appendix]{ZelP} ($R$ is a 
homomorphic image of a special Jordan algebra iff ${S(R)=\{0\}}$). 

For linear Jordan algebras, we use the notion of the algebra of 
$\mathfrak{M}$--quotients from \cite{JMQ}, different from Proposition 2.8.
If $R$ is a subalgebra of a Jordan algebra $Q$ over $F$ with $1/2$ and for 
any ${0\ne q\in Q}$ there is ${I\lhd R}$, 
${\Ann_R I=\{x\in R\mid x I=\{x, I, R\}=\{0\}\}=\{0\}}$, 
${\{0\}\ne I q\subseteq R}$, where  ${\{x, y, z\}=(x y) z+(z y) x-(x z) y}$ 
is the triple Jordan product of ${x, y, z\in R}$, then $Q$ is an 
\emph{algebra of $\mathfrak{M}$--quotients of $R$}. Similarly, with the choice 
${I\in \mathcal{F}}$, an algebra of $\mathfrak{M}$--quotients of $R$ with 
respect to the filter $\mathcal{F}$ is defined, where the \emph{filter 
$\mathcal{F}$} is a non-empty set of non-zero ideals of $R$ such that for any 
${I_1, I_2\in \mathcal{F}}$ there exists ${I_3\in \mathcal{F}}$, 
${I_3\subseteq I_1\cap I_2}$, $\mathcal{F}$ is a \emph{power filter} 
if for any ${I\in \mathcal{F}}$ there is ${J\in \mathcal{F}}$, 
${J\subseteq I^2}$. The \emph{maximal algebra of $\mathfrak{M}$--quotients 
$Q_m(R)$} is the algebra of $\mathfrak{M}$--quotients of $R$ in which all its 
algebras of $\mathfrak{M}$--quotients are embedded identically on $R$. 

If ${1\in R}$, $Q$ is an algebra of $\mathfrak{M}$--quotients of $R$ with 
respect to a power filter of ideals with zero annihilators 
$\mathcal{F}$, then $Q$ is an algebra of $\mathfrak{M}$--quotients of $R$ 
with respect to $\mathcal{F}$ in the sense of Proposition 2.8. In this 
case $1$ is the unit of $Q$, since in $Q$ 
\[
[r_x, r_1]\ =\ 0\,,\quad 
(r_y r_z+r_z r_y) r_1+r_{y z}\ =\ r_{y z} r_1+r_y r_z+r_z r_y\quad 
(x, y, z\in R)
\]
\cite[(25--27),~p.~86]{ZShS}, for any ${q\in Q}$, ${I\lhd R}$, 
${(q (r_1-\Id_Q)) I=\{0\}}$,
\begin{gather*}
q (r_1-\Id_Q) r_x r_y\ =\ q (r_1-\Id_Q) (r_x r_y+r_y r_x)\ =\ 
q (r_1-\Id_Q) r_{x y}\ =\ 0\quad (x\in R,\ y\in I)\,,
\\ 
\{0\}\ =\ ((q (r_1-\Id_Q)) R) I\ =\ ((q R) (r_1-\Id_Q)) I\ =\
((q)_R (r_1-\Id_Q)) I\ =\ ((q (r_1-\Id_Q))_R I,
\end{gather*}
${R\cap (q (r_1-\Id_Q))_R\subseteq \Ann_R J=\{0\}}$ for ${(a)_R=a M^Q(R)'}$, 
${a\in Q}$, ${J\in \mathcal{F}}$, ${J q\subseteq R}$, ${r_1=\Id_Q}$. 
At the same time, ${Q_T=(Q, \{\ ,\ ,\ \})}$ is a system of 
$\mathfrak{M}$--quotients of the linear triple Jordan system 
${R_T=(R, \{\ ,\ ,\ \})}$ with respect to $\mathcal{F}$ and for any 
${0\ne q\in Q}$ there is ${I\in \mathcal{F}}$, 
\begin{multline*}
\{0\}\ \ne\ \{(q)_{R_T}, I, R\}+\{(q)_{R_T}, R, I\}+
\{I, (q)_{R_T}, R\}\ \subseteq
\\ 
(\{q, I, R\}+\{q, R, I\}+\{I, q, R\})_{R_T}\ \subseteq\ R
\end{multline*}
for ${(A)_{R_T}=A M^{Q_T}(R_T)', A\subseteq Q, M^{Q_T}(R_T)'=F \Id_Q+
\langle\{x, y,\ \}, \{x,\ ,y\}, \{\ , x, y\}\mid x, y\in R\rangle}$ 
\cite[Proposition~5.2]{JMQ}, 
\cite[the~observation~before~Proposition~3.5]{GolO}, ${\{0\}\ne K (q)_{R_T}\subseteq R}$ for 
${K\in \mathcal{F}}$, ${K\subseteq I^2}$ 
\cite[the~end~of~the~proof~of~Proposition~5.2]{JMQ}, ${(q)_{R_T}=(q)_R}$ ($Q$ and $R$ have a common 1). 
In \cite{GolO} there appears a construction of $Q_m(R)$ for a 
non-degenerate $R$ over $F$ with $1/6$ from \cite{JMQ} with respect to 
$\mathcal{E}(R)$ 
\cite[the~observations~before~Proposition~3.3]{GolO}.

In view of ${I\cap Z(R)\ne \{0\}}$ for any non-degenerate Jordan $PI$--algebra 
$R$ over $F$ with 1/2 and ${\{0\}\ne I\lhd R}$ 
\cite[Theorem~6]{ZelP1}, 
Proposition 2.8 (\cite[Proposition~3.3]{GolO} for $F$ with 1/6), we obtain 

\begin{co}
If ${R=P(R)}$ is a non-degenerate Jordan $PI$--algebra over $F$ with 
1/2, then ${O(R)=R_{\mathcal{F}'}=Q_m(R)}$ is the maximal algebra of 
$\mathfrak{M}$--quotients of $R$.
\end{co}

Lemma 2.4, \cite[Lemma~2.8~and~the~observations~before~it]{GolO} (\cite[Theorem~5]{ZelP1}), the above information about the structure of strongly prime 
alternative $PI$--algebras allows us to deduce 

\begin{co}
If a non-degenerate $PI$--algebra ${R=P(R)=O(R)}$ is alternative (Jordan over 
$F$ with 1/2), ${m=\sup\limits_{R\ne Q\in \Spec_{Mc}(R)} \dim_{K_Q} R_Q< \infty}$, 
then ${m(R)=\max\limits_{B\in \mathcal{U}(R)} \dim_{K_B} R_B}$ is the minimum 
number of generators of the module $R$ over ${\CM(R)=Z(R)}$, ${m(R)\leq m}$, 
where 
\[
R_B\ =\ R/P R\ =\ P(R_B)\,,\quad 
K_B\ =\ Z(R_B)\ =\ \CM(R_B)\ =\ \CM(R)_B
\] 
for any ${B\in \mathcal{U}(R)}$, ${P=B(R)\setminus B}$.
\end{co}

\begin{proof}
The condition ${m< \infty}$ applies only to the Jordan $R$ and assumes 
${\dim_{K_Q} V_Q\leq m-1}$ for all ${R_Q\cong \Jr(V_Q, f_Q)}$, 
${R\ne Q\in \Spec_{Mc}(R)}$; in other cases, it immediately follows 
from the fulfillment of the essential identity on $R$. It remains to 
note that ${I\cap Z(R)\ne \{0\}}$ for any ${\{0\}\ne I\lhd R}$, 
${R=R_{\mathcal{F}'}}$ and ${\CM(R)=Z(R)}$ (the observations before 
Corollary 3.5, before and after Proposition 2.7).

The conclusions of \cite{ShK} and 
\cite[Theorem~4.1~(on~rank),~p.~47]{Raz} 
guarantee the equivalence for the Jordan $R$ of the condition ${m< \infty}$ 
and the finite generation of $R$ as a $\CM(R)$--algebra.
\end{proof}                              

As in Lemma 2.4, having relations for the bases of $R_B$ over $K_B$, 
${B\in \mathcal{U}(R)}$, given by Horn formulas, one can choose 
the generators of the direct summands of $R$ as $\CM(R)$--modules, 
related by the same formulas. 
For any special Jordan algebra $R$ over $F$ with $1/2$ and ${I\lhd R}$ 
we define a descending chain: ${I^{[0]}=I}$, 
${I^{[k+1]}=\{I^{[k]}, I^{[k]}, I^{[k]}\}\lhd R}$, ${k\geq 0}$. 

\begin{prop}
If a Jordan algebra ${R=P(R)=O(R)}$ over $F$ with 1/2 is a subdirect product 
of Albert rings, then $R$ is a free module over ${Z(R)=\CM(R)}$ of rank 27, 
${\overline{C} R\cong H_3(\mathcal{O}(\overline{C}))}$, 
${\overline{C}=\prod\limits_{B\in \mathcal{U}(R)}\overline{Z(R)_B}}$, 
${R\cong H_3(\mathcal{O}(Z(R)))}$ for the algebraic $R$ over the field 
${F=\overline{F}}$. 
\end{prop}

\begin{proof}
By hypothesis, the non-degenerate $PI$--algebra $R$ is a $m.s.p.$--algebra 
with a $PI$--algebra $M(R)$ (the observations before Remark 2.2). If an 
algebra $R_B$ is special for some ${B\in \mathcal{U}(R)}$, then $R_B$ 
satisfies the homogeneous Glennie $s$--identity ${g=0}$,
\begin{multline*}
g(x, y, z)\ =\ 2 \{x, z, x\} \{x, \{z, y^2, z\}, y\}
- 2 \{y, z, y\} \{x, \{z, x^2, z\}, y\}+
\\
\{y, \{z, \{y, \{x, z, x\}, x\}, z\}, y\}
-\{x, \{z, \{x, \{y, z, y\}, y\}, z\}, x\}\ \in\ F_{Jor}\langle X\rangle\,,
\end{multline*}
which is not satisfied in any Albert ring \cite[Theorem~4~(a)]{Gl},  
$G=\{g(x, y, z)\mid x, y, z\in R\}=O(G)\subseteq P R$, 
${P=B(R)\setminus B}$, there is ${\beta\in B}$, ${\beta G=\{0\}}$ 
\cite[Proposition~2.4~(6)]{GolO}, ${g=0}$ holds on 
${\{0\}\ne \beta R\lhd R}$ and there exists ${Q\in \Spec(R)}$, 
$R/Q=(\beta R+Q)/Q\oplus ((1-\beta) R+Q)/Q=(\beta R+Q)/Q$ is an Albert ring 
but satisfies ${g=0}$, a contradiction. Hence, $R_B$ is an Albert algebra over $K_B$ for all 
${B\in \mathcal{U}(R)}$, $R$ is a free module over ${Z(R)=\CM(R)}$ of rank 27 
and a subdirect product of $R_B$, ${B\in \mathcal{U}(R)}$, 
${\overline{C} R\cong H_3(\mathcal{O}(\overline{C}))}$, 
${M(R)=O(M(R))}$ (Lemmas 2.4, 2.9, Corollary 3.6 and their proofs). 

If $F$ is an algebraically closed field and $R$ is algebraic over $F$, then
${K_Q=F}$ for all ${R\ne Q\in \Spec(R)}$ and ${R/Q=R_Q=H_3(\mathcal{O}(F))}$ 
for the Albert ring $R/Q$. Therefore, in this case for any 
${B\in \mathcal{U}(R)}$ in ${R_B\cong H_3(\mathcal{O}(F))}$ the Horm formula 
\begin{gather*}
\Upsilon\ =\ 
(\exists x_1)\ldots (\exists x_{27})[C(x_1, \ldots, x_{27})\wedge 
\Phi_{27}(x_1, \ldots, x_{27})\wedge \Phi_{27}'(x_1, \ldots, x_{27})]\,,
\\ 
C(x_1, \ldots, x_{27})\ =\ \biggl\{\bigwedge_{i, j=1}^{27} 
\biggl(x_i x_j=\sum_{l=1}^{27} \beta_{ij}^l x_l\biggr)\biggr\}\,,
\end{gather*}
is true, where $\Phi_k$, $\Phi_k'$ from Lemma 2.4, integers 
$\{\beta_{ij}^l\}$ are the structure constants of $H_3(\mathcal{O}(F))$ 
in the basis 
${\{e_{ii}, \varepsilon e_{ij}+\overline{\varepsilon} e_{ji}\mid 
i=1,2,3,\ 1\leq i< j\leq 3,\ \varepsilon\in \mathcal{E}\}}$, 
$\{e_{ij}\}$ are the matrix units of $M_3(\mathcal{O}(F))$, 
${\mathcal{E}=\{\varepsilon_{11}, \varepsilon_{22}, 
\varepsilon_{12}^{(k)}, \varepsilon_{21}^{(k)}\mid k=1, 2, 3\}}$ is the 
basis of Cayley --- Dickson matrix units of $\mathcal{O}(F)$ with the 
structure constants and involution ${x\longmapsto \overline{x}}$, 
${x\in \mathcal{O}(F)}$ from \cite[(24), (23),~p.~61]{ZShS}. 
So, $\Upsilon$ is true in ${R\cong H_3(\mathcal{O}(Z(R)))}$ 
\cite[Theorem~2.3.9]{BD}, \cite[Corollary~5.22]{BMO}. 
\end{proof}

If $R$ with 1 over $F$ with 1/2 is commutative, ${U(R)=\{0\}}$, 
${R=Z(R)+V}$, ${{}_{Z(R)} V\subseteq {}_{Z(R)} R}$, 
${x y\in Z(R)}$ for all ${x, y\in V}$, then ${R=\Jr(V, g)}$, 
${g(x, y)=x y}$, ${x, y\in V}$. It is enough to note that 
${(Z(R)\cap V) R\subseteq Z(R), 
(Z(R)\cap V) R\subseteq U(R)=\{0\}, Z(R)\cap V=\{0\}, 
R=Z(R)\oplus V}$ with the Jordan operation  
${(z+x)(z'+y)=(z z'+x y)+(z y+z' x)}$, ${z, z'\in Z(R)}$, ${x, y\in V}$. The 
fulfillment of ${U_z=0}$ for ${z\in Z(R)}$ is equivalent to ${z^2=0}$, 
${U_x=0}$ for all ${x\in V}$, ${x V=\{0\}}$. If ${\prr(Z(R))=\{0\}}$, 
${U_{z+x}=0}$ for some ${z\in Z(R)}$, ${x\in V}$, then 
\begin{gather*}
0\ =\ (z+x)^2\ =\ z^2+x^2+2 z x\ =\ z^2+x^2\ =\ z x\ =\ z^3\ =\ z\ =\ x^2\,,
\\
0\ =\ (y x+z' x) x\ =\ z' x^2+(y x) x\ =\ (y x) x\ =\ (y x)^2\ =\ y x
\quad (z'\in Z(R),\ y\in V)\,.
\end{gather*}
Therefore, the non-degeneracy of $R$ is equivalent to the semiprimeness of 
$Z(R)$ and the non-degeneracy of $g$ on $V$.

\begin{prop}
If a Jordan algebra ${R=P(R)=O(R)}$ over $F$ with 1/2 is a subdirect product 
of strongly prime Jordan algebras of type (2) (and integrity domains), then 
$R=\Jr(V, g)\ne\linebreak Z(R)$ is the Jordan algebra of the non-degenerate 
bilinear symmetric form $g$ over the $\CM(R)$--module $V$ 
(${R=\beta R\oplus (\Id_R-\beta) R}$ with 
${(\Id_R-\beta) R=\Jr(V, g)\ne\linebreak Z((\Id_R-\beta) R)}$, 
$\beta R=Z(\beta R)=Q(\beta R)$ for some ${\beta\in B(R)}$).
\end{prop}

\begin{proof}
Let $R$ be a subdirect product of strongly prime Jordan algebras of type 
(2). Then ${U(R)=\{0\}}$ (in algebras of type (2) there are no non-zero 
associative ideals), ${\CM(R)=Z(R)}$, on ${R=R_{\mathcal{F}'}=Q_m(R)}$ 
${((z [r_x, r_y]^2)(z [r_x, r_y]^3))^{180}=0}$
and all its linearizations are satisfied, for any ${B\in \mathcal{U}(R)}$ the 
central simple algebra $R_B$ over the field $K_B$ is either associative, 
${R_B=K_B}$, or ${R_B=\Jr(V_B, g_B)}$, $g_B$ is a non-degenerate bilinear 
symmetric form on the $K_B$--space $V_B$, ${\dim_{K_B} V_B> 1}$ (Remark 3.5, 
the observations before it and before 
Proposition 2.7, Remark 2.3, \cite[Lemma~2.8]{GolO}, 
\cite[Theorem~1,~Lemma~9]{ZelP1}) and, as a consequence, the Horn formula 
\begin{multline*}
(\forall x)(\forall y)
(\exists a)(\exists b)[(a\in Z(R))\wedge (b\in Z(R))\wedge 
((x-a)^2\in Z(R))\wedge ((y-b)^2\in Z(R))\wedge
\\ 
((x-a)(y-b)\in Z(R))\wedge \mathcal{M}(x, a)\wedge \mathcal{M}(y, b)]
\end{multline*}
is true in $R_B$, 
${\mathcal{M}(x, a)=(\forall c)[(c\notin Z(R))\vee (c (x-a)=0)\vee 
(c (x-a)\notin Z(R))]}$. So, it is true in $R$ 
\cite[Theorem~2.3.9]{BD}, \cite[Corollary~5.22]{BMO}, for any 
${x, y\in R\setminus Z(R)}$ there is $z_x, z_y\in Z(R)$, ${x=z_x+v_x}$, 
${y=z_y+v_y}$, ${v_x^2, v_y^2, v_x v_y\in Z(R)}$, 
${Z(R)\cap Z(R) v_x=Z(R)\cap Z(R) v_y=\{0\}}$, they are uniquely defined, 
since ${(v_x+c)^2\in Z(R)}$, ${Z(R)\cap Z(R)(v_x+c)=\{0\}}$ for 
${c\in Z(R)}$ implies ${c\in \Ann_{Z(R)} v_x}$, 
${(v_x+c) \Ann_{Z(R)} v_x=c \Ann_{Z(R)} v_x=\{0\}}$, ${0=c^2=c}$. If 
$z=z_1 v_x+z_2 v_y\in Z(R)$ for ${z_1, z_2\in Z(R)}$, then 
${(z-z_1 v_x)^2, (z-z_2 v_y)^2\in Z(R)}$, ${0=z z_1 v_x=z z_2 v_y=z^2=z}$. 
Whence it 
follows that ${Z(R)\cap (Z(R) v_x+Z(R) v_y)=\{0\}}$, ${(v_x+v_y)^2\in Z(R)}$, 
${z_{x+y}=z_x+z_y}$, ${v_{x+y}=v_x+v_y}$. These equalities are satisfied on the 
entire $R$ with ${z_x=x}$, ${v_x=0}$ for all $x\in Z(R)$. Thus, 
${R=Z(R)\oplus V=\Jr(V, g)}$ for ${V=\{v_x\mid x\in R\}}$ and non-degenerate 
${g(v, w)=v w}$, ${v, w\in V}$ (the observation before Proposition 3.8).

If $R/P$ and $R/Q$ are subdirect products of integrity domains and strongly 
prime Jordan algebras of type (2), respectively, 
${\{0\}=P\cap Q\ne P, Q\lhd R}$, then there is ${0, 1\ne \beta\in B(R)}$, 
${R=\beta R\oplus (1-\beta) R}$, ${\gamma R=P(\gamma R)=O(\gamma R)}$ is a 
subdirect product of integrity domains for ${\gamma=\beta}$ and strongly 
prime Jordan algebra of type (2) for ${\gamma=1-\beta}$, 
${\beta R=Z(\beta R)=Q(\beta R)}$ and 
${Z((1-\beta) R)\ne (1-\beta) R=Q_m((1-\beta) R)=\Jr(V, g)}$ 
(the end of Sec. 2). 
\end{proof}

As in \cite{ZelP1}, $\{f_i\}$ are all linearizations of 
${((z [r_x, r_y]^2)(z [r_x, r_y]^3))^{180}}$, $F_{SJor}\langle X\rangle$ is 
the \emph{free special Jordan algebra} with the set of free generators $X$ 
over $F$ with 1/2, $F_{SJor}\langle X\rangle=
\langle X\rangle\subseteq F_{Ass}\langle X\rangle^{(+)}$, 
${T=\sum\limits_i (f_i(F_{SJor}\langle X\rangle))_{F_{SJor}(\langle X\rangle)}}$ 
is the ideal of $F_{SJor}\langle X\rangle$ generated by the values of 
$\{f_i\}$ on $F_{SJor}\langle X\rangle$. Let a non-degenerate Jordan algebra $R$ 
be special, $T(R)$ be the ideal of values of $T$ on $R$, 
${T(R)\in \mathcal{E}(R)}$, and so, ${T(R)^{[k]}\in \mathcal{E}(R)}$,
${R\cap \Ann(U, T(R)^{[k]})=\{0\}}$ for any ${k\geq 0}$ and associative 
enveloping algebra $U$ of $R$ (${\Ann_t(U, I), \Ann(U, I)\lhd U}$, 
${t=l, r}$ for all ${I\lhd R}$). Following \cite[\S~4]{ZelP1}, we  
choose among the quotient algebras of the universal associative 
enveloping of $R$ its semiprime enveloping $B$ with involution $*$, 
$T(R)=\Sym(A, *)\lhd R\subseteq \Sym(B, *)$, 
${A=\langle T(R)\rangle\subseteq B}$,
${\bigcup\limits_{k\geq 0} \Ann(B, T(R)^{[k]})=\{0\}}$, 
${I\cap T(R)\ne \{0\}}$ for all $\{0\}\ne I\lhd \Sym(B, *)$, for any 
${x\in B}$ there is ${k\geq 0}$, ${x T(R)^{[k]}+T(R)^{[k]} x\subseteq A}$, 
${\Sym(B, *)}$ inherits the non-degeneracy (and primeness, if any) 
of $R$ \cite[Lemmas~19,~20]{ZelP1} (the non-degeneracy (and primeness) of $R$ 
is inherited by all ${I\lhd R}$, this condition for $R$ is equivalent to its holding in any ${I\in \mathcal{E}(R)}$ (also 
\cite[Corollary~2,~p.~377]{ZShS})). Assume that ${\prr(A)\ne \{0\}}$. Then 
\[
\prr(A)\cap \Sym(B, *)\ =\ \prr(A)\cap T(R)\ \subseteq\ \prr(T(R))\ =\ \{0\}\,,
\]
in case ${(\prr(A))_B\cap \Sym(B, *)\ne \{0\}}$ we can apply 
\cite[the~reasoning~before~Lemma~21]{ZelP1} and choose ${0\ne x\in (\prr(A))_B\cap R}$, 
${k\geq 1}$, ${\{T(R)^{[k]}, x, T(R)^{[k]}\}=\{0\}}$, 
\begin{multline*}
0\ =\ a x b+b x a\ =\ a x a\ =\ a x b x a\ =\ 
a x (b_1\cdot b_2)+(b_1\cdot b_2) x a\ =
\\ 
-1/2(b_1 x a b_2+b_2 x a b_1+b_1 a x b_2+b_2 a x b_1)\ =\
-b_1 (x\cdot a) b_2-b_2 (x\cdot a) b_1\ =\\ 
b (x\cdot a) b\ =\ b_1 (x\cdot a) b_2 (x\cdot a) b_1\quad 
(a, b, b_1, b_2\in T(R)^{[k]})\,,
\end{multline*}
where $\cdot$ is the multiplication in $B^{(+)}$,
\begin{gather*}
\{0\}\ =\ T(R)^{[k]} U_x U_{T(R)^{[k]}}\ =\ 
T(R)^{[k]} U_x U_{T(R)^{[k]}} U_x\ =\ T(R)^{[k]} U_{T(R)^{[k]} U_x}\ =\ 
T(R)^{[k]} U_x\,,
\\
\{0\}\ =\ T(R)^{[k]} U_{x\cdot T(R)^{[k]}}\ =\ x\cdot T(R)^{[k]}\ =\ 
x (T(R)^{[k]})^2+(T(R)^{[k]})^2 x
\end{gather*}
(${Mc(T(R)^{[k]})=\{0\}}$), ${x\in \Ann(B, T(R)^{[k+1]})=\{0\}}$ and hence 
${(\prr(A))_B\cap \Sym(B, *)=\{0\}}$. If ${J\lhd B}$, 
${J\cap \Sym(B, *)=\{0\}}$, then ${u+u*, u u^*=-u^2\in J\cap \Sym(B, *)=\{0\}}$ 
for any $u\in J\cap J^*$, ${\{0\}=(J\cap J^*)^3=J\cap J^*}$, due to 
\[
a b+b a\ =\ 0\,,\quad a b c\ =\ -c (a b)\ =\ -a (c b)\ =\ c a b\ =\ 0\quad 
(a, b, c\in J\cap J^*) 
\]
\cite[Lemma~21~(a)]{ZelP1}. In particular, ${(\prr(A))_B=(\prr(A))_B^*=\{0\}}$. 
In the spirit of \cite[Example~2.11]{JMQ} and the related conclusions of 
\cite{GT}, one can make

\begin{rem}
If $A$ is a semiprime associative algebra with involution $*$ over 
$F$ with $1/2$, then ${\Sym(I, *)\in \mathcal{E}(\Sym(A, *))}$, 
${(J)_A\in \mathcal{E}(A)}$ for all ${I\in \mathcal{E}(A)}$, 
${J\in \mathcal{E}(\Sym(A, *))}$, ${\Sym(Q^s_m(A), *)}$ is an algebra 
of $\mathfrak{M}$--quotients of ${\Sym(A, *)}$ (of any 
${J\in \mathcal{E}(\Sym(A, *))}$ when $A=\langle \Sym(A, *)\rangle$). 
\end{rem}

\begin{proof} 
Remind that  
\begin{multline*}
\mathcal{E}(A)_*\ =\ \{I\in \mathcal{E}(A)\mid I=I^*\}\ =
\\
\{I=I^*\lhd A\mid I\cap J\ne \{0\}\ \forall\, \{0\}\ne J=J^*\lhd A\}\ =\ 
\{I\cap I^*\mid I\in \mathcal{E}(A)\}
\end{multline*}
\cite[the~observations~before~Remark~2.23]{GolA}, the algebra ${\Sym(A, *)}$ 
is non-degenerate, since the presence of 
${a=a^*\in \Sym(A, *)}$, ${U_a=\{0\}}$, implies 
${a \Sym(A, *) a=0}$ and for all ${x, y\in A}$ 
\begin{multline*}
0\ =\ 1/4 a (x-x^*) a (x-x^*) a\ =\ 
1/4 a ((x+x^*)+(x-x^*)) a ((x+x^*)+(x-x^*)) a\ =
\\ 
a x a x a\ =\ a y a x a x a\ =\ -a x a y a x a\ =\ a x a\ =\ a\,. 
\end{multline*}
Moreover, ${I\cap I^*=\{0\}}$ for all 
${I\lhd A}$, ${\Sym(I, *)=\{0\}}$, and therefore 
${\Sym(I, *)\ne \{0\}}$ for any ${I\in \mathcal{E}(A)}$ (see above).

If ${I\in \mathcal{E}(A)}$, ${J\lhd \Sym(A, *)}$, ${J\cap \Sym(I, *)=\{0\}}$, 
then ${b \Sym(I, *) b=\{0\}}$, ${b x b I b x b=0}$ for all ${b\in J}$, 
${x\in I'=I\cap I^*}$ (see the previous discussion), 
\[
\{0\}\ =\ ((b x b)_A I')^2\ =\ (b x b)_A I'\ =\ b I' b\ =\ 
((b)_A I')^2\ =\ (b)_A I'\ =\ (b)_A\ =\ J\,, 
\]
${\Sym(I, *)\in \mathcal{E}(\Sym(A, *))}$.
If ${J\in \mathcal{E}(\Sym(A, *))}$, ${\{0\}\ne I=I^*\lhd A}$, then 
$\{0\}\ne \Sym(I, *)\cap J\subseteq I\cap (J)_A$, 
${(J)_A\in \mathcal{E}(A)_*}$. 
The algebra $Q^s_m(A)$ is semiprime with involution $*$ uniquely extended 
from $A$. If ${q\in \Sym(Q^s_m(A), *)}$, 
${I\in \mathcal{E}(A)}$, ${q \Sym(I, *) q=\{0\}}$, then without loss of 
generality ${I=I^*}$, ${q I+I q\subseteq A}$ and hence
${q x q I q x q=\{0\}}$ for any ${x\in I}$ (see the discussion above; 
${(x-x^*) q (x-x^*)\in \Sym(I, *)}$), 
${\{0\}=(q x q I)^2=q x q I=q I q=(q I)^2=q I}$, ${q=0}$. 
If ${q\in \Sym(Q^s_m(A), *)}$, ${I\in \mathcal{E}(A)_*}$, 
${q I+I q\subseteq A}$ and ${q\cdot \Sym(I, *)=\{0\}}$, then  
\begin{gather*}
0\ =\ q x+x q\ =\ q x^2+x^2 q\ =\ 2 q x^2\ =\ 2 x^2 q\ =
q (x\cdot y)\ =\ (x\cdot y) q\quad (x, y\in \Sym(I, *))\,,
\\
\{0\}\ =\ 
q \underbrace{\Sym(A, *)^1\cdots \Sym(A, *)^1}_k \Sym(I, *)^{[1]}\ =\ 
\Sym(I, *)^{[1]} \underbrace{\Sym(A, *)^1\cdots \Sym(A, *)^1}_k q
\end{gather*}
for all ${k\geq 0}$ and, due to ${q \Sym(I^2, *) q\subseteq \Sym(A, *)}$, 
${\Sym(I, *)^{[1]}\in \mathcal{E}(\Sym(A, *))}$, ${I^2\in \mathcal{E}(A)}$,
\begin{multline*}
\{0\}\ =\ ((q \Sym(I^2, *) q)_{\Sym(A, *)}\cap \Sym(I, *)^{[1]})^2\ =
\\ 
(q \Sym(I^2, *) q)_{\Sym(A, *)}\cap \Sym(I, *)^{[1]}\ =\ 
(q \Sym(I^2, *) q)_{\Sym(A, *)}\,,\quad q\ =\ 0\,.
\end{multline*}
Hence ${\Sym(Q^s_m(A), *)}$ is an algebra of $\mathfrak{M}$--quotients of 
${\Sym(A, *)}$. 

If ${A=\langle \Sym(A, *)\rangle}$, ${J\in \mathcal{E}(\Sym(A, *))}$, then 
${(J)_A=J+A J=J+J A\in \mathcal{E}(A)}$, 
\begin{gather*}
\Ann_t(Q^s_m(A), J)\ =\ \Ann_t(Q^s_m(A), (J)_A)\ =\ 
\{0\}\quad (t=l, r)\,,
\\ 
\{q\in Q^s_m(A)\mid q\cdot J=\{0\}\}\ \subseteq\ 
\Ann(Q^s_m(A), J^{[1]})\ =\ \{0\}\,.
\end{gather*}
As a consequence, if ${q\in \Sym(Q^s_m(A), *)}$, ${I\in \mathcal{E}(A)_*}$, 
${q I+I q\subseteq A}$, then ${q\cdot K\subseteq \Sym(A, *)}$, 
${q\cdot K^{[1]}\subseteq K}$ for ${K=J\cap \Sym(I, *)}$, 
${K^{[1]}\in \mathcal{E}(\Sym(A, *))\cap \mathcal{E}(J)}$, due to 
\begin{multline*}
q\cdot \{a, b, c\}\ =\ 
1/4 (q (a b c+c b a)+(a b c+c b a) q)\ =
\\ 
1/2((q\cdot a) b c-a (q\cdot b) c+a b (q\cdot c)+
(q\cdot c) b a-c (q\cdot b) a+c b (q\cdot a))\ =
\\ 
\{q\cdot a, b, c\}-\{a, q\cdot b, c\}+\{a, b, q\cdot c\}\in K
\quad (a, b, c\in K)\,,
\end{multline*}
and ${q\cdot K^{[1]}\ne \{0\}}$ for ${q\ne 0}$. So,  
${\Sym(Q^s_m(A), *)}$ is an algebra of $\mathfrak{M}$--quotients of $J$.
\end{proof}

\begin{lemma}
Under the conditions for 
${T(R)=\Sym(A, *)\lhd R\subseteq \Sym(B, ^*)}$ (see above)
\begin{gather*}
\Ann(B, J)\ =\ \Ann_t(B, J)\ =\ \Ann_t(B, (J)_B)\ =\ \Ann(B, (J)_B)\quad 
(J\lhd R,\ t=l, r)\,,
\\
I\cap R\in \mathcal{E}(R)\,,\quad (J)_B\in \mathcal{E}(B)_*\,,\quad 
(H)_B\in \mathcal{E}((A)_B)_*\quad 
(I\in \mathcal{E}(B),\ J\in \mathcal{E}(R),\ H\in \mathcal{E}(T(R)))\,,
\\
\Sym(B, *)\ \hookrightarrow\ \Sym(Q^s_m(A), *)\,,\quad  
\Sym(B, *)\ \hookrightarrow\ 
\Sym(Q^s_m((A)_B), *)\ \hookrightarrow\ \Sym(Q^s_m(B), *)\,, 
\end{gather*}
${\Sym(Q^s_m(B), *)}$ is an algebra of $\mathfrak{M}$--quotients of any 
${S\in \mathcal{E}(\Sym(B, *))}$ and ${\Sym(B, *)}$ is an algebra of 
$\mathfrak{M}$--quotients of $R$, ${\Sym((A)_B, *)}$, $T(R)$. 
\end{lemma}

\begin{proof}
Equalities for ${J\lhd R}$ follow from  
${(J)_B=J+B J=J+J B}$ in ${B=\langle R\rangle}$, $\Ann(B, I)=\Ann_t(B, I)$, 
${t=l, r}$ for all ${I\lhd B}$ (semiprimeness of $B$). 
Since for any ${x\in B}$ 
there is ${k\geq 0}$, ${x T(R)^{[k]}+T(R)^{[k]} x\subseteq A}$, 
\[
x (T(R)^{[k]})_A+(T(R)^{[k]})_A x\ \subseteq\ A\,,\quad 
x (T(R)^{[k]})_B+(T(R)^{[k]})_B x\ \subseteq (A)_B 
\]
with ${x T(R)^{[k]}\ne \{0\}}$ and ${T(R)^{[k]} x\ne \{0\}}$ for ${x\ne 0}$ 
(properties of $B$). If ${x\cdot T(R)^{[k]}=\{0\}}$ for ${x\in B}$, 
${k\geq 0}$, then ${x\in\Ann(B, T(R)^{[k]}\cdot T(R)^{[k]})\subseteq 
\Ann(B, T(R)^{[k+1]})=\{0\}}$. 

Therefore, if ${x\in \Sym(B, *)}$, ${k\geq 0}$, 
${x T(R)^{[k]}+T(R)^{[k]} x\subseteq A}$, then 
${x\cdot T(R)^{[k]}\subseteq T(R)}$ and 
${x\cdot \Sym((T(R)^{[k]})_B, *)\subseteq \Sym((A)_B, *)}$,  
${\{0\}\ne x\cdot T(R)^{[k]}\subseteq x\cdot \Sym((T(R)^{[k]})_B, *)}$ 
for ${x\ne 0}$, ${\{0\}\ne T\cap T(R)\lhd R, T(R)}$ for all 
${\{0\}\ne T\lhd \Sym(B, *)}$ \cite[Lemma~20]{ZelP1}. 

If ${I\in \mathcal{E}(B)}$ and ${S\lhd R}$, ${S\cdot (I\cap R)=\{0\}}$, then 
${(I\cap R)_B\subseteq \Ann(B, (S^{[1]})_B)}$, for any $x\in \Sym(I, *)$ 
there is ${k\geq 0}$, ${x\cdot T(R)^{[k]}\subseteq I\cap T(R)\subseteq 
\Ann(B, (S^{[1]})_B)}$, ${x U_{T(R)^{[k]}}\subseteq I\cap T(R)}$ and 
\begin{gather*}
\{0\}\ =\ (S^{[1]}_B)(x y+y x)\ =\ (S^{[1]})_B) y x y\ =\ 
(S^{[1]})_B x y^2\quad (y\in T(R)^{[k]})\,,
\\
\{0\}\ =\ (S^{[1]})_B x (T(R)^{[k]}\cdot T(R)^{[k]})\ =\ 
(S^{[1]})_B x T(R)^{[k+1]}\ =\ (S^{[1]})_B x\,, 
\\
\{0\}\ =\ (S^{[1]})_B \Sym(I, *)\ =\ (S^{[1]})_B (\Sym(I, *))_B\ =\ 
(S^{[1]})_B\ =\ S
\end{gather*} 
(${\prr(B)=Mc(R)=\{0\}}$, ${\Sym(I, *)\in \mathcal{E}(\Sym(B, *))}$, 
${(\Sym(I, *))_B\in \mathcal{E}(B)_*}$ and Remark 3.9),  
${I\cap R\in \mathcal{E}(R)}$. For any ${J\in \mathcal{E}(R), 
H\in \mathcal{E}(T(R)), \{0\}\ne K=K^*\lhd B, \{0\}\ne P\lhd (A)_B}$ 
we have ${\{0\}\ne \Sym(K, *)\lhd \Sym(B, *)}$, 
${K\cap T(R)\cap J\ne \{0\}}$, ${(J)_B\in \mathcal{E}(B)_*}$, 
${(H)_A\in \mathcal{E}(A)_*}$, $P\cap A\cap (H)_A\ne \{0\}$,
${(H)_B\in \mathcal{E}((A)_B)_*}$, 
${\Sym((H)_B, *)\in \mathcal{E}(\Sym((A)_B, *))}$ 
(see above, Remark 3.9). 

So, ${(T(R)^{[k]})_B\in \mathcal{E}((A)_B)_*\cap \mathcal{E}(B)_*}$, 
${\Sym((T(R)^{[k]})_B, *)\in \mathcal{E}(\Sym((A)_B, *))}$ for all 
${k\geq 0}$ (${T(R)^{[k]}\in \mathcal{E}(T(R))\cap \mathcal{E}(R)}$) 
and ${\Sym(B, *)}$ is an algebra of $\mathfrak{M}$--quotients of $R$, 
${\Sym((A)_B, *)}$, $T(R)$. 

Since ${B=\langle R\rangle}$, ${\Sym(Q^s_m(B), *)}$ is an algebra of 
$\mathfrak{M}$--quotients of all ${S\in \mathcal{E}(\Sym(B, *))}$ 
and, in particular, of all ${(J)_{\Sym(B, *)}}$, ${(H)_{\Sym(B, *)}}$ 
for ${J\in \mathcal{E}(R)}$, ${H\in \mathcal{E}(T(R))}$,
\begin{gather*}
\Ann_t(Q^s_m(B), G)\ =\ \Ann_t(Q^s_m(B), (G)_B)\ =\ 
\{0\}\quad (t=l, r)\,,
\\
\{q\in Q^s_m(B)\mid q\cdot G=\{0\}\}\ \subseteq\  
\Ann(Q^s_m(B), G^{[1]})\ =\ \{0\}\quad 
(G\in \mathcal{E}(R)\cup \mathcal{E}(\Sym(B, *)))
\end{gather*}
(Remark 3.9). In the presence of ${m\geq 0}$, 
${\bigcap\limits_{k\geq 0} T(R)^{[k]}=T(R)^{[m]}}$, for all ${x\in \Sym(B, *)}$
\begin{multline*}
x\cdot T(R)^{[m]}\ \subseteq\ T(R)\,,\quad 
x\cdot T(R)^{[m]}\ =\ q\cdot T(R)^{[m+1]}\ \subseteq
\\ 
\{x\cdot T(R)^{[m]}, T(R)^{[m]}, T(R)^{[m]}\}+
\{T(R)^{[m]}, x\cdot T(R)^{[m]}, T(R)^{[m]}\}\ \subseteq\ T(R)^{[m]}\,,
\end{multline*}
${T(R)^{[m]}\in \mathcal{E}(\Sym(B, *)), q\cdot K^{[1]}\subseteq 
\{\Sym(B, *), K, K\}+\{K, \Sym(B, *), K\}\subseteq K}$ for any 
$q\in \Sym(Q^s_m(B), *)$, ${I\in \mathcal{E}(B)_*}$, 
${q I+I q\subseteq B}$ and ${K=I\cap T(R)^{[m]}, K^{[1]}\in \mathcal{E}(C)}$
(${I\cap R\in \mathcal{E}(R)}$, 
${K, K^{[1]}\in \mathcal{E}(R)\cap \mathcal{E}(\Sym(B, *))}$), 
${\Sym(Q^s_m(B), *)}$ is an algebra of $\mathfrak{M}$--quotients 
of any subalgebra $C$ of ${\Sym(B, *)}$, ${T(R)^{[m]}\subseteq C}$ 
($C$ inherits the non-degeneracy of ${T(R)^{[m]}\in \mathcal{E}(C)}$).

Just as in \cite[\S~4]{ZelP1}, it remains to be noted that $B$ is the symmetric 
Martindale ring of quotients of $A$ ($(A)_B$) with respect to the filter 
of ideals ${\{I_k=(T(R)^{[k]})_A\}_{k\geq 0}\subseteq \mathcal{E}(A)}$ 
(${\{I'_k=(T(R)^{[k]})_B\}_{k\geq 0}\subseteq \mathcal{E}((A)_B)}$), $B$ 
is embedded identically on $A$ ($(A)_B$) in  
${Q^s_m(A)\subseteq Q(A)}$ (${Q^s_m((A)_B)\subseteq Q((A)_B)}$), the 
identification of $B$ with its image in $Q^s_m(A)$ ($Q^s_m((A)_B)$) allows 
us to identify $*$ on $B$ with the restriction of $*$ from $Q^s_m(A)$ 
($Q^s_m((A)_B)$) to $B$ ($*$ extends from $A$ ($(A)_B$) to $Q^s_m(A)$ 
($Q^s_m((A)_B)$) by the rule: ${q^* x=(x^* q)^*}$, ${x\in I}$, 
${I\in \mathcal{E}(A)_*}$ (${I\in \mathcal{E}((A)_B)_*}$), 
${q I+I q\subseteq A}$ (${q I+I q\subseteq (A)_B}$)), 
${Q^s_m((A)_B)\hookrightarrow Q^s_m(B)}$ in accordance with $*$ \cite[the~observations~after~Corollary~2.11~with~the~transition~to~
essential~$*$--invariant~ideals]{GolA}.
\end{proof}

As a consequence, if ${R=Q_m(R)}$ (coincides with all its algebras of 
$\mathfrak{M}$--quotients), then ${R=\Sym(B, *)=\Sym(Q^s_m(B), *)}$. 
Applying \cite[Corollary~3.5]{SAm}, we obtain

\begin{co}
If under the conditions of Lemma 3.10 ${R=P(R)=O(R)}$ is a $PI$--algebra, 
then ${R=Q_m(R)=\Sym(Q(B), *)}$, the semiprime associative $PI$--algebra 
${Q^s_m(B)=Q(B)}$ is described by Propositions 3.2, 3.3.
\end{co} 

\begin{rem}
If $Q$ is an algebra of $\mathfrak{M}$--quotients of a non-degenerate 
special Jordan algebra $R$ over $F$ with $1/2$, ${T(R)\in \mathcal{E}(R)}$, 
then $Q$ is non-degenerate, special, ${T(Q)\in \mathcal{E}(Q)}$. 
\end{rem}

\begin{proof}
From the heredity of the radical $Mc$ of Jordan algebras over rings with 1/2 
on subalgebras, its ideal heredity \cite{ZelM}, 
\cite[Addition~1]{GolA}, 
\cite[Corollary~2,~p.~377]{ZShS} and ${\{0\}\ne I\cap R\lhd R}$ for all 
${\{0\}\ne I\lhd Q}$ it follows that
\[
\{0\}\ =\ Mc(Q)\cap Mc(R)\ =\ Mc(Mc(Q)\cap R)\ =\ Mc(Q)\cap R\ =\ Mc(Q)
\] 
(also \cite[Propositions~4.4,~5.2]{JMQ}). Since 
${T(R)\in \mathcal{E}(R)}$, ${\{0\}\ne T(R)\cap I\subseteq T(Q)\cap I}$ 
for all ${\{0\}\ne I\lhd Q}$, ${T(Q)\in \mathcal{E}(Q)}$. The speciality 
of $Q$ follows from \cite[Corollary~3.7]{PIS}.
\end{proof}

The conclusion of Remark 3.12 can easily be transferred from $Q$ to the algebras 
of any chain ${\{Q_{\alpha}\mid \alpha\geq 0\}}$, ${Q_0=\{0\}}$, ${Q_1=R}$, 
$Q_{\alpha}$ is an algebra of $\mathfrak{M}$--quotients of $Q_{\alpha-1}$ and 
${Q_{\alpha}=\bigcup\limits_{\beta< \alpha} Q_{\beta}}$ for the non-limit and 
limit transfinites ${\alpha> 1}$, respectively. The results of \cite{Mont} 
allows us to consider the elements of such a chain as subalgebras of the 
maximal algebra of quotients $Q(R)$ of the algebra $R$ and guarantee its 
stabilization at the step corresponding to a transfinite of cardinality no 
greater than the cardinality of $Q(R)$. 
Therefore, the transition to the case ${R=Q_m(R)}$ is always possible. 
From Remark 3.9 it is easy to deduce its simplified version 

\begin{rem}
If $A$ is a semiprime associative algebra over $F$ with $1/2$, then  
${I^{(+)}\in \mathcal{E}(A^{(+)})}$, ${(J)_A\in \mathcal{E}(A)}$ 
for all ${I\in \mathcal{E}(A)}$, ${J\in \mathcal{E}(A^{(+)})}$, 
$Q^s_m(A)^{(+)}$ is an algebra of $\mathfrak{M}$--quotients of any 
${J\in \mathcal{E}(A^{(+)})}$.
\end{rem}

\begin{proof} 
Due to ${U_x=l_x r_x}$, ${x\in A}$ ($t_x$ in $A$), $A^{(+)}$ is non-degenerate. 
If ${I\in \mathcal{E}(A)}$, ${J\lhd A^{(+)}}$, ${I\cap J=\{0\}}$, then 
${I\subseteq \Ann(A, J\cdot J)}$, ${J\cdot J\subseteq \Ann I=\{0\}}$, 
${J=\{0\}}$ and ${I^{(+)}\in \mathcal{E}(A^{(+)})}$.
Let\linebreak ${J\in \mathcal{E}(A^{(+)})}$. Then 
${\{0\}\ne I\cap J\subseteq I\cap (J)_A}$ for all ${\{0\}\ne I\lhd A}$, 
$(J)_A=J+J A=J+A J\in \mathcal{E}(A^{(+)})$, 
\begin{gather*}
\Ann_t(Q^s_m(A), J)\ =\ \Ann_t(Q^s_m(A), (J)_A)\ =\ \{0\}\quad (t=l, r)\,,
\\
\{q\in Q^s_m(A)\mid q\cdot J=\{0\}\}\ \subseteq\ 
\Ann(Q^s_m(A), J^{[1]})\ =\ \{0\}\,.
\end{gather*}
If ${q\in Q^s_m(A)^{(+)}}$, ${I\in \mathcal{E}(A)}$, ${q I+I q\subseteq A}$, 
then ${I\cap J, (I\cap J)^{[1]}\in \mathcal{E}(A^{(+)})\cap \mathcal{E}(J)}$, 
\[
q\cdot (I\cap J)^{[1]}\ \subseteq\ \{q\cdot (I\cap J), I\cap J, I\cap J\}+
\{I\cap J, q\cdot (I\cap J), I\cap J\}\ \subseteq\ I\cap J\,,
\]
${q\cdot (I\cap J)^{[1]}\ne \{0\}}$ for ${q\ne 0}$. Thus, $Q^s_m(A)^{(+)}$ 
is an algebra of $\mathfrak{M}$--quotients of $J$.
\end{proof}

Consider again $R$, $A$, $B$ from Lemma 3.10. In view of 
\cite[Lemma~20,~the~constructions~before~Lemma~21]{ZelP1}, for the strongly prime $R$, 
the non-primeness of $A$ is equivalent to 
\[
\mathcal{S}\ =\ \{\{0\}\ne I\lhd B\mid \Sym(I, *)=I\cap \Sym(B, *)=\{0\}\}\ \ne\ 
\emptyset
\] 
(${I\cap A, I^*\cap A\ne \{0\}=I\cap I^*}$). We select a $F$--submodule $A'$ 
($A''$) in $A$ generated by associative words of odd (even) length from 
$T(R)$, and write down the arguments of \cite[after~Lemma~21]{ZelP1} in the form 

\begin{rem}
If ${\mathcal{S}\ne \emptyset}$, $P$ is the maximal element of $\mathcal{S}$, 
then ${\{0\}\ne T=(P\cap A')^3}$ is a semiprime subalgebra of $A$, 
${T^{(+)}\cong I=\{x+x^*\mid x\in T\}=\Sym(T+T^*, *)\lhd R}$, there exists a 
homomorphism ${\sigma: B\longrightarrow Q^s_m(T)}$, ${\Ker \sigma=\Ann(B, P)}$. 
For ${\Ker \sigma=P^*}$ (in particular, for the strongly prime $R$ and the 
non-prime $A$) ${\sigma|_{\Sym(B, *)}: \Sym(B, *)\hookrightarrow Q^s_m(T)^{(+)}}$. 
\end{rem}

\begin{proof}
Since for all ${k\geq 0}$, ${x_1, \ldots, x_{2 k+1}, y\in B}$ 
\begin{multline*}
(x_1\ldots x_{2 k+1})\cdot y\ =\ 
(x_1\cdot y) x_2\ldots x_{2 k+1}-1/2 x_1 y x_2\ldots x_{2 k+1}+
1/2 x_1\ldots x_{2 k+1} y\ =
\\
\sum_{i=1}^{2 k+1} (-1)^{i-1} x_1\ldots x_{i-1} (x_i\cdot y)
x_{i+1}\ldots x_{2 k+1}\,,
\end{multline*}
${A'\cdot R\subseteq A'}$, ${(T(R)\cdot T(R)) A=(T(R)\cdot T(R))(A'+A'')\subseteq
A'+T(R) A''\subseteq A'}$, 
where ${x_1\ldots x_i}$ is the product in $B$. 
For any ${y\in P}$ there is ${m\geq 0}$, 
${y T(R)^{[m]}+T(R)^{[m]} y\subseteq A}$, and, as a consequence, 
${\{0\}\ne T(R)^{[1]} T(R)^{[m]} y\subseteq 
(T(R)\cdot T(R)) T(R)^{[m]} y \subseteq P\cap A'}$ for ${y\ne 0}$ (see above; 
similarly ${\{0\}\ne y T(R)^{[m]} (T(R)\cdot T(R))\subseteq P\cap A'}$), 
\begin{gather*}
\{0\}\ \ne\ S\ =\ \{y+y^*\mid y\in P\cap A'\}\ =\ 
\Sym((P\cap A')\oplus (P^*\cap A'), *)\ \lhd\ R\,,
\\
\{0\}\ \ne\ S^{[1]}\ \subseteq\ 
I\ =\ \{x+x^*\mid x\in T\}\ =\ \Sym(T+T^*, *)\ \lhd\ R
\end{gather*}
for ${T=(P\cap A')^3}$ (${M^n=\sum\limits_{a_i\in M} F a_1\ldots a_n}$, 
${{}_F M\subseteq B}$, ${n\geq 1}$; ${(J\cap A')^{2 k+1}\cdot R\subseteq 
(J\cap A')^{2 k+1}}$, ${J\lhd B}$, ${k\geq 0}$; ${P+P^*=P\oplus P^*}$). Due to 
\[
y_1 y_2 y_3 y_4 y_5 y_6\ =\ y_1 y_2 y_3 (y_4 y_5+y_5^* y_4^*) y_6\ \in\ 
y_1 y_2 y_3 T(R) y_6\ \subseteq\ P\cap A'\quad (y_i\in P\cap A')\,,
\]
$T$ is a subalgebra of $B$, ${x\longmapsto x+x^*}$, ${x\in T}$, is an 
isomorphism of $T^{(+)}$ and $I$, $\prr(T)^{(+)}=Mc(T^{(+)})=\{0\}$ 
(ideals of $T$ with zero multiplication are ideals of $T^{(+)}$ with 
zero multiplication). 

Since ${T\cdot R\subseteq T}$ and for all ${a\in R}$, ${b, c\in T}$ 
\[
a b c\ =\ (a b+b a) c-b a c\ =\ (a b+b a) c-b (c a+a c)+b c a\ 
\in\ (b c a+T)\cap (-b c a+T)\,,
\]
${a T^2+T^2 a\subseteq T}$ and for any ${y\in B=\langle R\rangle}$ 
there is ${k\geq 1}$, ${y T^k+T^k y\subseteq T}$. If ${y\in P}$, 
${y T=\{0\}}$ or (and) ${T y=\{0\}}$, then ${y\in \Ann(B, (T)_B)}$, 
there exists ${m\geq 0}$, 
\begin{gather*}
H\ =\ T(R)^{[1]} T(R)^{[m]} y+y T(R)^{[m]} T(R)^{[1]}
\subseteq (P\cap A')\cap \Ann (T)_B\,,
\\
(H)_B^4\ \subseteq\ (H)_B (P\cap A')_B^3\ =\ (H)_B (T)_B\ =\ \{0\}\,,\quad 
H\ =\ \{0\}\,,\quad y\ =\ 0
\end{gather*}
(${(P\cap A')\cdot R\subseteq P\cap A'}$, ${(P\cap A')_B^n=
((P\cap A')^n)_B}$, ${n\geq 1}$). If ${x\in \Ann_l(B, T^s)}$, ${s\geq 1}$, 
then ${P x T^{s-1}\subseteq P\cap \Ann_l(B, T)}$ and, as proven,  
${P x T^{s-1}=\{0\}}$, ${P x T^{s-2}\in P\cap \Ann_l(B, T)}$, etc., 
${P x=\{0\}}$, ${x\in \Ann(B, P)}$. Similarly, 
${\Ann_r(B, T^s)\subseteq \Ann(B, P)}$, ${s\geq 1}$. 

So, ${\sigma: y\longmapsto [(l_y, T^k)], y\in B, k\geq 1, 
y T^k+T^k y\subseteq T,}$ is a homomorphism of $B$ into $Q^s_m(T)$, 
$\sigma(B)$ is the symmetric ring of quotients of $T$ with respect to 
the filter of ideals $\{T^m\}_{m\geq 1}\subseteq \mathcal{E}(T)$, 
${\Ker \sigma=\Ann(B, P)\supseteq P^*}$, ${\Ann(B, P)=P^*}$ is 
equivalent to ${\Sym(\Ann(B, P), *)=\{0\}}$ and in this case  
${\sigma|_{\Sym(B, *)}: \Sym(B, *)\hookrightarrow Q^s_n(T)^{(+)}}$ 
(the choice of $P$). If ${\Ann(B, P)\ne P^*}$, then $R$ is not prime, since 
${(\Sym(\Ann(B, P), *))_B\subseteq \Ann(B, P+P^*)}$, 
$\{0\}=J_1\cdot J_2\ne J_1, J_2\lhd R$, 
${J_1=R\cap (P+P^*)=R\cap \Sym(P+P^*, *)}$, ${J_2=R\cap \Sym(\Ann(B, P), *)}$. 
\end{proof}

From the constructions of \cite[\S~4]{ZelP1}, it also follows

\begin{rem} 
If $R$ is a non-degenerate special Jordan algebra over $F$ with $1/2$, 
${I\in \mathcal{E}(R)}$, $B$ is an associative enveloping of $R$, 
${\prr(B)=\bigcup\limits_{k\geq 0} \Ann(B, I^{[k]})=\{0\}}$, 
${A=\langle I\rangle\subseteq B}$, then ${B\hookrightarrow Q^s_m(A)}$, 
${B\hookrightarrow Q^s_m((A)_B)\hookrightarrow Q^s_m(B)}$.
\end{rem}

\begin{proof}
To construct an associative enveloping of $R$ with the properties of $B$, 
it is sufficient to pass from any of its associative enveloping $B'$ to 
the quotient algebra $B'/J$ by the maximal ideal $J$ among ${J'\lhd B'}$, 
${R\cap J'=\{0\}}$, and identify $R$ with its image in $B'/J$, since 
${I\cap H\ne \{0\}}$ for all ${\{0\}\ne H\lhd B'/J, 
\{0\}\ne (I\cap H)^{[k]}\subseteq I^{[k]}\cap H}$ for all ${k\geq 0}$, 
$I\cap \prr(B'/J)\subseteq \prr(I)\subseteq Mc(I)=\{0\}$, 
${\prr(B'/J)=\{0\}}$ and 
\begin{multline*}
\Ann_t(B'/J, I^{[k]})\ =\ \Ann(B'/J, (I^{[k]})_{B'/J})\quad (t=l, r)\,,
\\
\shoveleft{
\{0\}\ =\ (I^{[k]}\cap \Ann(B'/J, (I^{[k]})_{B'/J}))\cdot 
(I^{[k]}\cap \Ann(B'/J, (I^{[k]})_{B'/J}))\ =}
\\ 
I^{[k]}\cap \Ann(B'/J, (I^{[k]})_{B'/J})\ =\ \Ann(B'/J, (I^{[k]})_{B'/J})\ =\ 
\Ann(B'/J, I^{[k]})\quad (k\geq 0)\,.
\end{multline*}
In relation to $B$, it is enough to note that
\begin{gather*}
\Ann_t(B, I^{[k]})\ =\ \Ann(B, (I^{[k]})_B)\ =\ \{0\}\quad (t=l, r)\,,
\\
\{x\in B\mid x\cdot I^{[k]}=\{0\}\}\subseteq 
\Ann(B, I^{[k+1]})\ =\ \{0\}\quad (k\geq 0)\,,
\end{gather*}
${(I^{[k]})_A\in \mathcal{E}(A)}$, ${(I^{[k]})_B\in \mathcal{E}(B)}$, 
${\Ann_t(B, A)=\Ann_t(B, (A)_B)=\Ann(B, (A)_B)=\{0\}}$ and, as a consequence, 
the left and right exact $A$ has $Q(A)$, $Q^s_m(A)$.
Since for any ${x\in B}$ there is ${k\geq 0}$, 
${x I^{[k]}+I^{[k]} x\subseteq x (I^{[k]})_A+(I^{[k]})_A x\subseteq A}$, 
${x (I^{[k]})_B+(I^{[k]})_B x\subseteq (A)_B}$ (the conclusion is similar 
to \cite[Lemma~20]{ZelP1}), ${B\hookrightarrow Q^s_m(A)}$, 
${B\hookrightarrow Q^s_m((A)_B)\hookrightarrow Q^s_m(B)}$ (Lemma 3.10 and 
its proof).
\end{proof}

Return to the decomposition of the non-degenerate Jordan algebra 
${R=P(R)}$ over $F$ with $1/2$ in the direct sum ${R=R_1\oplus R_2}$ of 
the exceptional ${R=\alpha_1 R}$ and special ${R_2=(\Id_R-\alpha_1) R}$ 
ideals for ${\alpha_1\in B(R)}$ (see above). We choose the maximal ideal 
$M$ among all ${I\lhd R_2}$, 
$I\cap T(R_2)=\{0\}$, ${T(R_2)=(\Id_R-\alpha_1) T(R)}$, together with 
${\beta=\beta (\Id_R-\alpha_1)\in B(R)}$, ${\beta|_M=\Id_M}$, 
${\beta T(R_2)=\{0\}}$. Then ${R_2=R_{21}\oplus R_{22}}$ for 
${R_{21}=\beta R\supseteq M}$, ${T(R_{21})=\beta T(R)=\beta T(R_2)=\{0\}}$ 
and ${R_{22}=(\Id_R-\beta-\alpha_1) R\supseteq T(R_2)}$, 
${T(R_2)\in \mathcal{E}(R_{22})}$ (${J\lhd R_2, R}$ for all 
${J\lhd R_{22}}$; if ${J\lhd R_{22}}$, ${J\cap T(R_2)=\{0\}}$, then 
${(J+M)\cap T(R_2)=\{0\}, J=\{0\}}$), $R_{21}$ is a subdirect product 
of its strongly prime quotient algebras, each of which is either of type 
(2), or an integrity domain \cite[Theorem~1]{ZelP1}, $R_{22}$ is under the 
conditions of Lemma 3.10 (for ${R=R_{22}}$). If $R=P(R)=O(R)$, then 
${R'=P(R')=O(R'), R'=R_1, R_2, R_{21}, R_{22}}$, and $R_1$, $R_{21}$ from 
Propositions 3.7, 3.8 (the observations at the end of Sec. 2). 
Corollary 3.5 allows us to obtain 

\begin{co}
If ${R=P(R)=O(R)}$ is a non-degenerate Jordan $PI$--algebra over $F$ with $1/2$, 
then ${R=Q_m(R)=R_1\oplus R_{21}\oplus R_{22}}$, ${R'=P(R')=O(R')=Q_m(R')}$, 
${R'=R_1, R_{21}}$, $R_{22}$ are described by Propositions 3.7, 3.8 and 
Corollary 3.11. 
\end{co}

In concluding this series of comments on conclusions of \cite{ZelP1}, we note 
that our choice of algebras of $\mathfrak{M}$--quotients is justified by the 
simplicity of the construction and the sequence of presentation, but 
more correct, in our opinion, is the description of such decompositions in 
terms of the maximal algebra of quotients from \cite{Mont, Mont1}. 

If $R$ is a prime non-Lie Mal'tsev $F$--algebra without $2$--torsion, then 
up to isomorphism ${P(R)=\mathcal{M}_{\CM(R)}(\mu, \beta, \gamma)}$ 
for some ${0\ne \mu, \beta, \gamma\in \CM(R)}$ \cite{Fil1, Fil2}, where 
${\mathcal{M}_{\CM(R)}(\mu, \beta, \gamma)}$ is a simple subalgebra of the 
Mal'tsev algebra ${\mathcal{O}_{\CM(R)}(\mu, \beta, \gamma)^{(-)}}$ of the 
Cayley --- Dickson algebra ${\mathcal{O}_{\CM(R)}(\mu, \beta, \gamma)}$ 
over the field $\CM(R)$, which consists of elements of 
${\mathcal{O}_{\CM(R)}(\mu, \beta, \gamma)}$ with zero trace, 
${\dim_{\CM(R)} \mathcal{M}_{\CM(R)}(\mu, \beta, \gamma)=7}$. 
Since below we consider Mal'tsev algebras over $F$ with $1/2$ (even $1/6$), 
the Mal'tsev $F$--algebra ${\mathcal{M}_F(\mu, \beta, \gamma)}$,  
${\mu, \beta, \gamma\in F}$, 
$\Ann_F (4 \mu+1)=\Ann_F \beta=\Ann_F \gamma=\{0\}$, can be defined as a 
free $F$--module with basis ${\{e_1, \ldots, e_7\}}$,
\begin{gather*}
e_i\cdot e_j\ =\ [e_i, e_j]\ =\ 
\begin{cases}
0\ \text{for ${i=j}$;}
\\
\alpha_{ij}(\mu, \beta, \gamma) e_{m_{ij}}\ 
\text{for ${1\leq i\ne j\leq 7}$,}
\end{cases}
\\
\alpha_{ij}(\mu, \beta, \gamma)\ =\ -\alpha_{ij}(\mu, \beta, \gamma)\ \in\ 
\{\pm 2, \pm 2 \tau, \pm 2 \tau \tau'\mid \tau\ne \tau',\ 
\tau, \tau'\in \{\alpha=(4 \mu+1)/4, \beta, \gamma\}\}\,,
\end{gather*}
${m_{ij}=m_{ji}\ne i=m_{m_{ij}j}, j}$, 
${\{m_{ij}\mid 1\leq j\ne i\leq 7\}=\{j\mid 1\leq j\ne i\leq 7\}}$, 
${\{e_0=1, e_1, \ldots, e_7\}}$ is the basis of the 
free $F$--module ${\mathcal{O}_F(\mu, \beta, \gamma)}$ with the 
structural constants in it from \cite[Example~3,~p.~48]{ZShS}. 

For ${F=Q(F)}$ and invertible ${\mu, \beta, \gamma\in F}$, the ideals of 
${R=\mathcal{M}_F(\mu, \beta, \gamma)}$ have the form $I R$, ${I\lhd F}$, 
since ${\{\alpha_{ij}(\mu, \beta, \gamma)\}}$ are invertible in $F$ and 
for any ${x=f_1 e_1+\ldots+f_7 e_7}$, ${f_i\in F}$,
\begin{gather*}
x r_{e_i}^2\ =\ \sum_{1\leq j\ne i\leq 7} f_j \alpha_{ji}(\mu, \beta, \gamma)
\alpha_{m_{ij} i}(\mu, \beta, \gamma) e_j\,,\quad 
[r_{e_i}^2, r_{e_j}^2]\ =\ 0\quad (i, j=1, \ldots, 7)\,,
\\
\delta_i\ =\ 
\biggl(\prod_{1\leq j\ne i\leq 7} \alpha_{i j}(\mu, \beta, \gamma)
\alpha_{m_{i j} j}(\mu, \beta, \gamma)\biggr)^{-1} 
\prod_{1\leq j\ne i\leq 7} r_{e_j}^2\,,\quad x \delta_i\ =\ f_i e_i\quad 
(i=1, \ldots, 7)\,,
\\
(x)_{\mathcal{M}_F(\mu, \beta, \gamma)}\ =\ (F f_1+\ldots+F f_7) R\ =\ 
\sum_{i=1}^7 (F f_1+\ldots+F f_7) e_i\,.
\end{gather*}
If ${\phi\in \Hom(I R, R)_{M(R)'}}$, ${I\lhd F}$, then for any ${f\in I}$ 
there is ${h\in F}$, 
\begin{gather*} 
\phi(f e_1)\ =\ \phi(f e_1 \delta_1)\ =\ \phi(f e_1) \delta_1\ =\ h e_1\,,
\\
\phi(f e_1 r_{\alpha_{1 i}(\mu, \beta, \gamma)^{-1} e_i})\ =\ 
\phi(f e_{m_{1 i}})\ =\ 
\phi(f e_1) r_{\alpha_{1 i}(\mu, \beta, \gamma)^{-1} e_i}\ =\ h e_{m_{1 i}}\quad 
(i=2, \ldots, 7)\,,
\end{gather*}
${\psi: f\longmapsto h, f\in I}$, is correctly defined, 
${\psi\in \Hom(I, F)_F, \psi=l_g|_I}$ for some ${g\in F}$, 
$\phi=l_g|_{I R}=g \Id_{I R}$. Hence ${R=P(R)=O(R)}$, 
${\CM(R)=F \Id_R\cong F}$ (Remark 1.1).  

\begin{prop}
If a Mal'tsev algebra ${R=P(R)=O(R)}$ over $F$ with $1/6$ is a subdirect 
product of non-Lie prime Mal'tsev algebras and ${M(R)=O(M(R))}$, then  
${R=\mathcal{M}_{\CM(R)}(\mu, \beta, \gamma)}$ for some invertible 
${\mu, \beta, \gamma\in \CM(R)}$. 
\end{prop}

\begin{proof}
By hypothesis, $R$ is a non-degenerate $m.s.p.$--algebra, ${M(R)=O(M(R))}$ 
is a $PI$--algebra (the observations before Remark 2.2, 
\cite[the~proof~of~Theorem~1.3]{Gol5}), all $R_B$, ${B\in \mathcal{U}(R)}$ are non-Lie 
(the beginning of the proof of Proposition 3.7 with $g$ replaced by 
$x (y z)+y (z x)+z (x y)\in F_{Lie}\langle X\rangle$) and strongly prime 
(as in \cite[Lemmas~2.8,~2.9]{GolO}, based on 
\cite[Theorem~1.3]{Gol5} 
(can be from \cite[Lemma~3.8]{GolO} for Mal'tsev algebras as triple Lie 
system)), ${R_B=P(R_B)=\mathcal{M}_{\CM(R)_B}(\mu_B, \beta_B, \gamma_B)}$, 
${0\ne \mu_B, \beta_B, \gamma_B\in \CM(R)_B=\CM(R_B)=Z(R^1_B)}$ for all 
${B\in \mathcal{U}(R)}$, $R$ is a free $\CM(R)$--module of rank 7 and 
a subdirect product of $R_B$, ${B\in \mathcal{U}(R)}$
(Proposition 2.5, Lemmas 2.6, 2.9). Therefore, in all $R^1_B$, 
${B\in \mathcal{U}(R)}$ and $R^1$ the Horn formula 
\begin{multline*}
(\exists x_1)\cdots (\exists x_8)
(\exists \mu)(\exists \beta)(\exists \gamma)
[\Delta(x_1, \ldots, x_8, \mu, \beta, \gamma)\wedge
\\ 
\shoveright{
\Phi_8(x_1, \ldots, x_8)\wedge \Phi_8'(x_1, \ldots, x_8)\wedge 
\Sigma(\mu, \beta, \gamma)]\,,}
\\
\shoveleft{
\Delta(x_1, \ldots, x_8, \mu, \beta, \gamma)\ =}
\\
\shoveright{ 
\Delta_7(x_1, \ldots, x_8)\wedge 
\biggl\{\bigwedge_{1\leq i\ne j\leq 7} 
(x_{i+1} x_{j+1}=\alpha_{ij}(\mu, \beta, \gamma) 
x_{m_{ij}+1})\biggr\}\wedge 
\biggl\{\bigwedge_{i=1}^7 (x_{i+1}^2=0)\biggr\}\,,}
\\
\shoveleft{
\Sigma(\mu, \beta, \gamma)\ =\ (\exists x)(\exists y)(\exists z)(\forall a)
[(\mu\in Z(R))\wedge (\beta\in Z(R))\wedge (\gamma\in Z(R))\wedge (x\in Z(R))
\wedge} 
\\ 
(y\in Z(R))\wedge (z\in Z(R))\wedge (x (4 \mu+1) a=a)\wedge 
(y \beta a=a)\wedge (z \gamma a=a)]\,,
\end{multline*}
is true (Lemma 2.6, \cite[Theorems~3.1.11,~2.3.9]{BD}, 
\cite[Corollary~5.22]{BMO}), ${R=\mathcal{M}_{\CM(R)}(\mu, \beta, \gamma)}$ 
for some invertible ${\mu, \beta, \gamma\in \CM(R)}$.

In view of Lemma 2.6, the condition ${M(R)=O(M(R))}$ is equivalent here to 
${m(R)< \infty}$ and ${R=R^2}$ (${R_B=R_B^2}$ for all ${B\in \mathcal{U}(R)}$).
\end{proof}

\begin{co}
If a Mal'tsev algebra ${R=P(R)=O(R)}$ over a field 
${\mathbb{F}=\overline{\mathbb{F}}}$, ${\Ch \mathbb{F}\ne 2}$, is a 
subdirect product of prime non-Lie Mal'tsev algebras and is 
algebraic over $\mathbb{F}$, then $R=\mathcal{M}_{\CM(R)}(0, 1, 1)$, 
$\CM(R)$ is algebraic over $\mathbb{F}$.
\end{co}

\begin{proof}
Since the non-Lie prime quotient algebras of such an algebra $R$ are equal up 
to isomorphism to ${\mathcal{M}_{\mathbb{F}}(0, 1, 1)}$, 
${R=\mathcal{M}_{\CM(R)}(0, 1, 1)}$ \cite[Lemma~2.2]{Gol5},  \cite[Corollary~of~Theorem~6,~p.~60]{ZShS}, the previous proof and we can assume that the 
basis ${\{e_1, \ldots, e_7\}}$ of the free $\CM(R)$--module $R$ has the 
structure constants from \cite[Sec.~3.2]{SA} 
(\cite[the~observations~after~Theorem~1.4]{Gol5}), in particular, ${e_1 e_2=2 e_2}$. The algebraicity of $R$ 
over $\mathbb{F}$ allows one to choose for any  
${\phi\in \CM(R)}$ such a  
${{}_{\phi} f(t)=t^n+{}_{\phi} f_{n-1} t^{n-1}+
\ldots+{}_{\phi} f_1 t\in \mathbb{F}[t]}$ that 
${e_2 {}_{\phi} f(l_{\phi e_1})={}_{\phi} f(2 \phi) e_2=0}$, 
${{}_{\phi} f(2 \phi)=0}$. Therefore $\CM(R)$ is algebraic over $\mathbb{F}$ and  
modulo the identification of $\mathbb{F}$ with $\CM(R)_B$,
${B\in \mathcal{U}(R)}$, ${|\{\phi+P \CM(R)\mid P\in \Spec(B(R))\}|< \infty}$ 
for all ${\phi\in \CM(R)}$.
\end{proof}

We call a Lie $F$--algebra $R$ a \emph{Lie algebra of classical type} if 
$R$ is a free $F$--module with a Chevalley 
basis of one of the finite-dimensional simple complex Lie algebras $L$, 
${R=L(F)=F\mathbin{\otimes_{\mathbb{Z}}} {}_{\mathbb{Z}} L}$, where  
${}_{\mathbb{Z}} L$ is the integer lattice (subring) of $L$ generated by 
its Chevalley basis. 

For a semiprime Lie algebra $R$ the following conditions are equivalent: 
\begin{enumerate}

\item $R$ is \emph{special} (has an associative enveloping $PI$--algebra); 

\item $R$ is \emph{generalized special} (${\Ad(R)=M(R)}$ is a $PI$--algebra); 

\item the dimensions of the central closures of non-zero prime quotient 
algebras of $R$ (if any) over their Martindale centroids do not exceed 
some ${d(L)> 1}$ 

\end{enumerate}
\cite{BP}, \cite[Theorem~4.1~(on~rank),~p.~47]{Raz}. In terms of Lie algebras 
of classical types, one can obtain the following description 

\begin{prop}
If a Lie algebra ${R=P(R)=O(R)}$ over a field $\mathbb{F}$, ${\Ch \mathbb{F}=0}$, 
is special, ${\Ad(R)=O(\Ad(R))}$, then there are ${0\ne \alpha_i\in B(R)}$, 
${\alpha_i \alpha_j=\delta_{ij} \alpha_i}$, ${\Id_R=\alpha_1+\ldots+\alpha_n}$, 
simple complex Lie algebras $L_i$ of dimension mo greater than $\pideg \Ad(R)$ 
and ${\mathcal{U}_i\subseteq \mathcal{U}(R)}$, ${i=1, \ldots, n}$, ${n\geq 1}$, 
${L_i\not\cong L_j}$, ${\mathcal{U}_i\cap \mathcal{U}_j\ne \emptyset}$ for 
${i\ne j}$ such that 
\begin{enumerate}

\item ${\alpha_i R}$ is a subdirect product of finite-dimensional simple 
Lie algebras ${R_B=P(R_B)}$ over fields ${\CM(R)_B=\CM(R_B)}$, 
${B\in \mathcal{U}_i}$, 
\[
\overline{R_B}\ =\ \overline{\CM(R)_B}\mathbin{\otimes_{\CM(R)_B}} R_B\ \cong\ 
L_i(\overline{\CM(R)_B})\,;
\]

\item ${\overline{C}_i \alpha_i R\cong L_i(\overline{C}_i)}$ for 
${\overline{C}_i=\prod\limits_{B\in \mathcal{U}_i} \overline{\CM(R)_B}\cong 
\alpha_i \overline{C}}$, ${\overline{C}=\prod\limits_{B\in \mathcal{U}} 
\overline{\CM(R)_B}}$, ${\mathcal{U}=\bigsqcup\limits_{i=1}^n \mathcal{U}_i}$, 
\[
R\ =\ \alpha_i R\oplus\ldots \oplus \alpha_n R\ \hookrightarrow\ 
\overline{C} R\ \cong\ 
L_1(\overline{C}_1)\oplus \ldots\oplus L_n(\overline{C}_n)\,.
\]
\end{enumerate}
\end{prop}

\begin{proof}
By hypothesis, $R$ is a non-degenerate Lie $m.s.p.$--algebra with a 
$PI$--algebra $\Ad(R)$, the Lie algebras $R/Q$, $P(R/Q)$ are strongly prime 
$P(R/Q)$ is simple and finite-dimensional over the field $\CM(R/Q)$, 
\begin{multline*}
n_Q\ =\ \dim_{\CM(R/Q)} P(R/Q)\ =\ \pideg \Ad(P(R/Q))\ =
\\ 
\pideg \Ad(R/Q)\ \leq\ \pideg \Ad(R)\ =\ \max_{R\ne Q\in \Spec(R)} n_Q\,,
\end{multline*}
${\Ad(P(R/Q))\cong M_{n_Q}(\CM(R/Q))}$ for all ${R\ne Q\in \Spec(R)}$
(the observations before Remark 2.2, \cite[Sec.~1]{Gol2}), ${R_B=P(R_B)}$ 
is simple and finite-dimensional over the field $\CM(R)_B=\CM(R_B)$ for all 
${B\in \mathcal{U}(R)}$ (Proposition 2.5, \cite[Lemma~2.9]{GolO}). For any 
${\beta\in B(R)}$ 
\begin{gather*}
\beta R\ =\ P(\beta R)\ =\ O(\beta R)\,,\quad 
\beta \CM(R)\ =\ \CM(\beta R)\,,\quad \beta B(R)\ =\ B(\beta R)\,,
\\
R\ =\ \beta R\oplus (\Id_R-\beta) R\,,\quad   
\CM(R)\ =\ \beta \CM(R)\oplus (\Id_R-\beta) \CM(R)\,,
\end{gather*}
${B(R)=\beta B(R)\oplus (\Id_R-\beta) B(R)}$ (the end of Sec. 2; up to 
isomorphism, $O(\beta R)$ is the ortho\-gonal completion of $\beta R$ in $R$). 
If ${\CM(R)\ne Q\in \Spec(\CM(R))}$, then ${Q=P \CM(R)}$, 
$B(R)\ne P=B(R)\cap Q\in \Spec(B(R))$, ${\gamma Q=\gamma \CM(R)}$ 
(${\gamma\in P}$) for one ${\gamma\in \{\beta, \Id_R-\beta\}}$, 
$\CM(R)/Q\cong (\Id_R-\gamma) \CM(R)/(\Id_R-\gamma) Q$ (the proof 
of Lemma 2.9). We identify $\beta \CM(R)$ and $\CM(\beta R)$ (the end of Sec. 2). If ${\gamma=\Id_R-\beta}$, then 
${\beta B(R)\ne \beta P\in \Spec(\beta B(R))}$, 
$\beta\in B=B(R)\setminus P\in \mathcal{U}(R)$, 
${\beta B=\beta B(R)\setminus \beta P\in \mathcal{U}(\beta R)}$. 
If ${B'\in \mathcal{U}(\beta R)}$, then ${\beta\in B'}$ 
(${\alpha=\alpha \beta}$ for all ${\alpha\in \beta B(R)}$) and 
${P'=(\beta B(R)\setminus B')\oplus (\Id_R-\beta) B(R)\in \Spec(B(R))}$, 
$B'=\beta B(R)\setminus \beta P'=\beta (B(R)\setminus P')$, 
${B(R)\setminus P'=B'\oplus (\Id_R-\beta) B(R)\in \mathcal{U}(R)}$. So, 
\[
\mathcal{U}(\beta R)\ =\ \{\beta B\mid \beta\in B\in \mathcal{U}(R)\}\,,\quad
\Spec(\beta B(R))\ =\ \{\beta P\mid \Id_R-\beta\in P\in \Spec(B(R))\}\,.
\]
The presence of ${0\ne \eta_i\in B(R)}$, ${\eta_i \eta_j=\delta_{ij} \eta_i}$, 
${R=\eta_1 R\oplus\ldots\oplus \eta_l R}$, 
$\eta_i R$ is a free $\eta_i \CM(R)$--module of rank $k_i$ with basis 
${\{x_{ij}\}_{j=1}^{k_i}, 1< k_1<\ldots < k_l=m(R)}$, and 
$\bigcap\limits_{P\in \Spec(B(R))} P \CM(R)=\{0\}$, implies that 
\begin{multline*}
\bigcap_{P\in \Spec(B(R))} \eta_i P R\ =\ 
\eta_i \bigcap_{P\in \Spec(B(R))} P R\ =
\\ 
\sum_{j=1}^{k_i} 
\eta_i \biggl(\bigcap_{P\in \Spec(B(R))} P \CM(R)\biggr) x_{ij}\ =\ 
\{0\}\quad (i=1, \ldots, l)\,,
\end{multline*}
${\bigcap\limits_{P\in \Spec(B(R))} P R=\{0\}, R}$ is a subdirect product 
of ${R_B, B\in \mathcal{U}(R)}$ and ${m(R)=\pideg \Ad(R)}$ (the proofs of Lemmas 2.6, 2.9). Therefore, one can choose 
${\alpha_i\in B(R), \alpha_i \alpha_j=\delta_{ij}\alpha_i}$, 
$R=\linebreak\alpha_1 R\oplus\ldots \oplus \alpha_n R$, ${n\geq 1}$, 
$\alpha_i R$ is a subdirect product of Lie algebras  
${R_B\cong (\alpha_i R)_{\alpha_i B}}$ over the field  
${\CM(R)_B\cong (\alpha_i \CM(R))_{\alpha_i B}}$ of dimension $m_i$, 
${\alpha_i\in B\in \mathcal{U}_i, 
\overline{R_B}\cong L_i(\overline{\CM(R)_B})}$ for some simple complex 
Lie algebra $L_i$ of dimension ${m_i\leq m(R)}$, 
${\mathcal{U}_i\subseteq \mathcal{U}(R)}$, 
${L_i\not\cong L_j}$, ${\mathcal{U}_i\cap \mathcal{U}_j=\emptyset}$,  
${i\ne j}$ (the end of Sec. 2). For any ${i=1, \ldots, n}$ 
\[
\alpha_i R\ =\ \alpha_i \eta_1 R\oplus\ldots \oplus \alpha_i \eta_l R\ =\ 
\alpha_i \eta_{j_{i 1}} R\oplus\ldots \alpha_i \eta_{j_{i l_i}} R\,,\quad 
\{j_{i s}\}\ =\ \{j\mid \alpha_i \eta_j\ne 0\}\,, 
\]
$\alpha_i \eta_{j_{i s}} R$ is a free 
$\alpha_i \eta_{j_{is}} \CM(R)$--module of rank $k_{j_{i s}}$, for 
any ${B\in \mathcal{U}_i}$, ${P=B(R)\setminus B}$, there is 
${t=t(P)}$, ${\alpha_i P R=\alpha_i \eta_{i j_t} P R\oplus 
\alpha_i (\Id_R-\eta_{i j_t}) R}$, 
${\dim_{(\alpha_i \CM(R))_{\alpha_i B}} (\alpha_i R)_{\alpha_i B}=
k_{j_{i t}}=m_i}$ and so, ${l_i=1}$, 
${k_{j_{i1}}=m_i}$, ${\alpha_i=\alpha_i \eta_{j_{i 1}}}$, 
${\bigcap\limits_{\genfrac{}{}{0pt}{1}{P=B(R)\setminus B,}
{B\in \mathcal{U}_i}} \alpha_i P \CM(R)=\{0\}}$, 
$\alpha_i \CM(R)$ is a subdirect product of ${(\alpha_i \CM(R))_{\alpha_i B}}$, 
${B\in \mathcal{U}_i}$, ${\overline{C}_i \alpha_i R\cong L_i(\overline{C}_i)}$, 
\[
\overline{C}_i\ \cong\ \prod_{B\in \mathcal{U}_i} 
\overline{(\alpha_i \CM(R))_{\alpha_i B}}\ \cong\ 
\prod_{B\in \mathcal{U}_i} {\alpha_i}_B \overline{\CM(R)_B}\ \cong\
\prod_{B\in \mathcal{U}} {\alpha_i}_B \overline{\CM(R)_B}\ =\ 
\alpha_i \overline{C}\,,
\]
${{\alpha_i}_B=\alpha_i+(B(R)\setminus B)}$ is 1 in $\CM(R)_B$, 
${B\in \mathcal{U}_i}$ and 0 in $\CM(R)_B$, 
${B\in \mathcal{U}\setminus \mathcal{U}_i}$ (Lemma 2.9; 
${\alpha_i=\alpha_i (\Id_R-\alpha_j)\in B(R)\setminus B}$, 
${\alpha_j\in B\in \mathcal{U}_j}$, ${j\ne i}$), 
\[
\overline{C} R\ =\ 
\overline{C} \alpha_1 R\oplus\ldots \oplus \overline{C} \alpha_n R\ \cong\ 
\overline{C}_1 \alpha_1 R\oplus\ldots \oplus \overline{C}_n \alpha_n R\ \cong\ 
L_1(\overline{C}_1)\oplus \ldots \oplus L_n(\overline{C}_n)\,.
\]
Since $R$ is a $m(R)$--generated $\CM(R)$--module and a subdirect product of 
$R_B$, ${B\in \mathcal{U}(R)}$ (see above), 
\[
\overline{C} R\ \hookrightarrow\ \overline{C}' R\ \cong\ 
\prod_{B\in \mathcal{U}(R)} K_{i(B)}(\overline{\CM(R)_B})\ \cong\ 
K_1(\overline{C}'_1)\oplus \ldots \oplus K_m(\overline{C}'_m)\,,
\]
where ${1\leq i(B)\leq m}$, $\{K_i\}_{i=1}^m$ are representatives of all 
classes of isomorphic simple Lie algebras $L$ of dimension not greater than 
$m(R)$, ${\overline{\CM(R)_B}\mathbin{\otimes_{\CM(R)_B}} R_B
\cong M(\overline{\CM(R)_B})}$ for some ${B\in \mathcal{U}(R)}$, ${m\geq n}$, 
${\overline{C}'=\prod\limits_{B\in \mathcal{U}(B)} \overline{\CM(R)_B}}$, 
${\overline{C}'_i=\prod\limits_{B\in \mathcal{U}'_i} \overline{\CM(R)_B}}$, 
\[
\mathcal{U}'_i\ =\ \{B\in \mathcal{U}(R)\mid 
\overline{\CM(R)_B}\mathbin{\otimes_{\CM(R)_B}} R_B\cong 
K_i(\overline{\CM(R)_B})\}\,.
\]
\end{proof} 

\begin{co}
If ${R=P(R)=O(R)}$ is a non-degenerate Mal'tsev $PI$--algebra with an 
algebraic regular representation over a field 
${\mathbb{F}=\overline{\mathbb{F}}}$, ${\Ch \mathbb{F}=0}$, then there are 
${0\ne \nu_i\in B(R)}$, ${\nu_i \nu_j=\delta_{ij} \nu_i}$, 
${R=\nu_1 R\oplus \ldots \oplus\nu_l R}$, ${l\geq 1}$, $\nu_i R$ has the 
form ${\mathcal{M}_{\nu_i \CM(R)}(0, 1, 1)}$ for at most one 
$i$ and ${L_i(\nu_i \CM(R))}$ for all other $i$, $L_i$ is a representative 
of the class of isomorphic simple complex Lie algebras of bounded 
dimension, ${L_i\not\cong L_j}$ for ${i\ne j}$, $\CM(R)$ is algebraic over 
$\mathbb{F}$. 
\end{co}

\begin{proof}
Since each $R_B$, ${B\in \mathcal{U}(R)}$ is a non-degenerate algebra over 
${\CM(R)_B=\mathbb{F}}$ of dimension not greater than some ${n(R)> 1}$, $R_B$ 
is equal up to isomorphism to either ${\mathcal{M}_{\mathbb{F}}(0, 1, 1)}$ or 
$L(\mathbb{F})$ for some simple complex Lie algebra of dimension not greater 
than $n(R)$ \cite[Theorem~2.3]{Gol5} ($R_B$, ${B\in \mathcal{U}(R)}$ are 
strongly prime (the proof of Proposition 3.17)). Consequently, in $R^1_B$ 
for ${R_B=\mathcal{M}_{\mathbb{F}}(0, 1, 1)}$ the Horn formula 
\[
(\exists x_1)\cdots (\exists x_8)
[\Delta(x_1, \ldots, x_8, 0, 1, 1)\wedge
\Phi_8(x_1, \ldots, x_8)\wedge \Phi_8'(x_1, \ldots, x_8)]
\]
is true, and for ${R_B=L_i(\mathbb{F})}$, ${i=i(B)}$, the Horn formula 
\begin{gather*}
(\exists x_1)\cdots (\exists x_{n_i+1}) 
[\Delta_{L_i}(x_1, \ldots, x_{n_i+1})\wedge
\Phi_{n_i+1}(x_1, \ldots, x_{n_i+1})
\wedge \Phi_{n_i+1}'(x_1, \ldots, x_{n_i+1})]\,,
\\
\Delta_{L_i}(x_1, \ldots, x_{n_i+1})\ =\ 
\Delta_{n_i}(x_1, \ldots, x_{n_i+1})\wedge 
\biggl\{\bigwedge_{i, j=1}^{n_i} 
\biggl(x_{i+1} x_{j+1}=\sum_{m=1}^{n_i} \alpha_{ij}^m x_{m+1}\biggr)\biggr\}\,,
\end{gather*}
is true, $\{L_i\}$ are representatives of all classes of isomorphic simple 
complex Lie algebras among Lie $R_B$, ${B\in \mathcal{U}(R)}$, $n_i$ is the 
dimension of $L_i$, ${n_i\leq n(R)}$, $\{\alpha_{ji}^m\}$ are the structure 
constants of $L_i$ in its Chevalley basis, and there are 
${0\ne \nu_i\in B(R)}$, ${\nu_i \nu_j=\delta_{ij} \nu_i}$, 
${R=\nu_1 R\oplus \ldots\oplus \nu_l R}$, ${l\geq 1}$ does not exceed the 
number of the above-mentioned Horn formulas, each of which is true in 
$\nu_i R$ for at most one $i$ (the proofs of Lemma 2.6, Proposition 3.17, \cite[Theorems~3.1.11,~2.3.9]{BD}, \cite[Corollary~5.22]{BMO}). Therefore  
${\nu_i R=\mathcal{M}_{\nu_i \CM(R)}(0, 1, 1)}$ for at most one $i$, 
${\nu_j R=L_{i_j}(\nu_j \CM(R))}$ for other ${j\ne i}$. The presence of 
$\mathfrak{sl}(2)$--triples in the Chevalley bases of $\{L_i\}$ and the 
proof of Corollary 3.18 allow us to find ${x_i, y_i, x_i y_i=2 y_i}$, among 
the generators of the free $\nu_i \CM(R)$--module $\nu_i R$, 
${i=1, \ldots, l}$ and to deduce for any ${\phi\in \CM(R)}$ from the 
algebraicity of ${\phi x}$ over $\mathbb{F}$,  
${x=x_1+\ldots+x_l}$, ${x y=2 y}$ for ${y=y_1+\ldots+y_l}$, the algebraicity 
of $\phi$ over $\mathbb{F}$ (${0=y {}_{\phi}(l_{\phi x})=
\nu_i\, {}_{\phi}f(2 \phi) y={}_{\phi}f(2 \phi) y_i}$, 
${0=\nu_i\, {}_{\phi}f(2 \phi)={}_{\phi}f(2 \phi)}$, ${i=1, \ldots, l}$).
\end{proof}

\begin{co}
If ${R=P(R)=O(R)}$ is a Mal'tsev algebra over a field $\mathbb{F}$, 
${\Ch \mathbb{F}=0}$, with a $PI$--algebra ${M(R)=O(M(R))}$, then there 
exists ${\alpha\in B(R)}$, ${R=\alpha R\oplus (\Id_R-\alpha) R}$, 
$\alpha R$ (${(\Id_R-\alpha) R}$) for ${\alpha\ne 0}$ 
(${\alpha\ne \Id_R}$) is described by Proposition 3.17 (3.19).
\end{co}

\begin{proof}
Such a $m.s.p$--algebra $R$ is non-degenerate, $M(R)$ is a finitely generated 
$\CM(R)$--mo\-dule, $M(\beta R)$ is a finitely generated module over 
${\beta \CM(R)\cong \CM(\beta R)}$, ${M(\beta R)=O(M(\beta R))}$ for all 
${\beta\in B(R)}$ and there is ${\alpha\in B(R)}$, 
${\alpha R=P(\alpha R)=O(\alpha R)}$ for ${\alpha\ne 0}$ is a subdirect 
product of prime non-Lie Mal'tsev algebras, 
$(\Id_R-\alpha) R=P((\Id_R-\alpha) R)=O((\Id_R-\alpha) R)$ for  
${\alpha\ne \Id_R}$ is a special Lie algebra (the proofs of Propositions 
3.17, 3.19, Remark 2.2 and the observations before it, Lemma 2.4, the end of 
Sec. 2).
\end{proof}

\begin{rem}
If $F$ is an algebra over a field $\mathbb{F}$, 
${(R, \{R_{\gamma}\}_{\gamma\in \Gamma})}$ is a homogeneously simple algebra 
over ${\mathbb{F}=\CM_{gr}(R)}$, then the homogeneous ideals of the 
$F$--algebra (${F R=F\mathbin{\otimes_{\mathbb{F}}} R}$, 
$\{F R_{\gamma}\}_{\gamma\in \Gamma}$) have the form ${I R}$, ${I\lhd F}$, 
${\mathcal{E}_{gr}({}_F F R)=\{I R\mid I\in \mathcal{E}(F)\}}$ 
(\cite[Theorem~3.9~(1)]{EMO}).
\end{rem}

\begin{proof}
We identify $R$ and ${1\otimes R}$ ($x$ and ${1\otimes x}$, ${x\in R}$). If 
$\{x_a\}_{a\in A_{\gamma}}$ is a $\mathbb{F}$--basis of $R_{\gamma}$, 
$x=\linebreak \sum\limits_{i=1}^n f_i x_{a_i}\in F R_{\gamma}$, 
${a_i\in A_{\gamma}}$, ${f_i\in F}$, ${n\geq 1}$, then one can choose 
${\gamma_i\in \Gamma}$, ${\delta_i\in M^{F R}(R)'}$, 
$F R_{\beta} \delta_i\subseteq F R_{\beta+\gamma_i}$, 
${x_{a_i} \delta_i\ne 0=x_{a_j} \delta_i}$ for all ${i=1, \ldots, n}$, 
${\beta\in \Gamma}$, ${j\ne i}$. The latter follows from 
\cite[Theorem~3.1~in~the~graded~version]{EMO}: if 
${(A, \{A_{\gamma}\}_{\gamma\in \Gamma})}$ is homogeneously prime, 
${A=P_{gr}(A)}$, then for any ${\gamma\in \Gamma}$, 
${n\geq 1}$ and linearly independent ${a_i\in A_{\gamma}}$, 
${i=1, \ldots, n}$ over $\CM_{gr}(A)$ there are ${\phi\in M_{\alpha}}$, 
${\alpha\in \Gamma}$, ${a_1 \phi\ne 0=a_i \phi}$, ${i=2, \ldots, n}$ 
(${(M(A)', \{M_{\gamma}\}_{\gamma\in \Gamma})}$ from the Introduction). 
As\linebreak in \cite{EMO}, the proof is by induction on ${n\geq 1}$ with an 
obvious basis for ${n=1}$, ${a_1\ne 0}$. Let this already be proven for 
all ${1\leq n< m}$ and ${a_1 \phi=0}$ for any ${\phi\in M_{\beta}}$, 
${\beta\in \Gamma}$ (${\phi\in M(A)'}$), ${a_i \phi=0}$, ${i=2, \ldots, m}$, 
${J=\{\psi\in M(A)'\mid a_i \psi=0,\ i=2, \ldots, m-1\}}$
for ${m> 2}$, ${J=M(A)'}$ for ${m=2}$. Then $J$ is a homogeneous right ideal 
of $M(A)'$, ${a_m J\lhd_{gr} A}$, the homogeneous  
${\psi\in \Hom(a_m J, A)_{M(A)'-gr}}$, 
${\psi: a_m \phi\longmapsto a_1 \phi}$, ${\phi\in J}$, is defined correctly 
and extends to 
$\overline{\psi}\in \CM_{gr}(A)$, ${(\overline{\psi} a_m-a_1) J=\{0\}}$, 
by the induction hypothesis for ${a_2, \ldots, a_{m-1}}$ and 
${\overline{\psi} a_m-a_1}$ there are 
${\tau\in M_{\beta}}$, ${\beta\in \Gamma}$, 
${a_i \tau=0\ne (\overline{\psi} a_n-a_1) \tau}$, ${i=2, \ldots, m-1}$, 
${\tau\in J}$, a contradiction. Therefore $R=(x_{a_i} \delta_i)_R=x_{a_i} \delta_i M(R)'$ 
(homogeneous simplicity of $R$), 
\[
f_i F R\ =\ (x \delta_i)_{F R}\ =\ (f_i x_{a_i} \delta_i)_{F R}\ \subseteq\ 
(x)_{F R}\ =\ (f_1 F+\ldots+f_n F) R\,.
\]
So, if ${J\lhd_{gr} {}_F F R}$, then ${J=I R}$ for  
${I\lhd F}$ generated by the coefficients of the decomposition of the elements 
of $J$ in $\bigsqcup\limits_{\gamma\in \Gamma} \{x_a\}_{a\in A_{\gamma}}$, 
${J\in \mathcal{E}_{gr}({}_F F R)}$ is equivalent to ${I\in \mathcal{E}(R)}$. 
\end{proof}

\begin{rem}
If ${(R, \{R_{\gamma}\}_{\gamma\in \Gamma})}$ is homogeneously semiprime, 
${R=P_{gr}(R)=O_{gr}(R)}$, then 
\[
\PDer_{\CM_{gr}(R)}(I R, R)\ =\ \Der_{\CM_{gr}(R)}(I R)\ =\ 
\{D|_{I R}\mid D\in \Der_{\CM_{gr}(R)}(R)\}\quad (I\lhd \CM_{gr}(R))\,.
\] 
\end{rem}

\begin{proof}
By virtue of ${B_{gr}(R)=B(\CM_{gr}(R))}$, 
${\CM_{gr}(R)=Q(\CM_{gr}(R))=O(\CM_{gr}(R))}$ and 
$I=(I\cap B_{gr}(R)) \CM_{gr}(R)$, any  
${D\in \PDer_{\CM_{gr}(R)}(I R, R)=\Der_{\CM_{gr}(R)}(I R)}$ can be extended to 
${\overline{D}\in \Der_{\CM_{gr}(R)}(O(I) R)}$, ${O(I)\lhd \CM_{gr}(R)}$ (the notation after Remark 3.24), 
\[
\biggl(\sum_{i=1}^n 
\biggl({\sum_{a\in A_i}}^{\perp} \xi_{i a} x_{i a} \biggr) y_i\biggr) 
\overline{D}\ =\ 
\sum_{i=1}^n {\sum_{a\in A_i}}^{\perp} \xi_{i a} (x_{i a} y_i) D\ =\ 
{\sum_{(a_1, \ldots, a_n)\in A}}^{\perp} \xi_{1 a_1}\cdots \xi_{n a_n} 
\sum_{i=1}^n (x_{i a_i} y_i) D
\]
for any dense orthogonal ${\{\xi_{i a}\}_{a\in A_i}\subseteq B_{gr}(R)}$, 
${x_{i a}\in I}$, ${y_i\in R}$, ${n\geq 1}$, and dense orthogonal 
${\{\xi_{1 a_1}\cdots \xi_{n a_n}\}_{(a_1, \ldots, a_n)\in A}\subseteq B_{gr}(R)}$, 
${A=A_1\times \cdots \times A_n}$ (the beginning of Sec. 2, \cite[Sec.~3]{GolO}). This extension is correct, since from 
\[
0\ =\ \sum_{i=1}^n 
\biggl({\sum_{a\in A_i}}^{\perp} \xi_{i a} x_{i a}\biggr) y_i\ =\ 
{\sum_{(a_1, \ldots, a_n)\in A}}^{\perp} 
\xi_{1 a_1}\cdots \xi_{n a_n} \biggl(\sum_{i=1}^n x_{i a_i} y_i\biggr)
\]
it follows that
\begin{gather*}
0\ =\ \xi_{1 a_1}\cdots \xi_{n a_n}
\sum_{i=1}^n x_{i a_i} y_i\ =\ \xi_{1 a_1}\cdots \xi_{n a_n}
\sum_{i=1}^n (x_{i a_i} y_i) D\quad (a_i\in A_i,\ i=1, \ldots, n)\,,
\\
0\ =\ {\sum_{(a_1, \ldots, a_n)\in A}}^{\perp} 
\xi_{1 a_1}\cdots \xi_{n a_n} \sum_{i=1}^n (x_{i a_i} y_i) D\ =\ 
\sum_{i=1}^n {\sum_{a\in A_i}}^{\perp} \xi_{i a} (x_{i a_i} y_i) D\,,
\end{gather*}
and for any dense orthogonal ${\{\xi_a\}_{a\in A}, \{\chi_b\}_{b\in B}
\subseteq \CM_{gr}(R)}$, ${x_a, x'_b\in I}$, 
$x={\sum\limits_{a\in A}}^{\perp} \xi_a x_a,
x'={\sum\limits_{b\in B}}^{\perp} \chi_b x'_b\in O(I)$, ${y, y'\in R}$,
${\alpha\in \CM_{gr}(R)}$ 
\begin{multline*}
(\alpha x y) \overline{D}\ =\ 
\biggl(\biggl({\sum\limits_{a\in A}}^{\perp} \xi_a \alpha x_a\biggr) y\biggr) 
\overline{D}\ =\ {\sum\limits_{a\in A}}^{\perp} \xi_a \alpha (x_a y) D\ =\ 
\alpha (x y) \overline{D}\,,
\\
\shoveleft{
(x y x' y')\overline{D}\ =\ 
\biggl(\biggl({\sum_{(a, b)\in A\times B}}^{\perp} \xi_a \chi_b x_a x'_b\biggr)
y y'\biggr) \overline{D}\ =\ 
{\sum_{(a, b)\in A\times B}}^{\perp} \xi_a \chi_b (x_a y x'_b y') D\ =}
\\ 
{\sum_{(a, b)\in A\times B}}^{\perp} \xi_a \chi_b 
((x_a y) D x'_b y'+x_a y ((x'_b y') D))\ =\ 
(x y)\overline{D} x' y'+x y ((x' y') \overline{D})\,.
\end{multline*}
Since ${\mu=\mu(I\cap B_{gr}(R))=\sup I\cap B_{gr}(R)\in O(I)}$ 
\cite[Lemma~1.17]{BMO}, ${\alpha=\alpha \mu}$ for all ${\alpha\in O(I)}$, 
${\mu=\mu(O(I)\cap B_{gr}(R))}$, ${O(I)=\mu \CM_{gr}(R)}$, ${O(I) R=\mu R}$, 
${R=\mu R\oplus (\Id_R-\mu) R}$ and we just have to extend 
$\overline{D}$ to ${\overline{D}\in \Der_{\CM_{gr}(R)}(R)}$, 
${\overline{D}|_{(\Id_R-\mu) R}=0}$. Similarly, we can deduce 
\[
\Hom_{\CM_{gr}(R)}(I R, R)\ =\ \End_{\CM_{gr}(R)}(I R)\ =\ 
\{\psi|_{I R}\mid \psi\in \End_{\CM_{gr}(R)}(R)\}\,.
\]
\end{proof}

Ungraded versions of Remarks 3.22, 3.23 are obtained by formal transition 
to ${\Gamma=\{0\}}$. Due to ${[\Der(R), \CM(R)]\subseteq \CM(R)}$, 
${[D, \beta]=2 \beta [D, \beta]=0}$ for any ${R=P(R)}$, ${\beta\in B(R)}$, 
$D\in \Der(R)$, ${[\Der(R), B(R)]=\{0\}}$ \cite[Remark~2.4]{GolA}.  

\begin{rem}
If a Lie algebra ${R=P(R)=O(R)}$ over a field $\mathbb{F}$, ${\Ch \mathbb{F}=0}$, 
is special, ${\Ad(R)=O(\Ad(R))}$, then ${\Der_{\CM(R)} R=\ad(R)}$.
\end{rem}

\begin{proof}
Each ${D\in \Der_{\CM(R)}(R)}$ extends to ${D\in \Der_{Z(R^1)}(R^1), 
(\alpha+x) D=x D}$, $\alpha\in \CM(R)=Z(R^1)$, ${x\in R}$. For the 
algebraic system ${R^1=O(R^1)}$ with ${D\in \Der(R^1)}$ in the signature and\linebreak 
Horn formulas $\neg \mathcal{A}$, $\mathcal{A}'$, where $\mathcal{A}$ is a 
hereditary formula (${\gamma Z(R^1)\subseteq \beta Z(R^1)=Z(\beta R^1)}$, 
$\gamma=\gamma \beta$, ${\beta\in B(R)}$), 
\[
\mathcal{A}\ =\ (\forall x)[(x\notin Z(R))\vee (x D=0)]\,,\quad 
\mathcal{A}'\ =\ (\exists x)(\forall y)[(x\in [R, R]_k)\wedge (y D=[y, x])]\,,
\] 
${\mathcal{A}\rightarrow \mathcal{A}'}$ is true in $R^1_B$ for all
${B\in \mathcal{U}(R)}$, ${k=m(R)=\pideg \Ad(R)}$, since the Lie 
algebra ${R_B=P(R_B)}$ over the field ${\CM(R)_B=\CM(R_B)}$ is simple, 
${\dim_{\CM(R)_B} R_B\leq m(R)}$, ${R=R^2}$ and ${R D\subseteq R}$ for all 
${D\in \Der(R^1)}$ (the proofs of Lemma 2.6, Proposition 3.19, 
\cite[Sec.~8.13]{BMO}, \cite[Theorem~6,~p.~87]{JacL}). Therefore 
${\Der_{Z(R^1)}(R^1)=\{[\ , x]\mid x\in R\}}$, ${\Der_{\CM(R)}(R)=\ad(R)}$ 
\cite[Theorem~2.3.9]{BD}, \cite[Sec.~8.13,~Corollary~5.22]{BMO}.
\end{proof}

If ${(R, \{R_{\gamma}\}_{\gamma\in \Gamma})}$, 
${(Q, \{Q_{\gamma\in \Gamma}\})}$ are $\Gamma$--graded Lie algebras, $R$ is
a homogeneous subalgebra of $Q$, for any ${0\ne q\in Q_{\gamma}}$, 
${\gamma\in \Gamma}$ there is ${I\lhd_{gr} R}$, 
${\Ann_R I=\{0\}\ne [I, q]\subseteq R}$ (equivalent to: for any 
${u\in Q_{\alpha}}$, ${v\in Q_{\beta}}$, ${\alpha, \beta\in \Gamma}$, 
${u\ne 0}$, one can choose ${x\in R_{\gamma}}$, ${\gamma\in \Gamma}$, 
${[x, u]\ne 0}$, ${[x, (v)_R]\subseteq R}$), where 
${\Ann_R I=\Ann_t I}$, ${t=l, r}$, ${(v)_R=v \Ad^Q(R)'}$, then $Q$ is a  
\emph{graded algebra of quotients of} $R$. For a homogeneously semiprime 
${(R, \{R_{\gamma}\}_{\gamma\in \Gamma})}$, the \emph{maximal graded algebra 
of quotients ${(Q_{m-gr}(R), \{Q_{m-gr}(R)_{\gamma}\}_{\gamma\in \Gamma})}$}, 
into which the graded algebras of quotients of $R$ are embeddable using  
homogenepus monomorphisms identical on $R$, can be constructed as a 
quotient set ${\mathcal{D}_{gr}(R)/\sim}$ of the set 
${\mathcal{D}_{gr}(R)=\{(D, I)\mid D\in \PDer_{gr}(I, R),\ 
I\in \mathcal{E}_{gr}(R)\}}$, where
${\mathcal{E}_{gr}(R)=\{I\lhd_{gr} R\mid I\in \mathcal{E}(R)\}}$,
\begin{gather*}
\PDer(I, R)\ =\ 
\{D\in \Hom_F(I, R)\mid (x y) D=(x D) y+x (y D)\ \forall x, y\in I\}\,,
\\
\PDer_{gr}(I, R)_{\gamma}\ =\ \{D\in \PDer(I, R)\mid I_{\alpha} D\subseteq 
R_{\alpha+\gamma}\ \forall \alpha\in \Gamma\}\quad (\gamma\in \Gamma)\,,
\\
\ad(R)|_I\ \subseteq\ 
\PDer_{gr}(I, R)\ =\ \bigoplus_{\gamma\in \Gamma} \PDer_{gr}(I, R)_{\gamma}\,,
\end{gather*}
under the equivalence relation: ${(D, I)\sim (D', I')}$ if ${D=D'}$ on 
${I\cap I'}$ (on some ${I''\in \mathcal{E}_{gr}(R)}$, 
${I''\subseteq I'\cap I''}$), with the operations 
\begin{gather*}
[(D, I)]+[(T, J)]\ =\ [(D+T, I\cap J)]\,,\quad f [(D, I)]\ =\ [(f D, I)]\,,
\\
[[(D, I)], [(T, J)]]\ =\ [([D, T], (I\cap J)^2)]
\end{gather*}
${(D, I), (T, J)\in \mathcal{D}_{gr}(R)}$, ${f\in F}$, $\Gamma$--grading 
\[
Q_{m-gr}(R)_{\gamma}\ =\ \{[(D, I)]\mid D\in \PDer_{gr}(I, R)_{\gamma},\ 
I\in \mathcal{E}_{gr}(R)\}\quad (\gamma\in \Gamma) 
\]
and ${R \hookrightarrow Q_{m-gr}(R), x\longmapsto [(\ad_x, R)], x\in R}$.
For an ungraded $R$, the definition of the algebra of quotients and the 
description of the maximal algebra of quotients $Q_m(R)$ of a semiprime 
$R$ are obtained by a formal transition to ${\Gamma=\{0\}}$. The graded 
algebras of quotients of a homogeneously semiprime $R$ are its algebras of 
quotients as an algebra without grading. The homogeneous semiprimeness 
(primeness) of $R$ and for $R$ without $6$--torsion, the homogeneous 
non-degeneracy and non-degeneracy are inherited by its graded 
algebras of quotients \cite[the~observations~before~Proposition~3.1]{GolO}, \cite{JMQ, QTS, OMol, Mol} (homogeneous non-degeneracy of $R$ is the absence of ${0\ne x\in R_{\gamma}}$, ${\gamma\in \Gamma}$, 
${\ad_x (F \Id_R+\ad(R)) \ad_x=\{0\}}$). 

\begin{rem}
If ${R=P(R)=O(R)}$ is a non-degenerate Lie $PI$--algebra with an 
algebraic adjoint representation over a field 
${\mathbb{F}=\overline{\mathbb{F}}}$, ${\Ch \mathbb{F}=0}$, then 
\[
R\ =\ Q_m({}_{\CM(R)} R)\ \cong\ \ad(R)\ =\ \Der(R)\,.
\]
\end{rem}

\begin{proof}
Since for the algebraic system $R$ with ${D\in \Der(R)}$ in the 
signature in all ${R_B, B\in \mathcal{U}(R)}$ the Horn formula 
${(\exists x)(\forall y)[(y D=y x)]}$ is true, it is true in $R$, 
${\Der(R)=\ad(R)}$ (the proof of Corollary 3.20, \cite[Sec.~8.13]{BMO}, \cite[Theorem~6,~p.~87]{JacL}). Moreover, 
${R=\nu_1 R\oplus \ldots\oplus \nu_l R}$, 
\[
\nu_i R\ =\ P(\nu_i R)\ =\ O(\nu_i R)\ =\ L_i(\nu_i \CM(R))\ =\
\nu_i \CM(R)\mathbin{\otimes_{\mathbb{F}}} L_i(\mathbb{F})\ =\
\nu_i \CM(R) L_i(\mathbb{F})\,,
\]
${\nu_i \CM(R)=\CM(\nu_i R)}$ for suitable ${0\ne \nu_i\in B(R)}$, 
${\nu_i \nu_j=\delta_{ij} \nu_i}$, ${l\geq 1}$ and simple complex Lie 
algebras $L_i$ of bounded dimension, ${L_i\not\cong L_j, i\ne j}$ 
(Corollary 3.20 and the end of Sec. 2), the $\CM(R)$--ideals 
($\nu_i \CM(R)$--ideals) of $\nu_i R$ have the form 
${L_i(\nu_i J)=\nu_i J R}$, ${J\lhd \CM(R)}$, and the $\CM(R)$--ideals of 
$R$ have the form  
\[
I R\ =\ \nu_1 I_1 R\oplus \ldots\oplus \nu_l I_l R\ =\ 
L_1(\nu_1 I_1)\oplus \ldots\oplus L_l(\nu_l I_l)\quad 
(I=\nu_1 I_1\oplus\ldots \oplus \nu_l I_l, I_i\lhd \CM(R))
\] 
(Remark 3.22; Martindale centroids of simple finite-dimensional algebras 
over ${\mathbb{F}=\overline{\mathbb{F}}}$ coincide with $\mathbb{F}$). 
Consequently, for all ${H\lhd {}_{\CM(R)} R}$, ${i=1, \ldots, l}$ 
\begin{gather*}
\PDer_{\CM(R)}(H, R)\ =\ \{D|_H\mid D\in \Der_{\CM(R)}(R)\}\,,
\\
\PDer_{\CM(R)}(\nu_i H, \nu_i R)\ =\ 
\{D|_{\nu_i H}\mid D\in \Der_{\CM(R)}(\nu_i R)\}\ =\ 
\{D|_{\nu_i H}\mid D\in \Der_{\CM(R)}(R)\}
\end{gather*}
(Remark 3.23, ${\Der_{\CM(R)}(\nu_i R)=\Der_{\nu_i \CM(R)}(\nu_i R)=
\{D|_{\nu_i R}\mid D\in \Der_{\CM(R)}(R)\}}$), 
\[
R\ =\ Q_m({}_{\CM(R)} R)\ =\ 
Q_m({}_{\CM(R)} \nu_1 R)\oplus\ldots \oplus Q_m({}_{\CM(R)} \nu_l R)\ \cong\ 
\Der(R)\ =\ \ad(R)\,
\]
where $Q_m({}_{\CM(R)} R)$ is the maximal algebra of quotients of the Lie
$\CM(R)$--algebra $R$.
\end{proof}  

Recall that the standard Lie enveloping of a triple Lie system 
${(R, [\ ,\ ,\ ])}$ is a $\mathbb{Z}_2$--graded Lie algebra 
${\mathcal{L}(R)=R\oplus [R, R]}$ with 
\[
\mathcal{L}(R)_1\ =\ R\,,\quad \mathcal{L}(R)_0\ =\ [R, R]\ =\ 
\biggl\{\sum_{i=1}^k [x_i, y_i,\ ]\biggl| 
x_i, y_i\in R,\ k\geq 1\biggr\}\ \subseteq\ \Der(R)
\] 
and the multiplication  
\[
[x_1+D_1, x_2+D_2]\ =\ (x_1 D_2-x_2 D_1)+([x_1, x_2,\ ]+[D_1, D_2])
\quad (x_i\in R,\ D_i\in [R, R])\,,
\]
a $F$--submodule ${I\subseteq R}$ is an ideal of $R$ if 
${[I, R, R]\subseteq I}$, 
\[
\Ann_R I\ =\ \{x\mid [x, I, R]=[x, R, I]=\{0\}\}\ =\ 
R\cap \Ann_{\mathcal{L}(R)} (I)_{\mathcal{L}(R)}\,,
\quad (I)_{\mathcal{L}(R)}\ =\ I\oplus [I, R]\,,
\]
$R$ is \emph{$\mathcal{R}$--semiprime} (\emph{$\mathcal{R}$--prime}) if ${[I, R, I]\ne \{0\}}$ for 
all ${\{0\}\ne I\lhd R}$ ($R$ is $\mathcal{R}$--semiprime and ${I\cap J\ne \{0\}}$ for 
all ${\{0\}\ne I, J\lhd R}$), and \emph{non-degenerate} if ${[x,\ , x]\ne 0}$ 
for all ${0\ne x\in R}$. 

The homogeneous primeness (semiprimeness) of $\mathcal{L}(R)$ is equivalent 
to the $\mathcal{R}$--primeness ($\mathcal{R}$--semiprimeness) of $R$. If $R$ has no 
$6$--torsion, then the non-degeneracy of $R$ is equivalent to the 
non-degeneracy (homogeneous non-degeneracy) of $\mathcal{L}(R)$ and for 
$R$ the following are equivalent: primeness; $\mathcal{R}$--primeness; 
${I\cap J\ne \{0\}}$ for all ${\{0\}\ne I, J\lhd R}$. A triple Lie system 
${(Q, [\ ,\ ,\ ])}$ is a \emph{Lie system of quotients of} ${(R, [\ ,\ ,\ ])}$ 
if $R$ is a subsystem of $Q$ and for any ${0\ne q\in Q}$ there exists  
${I\lhd R}$, ${\Ann_R I=\{0\}}$, ${\{0\}\ne [q, I, R]+[q, R, I]\subseteq R}$. 
The maximal system of quotients $Q_m(R)$ of a $\mathcal{R}$--semiprime $R$, in which the 
systems of quotients of $R$ are embedded identically on $R$, can be described 
up to an isomorphism identical on $R$ as ${Q_m(R)=Q_{m-gr}(\mathcal{L}(R))_1}$ 
\cite[the~observations~before~and~after~Lemma~3.8]{GolO}, \cite{QTS, NLS}. We 
can consider Mal'tsev algebras as triple Lie systems and make 

\begin{rem}
In Corollary 3.18 for ${\Ch \mathbb{F}=0}$ and 3.20
${\mathcal{L}(R_T)=Q_{m-gr}({}_{\CM(R)} \mathcal{L}(R_T))}$ and 
${R_T=(R, [\ ,\ ,\ ])=Q_m({}_{\CM(R)} R_T)}$.
\end{rem}

\begin{proof}
Without loss of generality ${R=\nu_1 R\oplus\ldots \oplus \nu_l R}$, 
${0\ne \nu_i\in B(R)}$, ${\nu_i \nu_j=\delta_{ij} \nu_i}$, 
$\nu_i R=P(\nu_i R)=O(\nu_i R)=L_i(\nu_i \CM(R))$, 
${\nu_i \CM(R)=\CM(\nu_i R)}$ for 
${L_1(\mathbb{F})=\mathcal{M}_{\mathbb{F}}(0, 0, 1)}$ 
and some simple complex Lie algebras $L_i$ of bounded dimension, 
${L_i\not\cong L_j}$, ${2\leq i\ne j\leq l}$, 
and therefore ${R_T=\nu_1 R_T\oplus \ldots\oplus \nu_l R_T}$, 
\[
[x, y, z]\ =\ \begin{cases}
2 (x y) z-(y z) x-(z x) y\,,\ x, y, z\in \nu_1 R;
\\
3 (x y) z\,,\ x, y, z\in (\Id_R-\nu_1) R\,.
\end{cases}
\]

Due to the non-degeneracy of $R$ and $R_T$, 
${L_1(\mathbb{F})=J(L_1(\mathbb{F}))}$, 
${L_i(\mathbb{F})=L_i(\mathbb{F})^2}$, ${i=2, \ldots, l}$, where 
$J(A)$ is a submodule (an ideal for $A$ without $6$--torsion) of 
the Mal'tsev algebra $A$ generated by ${(a b) c+(b c) a+(c a) b}$, 
${a, b, c\in A}$, and for $A$ without $2$--torsion
\[
[r_a, r_{b c}]+[r_b, r_{c a}]+[r_c, r_{a b}]+2 r_{a (b c)+b (c a)+c (a b)}\ =\ 
r_{a (b c)+b (c a)+c (a b)}\quad (a, b, c\in A) 
\]
\cite[Theorem~3.5,~(2.12)]{SA}, ${M(L_i(\mathbb{F}))=M(L_i(\mathbb{F})_T)=
\mathcal{L}(L_i(\mathbb{F})_T)_0, \mathcal{L}(L_i(\mathbb{F})_T)}$ 
is homogeneously simple for ${i=1, \ldots, l}$, 
${\nu_1 R=J(\nu_1 R)}$, ${(\Id_R-\nu_1) R=((\Id_R-\nu_1) R)^2}$, 
\[
M(R)\ =\ M(R_T)\ =\ \mathcal{L}(R_T)_0\,,\quad 
\CM(R)\ =\ \CM(R_T)\ \cong\ \CM_{gr}(\mathcal{L}(R_T))\ =\ 
\CM(R) \Id_{\mathcal{L}(R_T)}\,,
\]
${R_T=P(R_T)}$, ${\mathcal{L}(R_T)=P_{gr}(\mathcal{L}(R_T))}$, 
${(R_T)_B=(R_B)_T}$ is strongly prime,
${\mathcal{L}((R_T)_B)=\mathcal{L}(R_T)_B}$ is non-degenerate and 
homogeneously prime, ${\Der(\mathcal{L}(R_T)_B)=\ad(\mathcal{L}(R_T)_B)}$ 
for all ${B\in \mathcal{U}(R)}$ \cite[Lemma~3.8,~the~observations~before~it]{GolO}, 
\cite[Theorem~6,~p.~87]{JacL}, 
${\Der(\mathcal{L}(R_T))=\ad(\mathcal{L}(R_T))}$ (the proof of 
Remark 3.25). Since ${\mathcal{L}(R_T)\cong 
\mathcal{L}(\nu_1 R_T)\oplus\ldots \oplus \mathcal{L}(\nu_l R_T)}$ and 
\[
\mathcal{L}(\nu_i R_T)\ =\ \mathcal{L}(L_i(\nu_i \CM(R))_T)\ =\ 
\nu_i \CM(R)\mathbin{\otimes_{\mathbb{F}}} 
\mathcal{L}(L_i(\mathbb{F})_T)\ =\ 
\nu_i \CM(R) \mathcal{L}(L_i(\mathbb{F})_T)\,,
\]
${i=1, \ldots, l}$, any ${H\lhd_{gr} {}_{\CM(R)} \mathcal{L}(R_T)}$ has the 
form $I \mathcal{L}(R_T)$, ${I\lhd \CM(R)}$,
\[
\PDer_{\CM(R)}(H, \mathcal{L}(R_T))\ =\ 
\{D|_H\mid D\in \Der_{\CM(R)}(\mathcal{L}(R_T))\}\ =\ 
\{\ad_x|_H\mid x\in \mathcal{L}(R_T)\}
\] 
(Remarks 3.22, 3.23), and ${{\PDer_{gr}}_{\CM(R)}(H, \mathcal{L}(R_T))_k=
\{\ad_x|_H\mid x\in \mathcal{L}(R_T)_k\}, k=0, 1}$ for all 
${H\in \mathcal{E}({}_{\CM(R)} \mathcal{L}(R_T))=
\mathcal{E}_{gr}({}_{\CM(R)} \mathcal{L}(R_T))=
\{I \mathcal{L}(R_T)\mid I\in \mathcal{E}(\CM(R))\}}$ 
(for $x=x_k+x_l$, $x_s=\mathcal{L}(R_T)_s$, ${s\in \{k, l\}=\{0, 1\}}$, 
${\ad_x|_H\in {\PDer_{gr}}_{\CM(R)}(H, \mathcal{L}(R_T))_k}$ implies 
${[H, x_l]=\{0\}}$, ${x_l=0}$). So, 
${\mathcal{L}(R_T)=Q_{m-gr}({}_{\CM(R)} \mathcal{L}(R_T))\cong 
\Der_{\CM(R)}(\mathcal{L}(R_T))=\ad(\mathcal{L}(R_T))}$,
${R_T=Q_m({}_{\CM(R)} R_T)}$.\linebreak  
As in the proof of Remark 3.25, one can deduce 
\begin{gather*}
Q_{m-gr}({}_{\CM(R)} \mathcal{L}(R_T))\ =\ 
Q_{m-gr}({}_{\CM(R)} \mathcal{L}(\nu_1 R_T))\oplus \ldots\oplus
Q_{m-gr}({}_{\CM(R)} \mathcal{L}(\nu_l R_T))\,,
\\ 
Q_m({}_{\CM(R)} R_T)\ =\ Q_m({}_{\CM(R)} \nu_1 R_T)\oplus \ldots\oplus 
Q_m({}_{\CM(R)} \nu_l R_T)\,.
\end{gather*}
\end{proof}

The following version of \cite[Theorem~3.8]{EMO} is entirely based on its 
proof and is given in an ungraded form sufficient for our purposes. 

\begin{prop}
If an algebra ${B=P(B)}$ is prime, a $\CM(B)$--algebra ${A=P(A)}$ with 1 is 
semiprime, then the $\CM(B)$--algebra ${C=A\mathbin{\otimes_{\CM(B)}} B}$ 
is semiprime, ${C=P(C)}$, $\CM(C)=\CM(A)\otimes \Id_B\cong \CM(A)$ and 
${C=O(C)}$ for ${A=O(A)}$, ${\dim_{\CM(B)} B< \infty}$.
\end{prop} 

\begin{proof}
Following \cite{EMO}, for any ${J\lhd C}$, ${a\in A}$ we select the ideals 
\[
U_{(J, a)}\ =\ \{b\in B\mid a\otimes b\in J\}\ \lhd\ B\,,\quad 
V_J\ =\ \{a\in A\mid U_{(J, a)}\ne \{0\}\}\ \lhd\ A\,.
\]
If ${J\ne \{0\}}$, then ${V_J^2\ne \{0\}}$ \cite[Lemma~3.7]{EMO}, 
${U_{(J, x)}\ne \{0\}}$ for all ${0\ne x\in V_J}$ and hence there 
are ${x, x'\in V_J}$, ${x x'\ne 0}$, 
${\{0\}\ne U_{(J, x)}\cap U_{(J, x')}\subseteq U_{(J, x+x')}}$, 
\[
\{0\}\ \ne\ (x\otimes (U_{(J, x)}\cap U_{(J, x')}))
(x'\otimes (U_{(J, x)}\cap U_{(J, x')}))\ =\ 
(x x')\otimes (U_{(J, x)}\cap U_{(J, x')})^2\ \subseteq\ J^2\,.
\]
So, $C$ is semiprime as a ring and as an algebra over any subdomain 
of the field $\CM(B)$ ((semi-)primeness of algebras is equivalent to their 
(semi-)primeness as rings).

Let ${\phi\in \Hom(J, C)_{M(C)'}}$, ${\{0\}\ne J\lhd C}$. For any ${x\in V_J}$ 
and ${0\ne y\in U_{(J, x)}}$ we can write  
${\phi(x\otimes y)=\sum\limits_{i=1}^k x_i\otimes y_i}$ for ${k\geq 1}$, 
${\{y_i\}\subseteq B}$ and linearly independent ${\{x_i\}\subseteq A}$ over 
$\CM(B)$. Due to the choice of $\{x_i\}$, 
${\Id_A\otimes M(B)'=M^C(1\otimes B)'}$ and  
\[
\phi((x\otimes y) (\Id_A\otimes \nu))\ =\ \phi(x\otimes (y \nu))\ =\ 
\sum\limits_{i=1}^k x_i\otimes (y_i \nu)\quad 
(\nu\in M(B)')\,,
\]
${\tau_i: y \mu\longmapsto y_i \mu, \mu\in M(R)'}$, is correctly defined, 
${\tau_i\in \Hom((y)_B, B)_{M(B)'}}$ and there exists 
$\alpha_i\in \CM(B)$, ${\tau_i=\alpha_i \Id_{(x)_B}}$, ${i=1, \ldots, k}$, 
${\phi(x\otimes y)=\sum\limits_{i=1}^k x_i\otimes (\alpha_i y)=z\otimes y}$ 
for ${z=\sum\limits_{i=1}^k \alpha_i x_i}$. 

If ${0\ne y'\in U_{(J, x)}}$, ${\phi(x\otimes y')=z'\otimes y'}$ 
for ${z'\in A}$, then ${(y)_B\cap (y')_B\ne \{0\}}$ and there are 
${\eta, \eta'\in M(B)'}$, ${y \eta=y' \eta'\ne 0}$, 
\[
z\otimes (y \eta)\ =\ \phi(x\otimes (y \eta))\ =\ 
\phi(x\otimes (y' \eta'))\ =\ z'\otimes (y' \eta')\,,\quad z\ =\ z'\,. 
\]
Therefore, ${\psi: x\longmapsto z, x\in V_J}$, is correctly defined. Since 
for any ${x, x'\in V_J}$ one can choose 
${0\ne d\in U_{(J, x)}\cap U_{(J, x')}\subseteq U_{(J, x+x')}}$, 
\[
\psi(x+x')\otimes d\ =\ \phi((x+x')\otimes d)\ =\ 
\phi(x\otimes d)+\phi(x'\otimes d)\ =\ (\psi x)\otimes d+(\psi x')\otimes d\,,
\]
and for any ${a\in A}$, ${y\in U_{(J, x)}}$, ${b\in B}$, ${y b\ne 0}$,  
${(x\otimes y)(a\otimes b)=(x a)\otimes (y b)\in J}$, ${y b\in U_{(J, x a)}}$, 
\[
\psi(x a)\otimes (y b)\ =\ \phi((x a)\otimes (y b))\ =\ 
\phi(x\otimes y) (a\otimes b)\ =\ ((\psi x)\otimes y)(a\otimes b)\ =\ 
((\psi x) a)\otimes (y b)\,,
\]
${\psi(x a)=(\psi x) a}$ and similarly ${\psi(a x)=a (\psi x)}$, 
${\psi\in \Hom(V_J, A)_{M(A)}}$, there is ${\overline{\psi}\in \CM(A)}$, 
$\psi=\overline{\psi}|_{V_J}$, 
${\phi|_{I(J)}=\overline{\psi}\otimes \Id_B|_{I(J)}}$, where 
${I(H)=\langle a\otimes b\mid a\in V_H,\ b\in U_{(H, a)}\rangle\lhd C}$, 
${I(H)\subseteq H}$ for all ${H\lhd C}$ and ${I(H)\ne \{0\}}$ for 
${H\ne \{0\}}$ \cite[Lemma~3.7]{EMO}. Since $C$ is semiprime, 
${\alpha H\lhd C}$ for all ${H\lhd C, H\subseteq J, 
\alpha\in \Hom(H, C)_{M(C)'}}$, ${\tau((J^2)_C)\subseteq J}$ for 
${\tau=\phi-\overline{\psi}\otimes \Id_B|_J}$, 
\begin{gather*}
I(\tau((J^2)_C))^2\ \subseteq\  
(\tau((J^2)_C)) I(\tau((J^2)_C))\ \subseteq\ \tau(J I(J))\ =\ \{0\}\,,
\\ 
\{0\}\ =\ I(\tau((J^2)_C))\ =\ \tau((J^2)_C)\ =\ 
J (\tau J)\ =\ (\tau J)^2\ =\ \tau J\,,\quad 
\phi\ =\ \overline{\psi}\otimes \Id_B|_J\,.
\end{gather*}
Thus, ${C=P(C)}$, ${\CM(C)=\CM(A)\otimes \Id_B\cong \CM(A)}$.

If ${A=O(A)}$, ${\dim_{\CM(B)} B=n< \infty}$, $\{e_i\}_{i=1}^n$ is a 
$\CM(B)$--basis of $B$, then ${C=O(C)}$, since
\[
{\sum\limits_{u\in U}}^{\perp} (\xi_u\otimes \Id_B) c_u\ =\ 
\sum\limits_{i=1}^n 
\biggl({\sum\limits_{u\in U}}^{\perp} \xi_u x_{u i}\biggr)\otimes e_i\ \in\ C
\]
for any dense orthogonal ${\{\xi_u\otimes \Id_B\}_{u\in U}\subseteq B(C)}$ 
(dense orthogonal ${\{\xi_u\}_{u\in U}\subseteq B(A)}$) and 
${\{c_u\}_{u\in U}\subseteq C}$, 
${c_u=\sum\limits_{i=1}^n x_{u i}\otimes e_i}$, ${x_{u i}\in \CM(B)}$.
\end{proof}

\begin{co}
If ${F=Q(F)}$ is a semiprime algebra over a field $\mathbb{F}$, ${R=P(R)}$ 
is a prime $\mathbb{F}$--algebra, ${\mathbb{F}=\CM(R)}$, then 
the $F$--algebra ${F R=F\mathbin{\otimes_{\mathbb{F}}} R}$ is semiprime, 
\[
F R\ =\ P({}_{\mathbb{F}} F R)\ =\ P({}_F F R)\,,\quad 
\CM(F R)\ =\ \CM({}_{\mathbb{F}} F R)\ =\ \CM({}_F F R)\ =\ 
F \Id_{F R}\ \cong\ F
\] 
and ${F R=O(F R)}$ for ${\dim_{\mathbb{F}} R< \infty}$. 
\end{co} 

In this case ${F=P(F)\cong \CM(F)}$ and ${F R=O(F R)}$ for  
${\dim_{\mathbb{F}} R=n< \infty}$ also follows 
from Remark 1.1. Applying \cite[Lemma~3.4]{Gol2}, we obtain

\begin{co}
If ${F=Q(F)}$ is a semiprime algebraic algebra over a field 
${\mathbb{F}=\overline{\mathbb{F}}}$, $R$ is a simple 
$\mathbb{F}$--algebra, ${\dim_{\mathbb{F}} R< \infty}$, 
then ${F R=P(F R)=O(F R)}$ is locally finite-dimensional 
and strongly algebraic over ${\mathbb{F}=\CM(R)}$.
\end{co}

Remarks 1.1, 1.2, 3.22 imply that for any $F_i$ and $R_i$, 
${i=1, \ldots, n}$ over ${\mathbb{F}=\overline{\mathbb{F}}}$ 
under the conditions of Corollary 3.29
${R=F_1 R_1\oplus \ldots\oplus F_n R_n=P(R)=O(R)}$ is locally 
finite-dimensional and strongly algebraic over $\mathbb{F}$, 
${\CM(R)\cong \CM(R_1)\oplus \ldots \oplus \CM(R_n)\cong 
F_1\oplus \ldots\oplus F_n}$ (cf. Corollary 3.20).

\end{document}